\documentclass[10pt,english,a4paper, intlimits]{amsart}

\usepackage[T1]{fontenc}
\usepackage[latin1]{inputenc}
\usepackage{amsmath}
\usepackage{amssymb}
\usepackage[colorlinks=true, pdfstartview=FitV, linkcolor=blue, citecolor=blue,
urlcolor=blue]{hyperref}

\usepackage{color}
\usepackage{calc}

\usepackage{ifthen}
\usepackage{amsfonts}
\usepackage{amsthm}

\allowdisplaybreaks

\numberwithin{equation}{section}
\newtheorem{theorem}{Theorem}[section]
\newtheorem{proposition}[theorem]{Proposition}
\newtheorem{lemma}[theorem]{Lemma}

\newtheorem{corollary}[theorem]{Corollary}

\theoremstyle{definition}
\newtheorem{definition}[theorem]{Definition}
\newtheorem{remark}[theorem]{Remark}

\newlength{\breite}
\newlength{\hbreite}
\newcommand{\lstackrel}[2]{
\mathrel{\mathop{#2}\limits^{\makebox[0pt]{$\scriptstyle #1$}}}
\settowidth{\breite}{$\scriptstyle #1$}
\settowidth{\hbreite}{$#2$}
\addtolength{\breite}{-1 \hbreite}
\addtolength{\breite}{-.5 \breite}
\makebox[\breite]{}
}

\newcommand{\g}{c}

\DeclareMathOperator{\hoel}{h\text{\"o}l}

\DeclareMathOperator{\im}{im}
\DeclareMathOperator{\spt}{spt}

\makeatother

\usepackage{babel}

\title[Gradient Flow of Knot Energies]{The Gradient Flow of O'Hara's Knot Energies}

\author{
Simon Blatt}
\address[Simon Blatt]{Paris Lodron Universit\"at Salzburg, Hellbrunner Strasse 34, 5020 Salzburg, Austria}
\email{simon.blatt@sbg.ac.at}

\begin{document}

\date{\today}

\begin{abstract}
Jun O'Hara invented a family of knot
energies $E^{j,p}$, $j,p \in (0, \infty)$, in \cite{OHara1992}.
We study the negative gradient flow of the sum of one of the energies $E^\alpha = E^{\alpha,1}$, $\alpha \in (2,3)$, and a positive multiple of the length. 

Showing that the gradients of these knot energies can be written as the normal part of 
a quasilinear operator, we derive short time existence results for these flows.
We then prove long time existence and convergence to critical points.
\end{abstract}

\maketitle

\section{Introduction}

Is there an optimal way to tie a knot in Euclidean space? And if so, how nice are these optimal shapes? Is there a natural way to transform a given knot into this optimal
shape?

To give a precise meaning to such questions a variety of energies for
immersions have been invented and studied during the last twenty five years which 
are subsumed under the term knot energies. 

In this article we deal with the third of the questions above. But first of all, let us gather some known answers to the first two questions. 

The first family of geometric knot energies goes back to O'Hara. In \cite{OHara1992}, O'Hara suggested for $j,p \in (0,\infty)$ the energy
\begin{equation*} 
E^{j,p}(\g) = \iint_{(\mathbb R / l\mathbb Z)^2} \bigg(\frac 1
{|\g(x)-\g(y)|^j} - \frac 1 {d_\g (x,y)^j} \bigg)^p
|\g'(x)|\cdot |\g'(y)| dx dy
\end{equation*}
of a regular closed curve $\g \in C^{0,1}(\mathbb R / l \mathbb Z, \mathbb R^n)$. Here, $d_\g (x,y)$ denotes the distance of the points $x$ and $y$
along the curve $\g$, i.e., the length of the shorter arc connecting these
two points.


O'Hara observed that these energies are knot energies if and only if $jp > 2$ \cite[Theorem~1.9]{OHara1992} in the sense that then both pull-tight of a knot and selfintersections are punished. Furthermore, he showed that minimizers of the energies exist within every knot class if $jp >2$. So: Yes, there is an optimal way to tie a knot -- actually even several ways to do so.

Abrams et al proved in \cite{Abrams2003} that for $p\geq 1$ and $jp -1 < 2p$ these energies are minimized by circles and that these energies are infinite for every closed regular curve if $jp-1 \geq 2p$.\footnote{In fact one can even show that for $jp -1 \geq 2p$ the energy is only finite for open curves that
are part of a straight line - in which case the energy is 0 (cf. \cite[Remark~1.3]{Blatt2015d}).}

There is a reason why for the rest of our questions we will only consider the case $p=1$:
For $p \not= 1$, we expect that the first
variation of $E^{j,p}$ leads to a \emph{degenerate} elliptic operator of fractional
order -- even after breaking the symmetry of the equation coming from the invariance under re-parameterizations. We will only consider the non-degenerate case $p=1$ and look at the one-parameter family
\begin{equation} \label{eq:DefinitionOfKnotEnergies}
 E^{\alpha}(\g) := E^{\alpha, 1} (\g) = \iint_{\mathbb R / L \mathbb Z }  \bigg(\frac 1
{|\g(x)-\g(y)|^\alpha} - \frac 1 {d_\g (x,y)^\alpha} \bigg)
|\g'(x) | \, |\g'(y)| dx dy.
\end{equation}
We leave the case $p \not= 1$ for a later study.

The most prominent member of this
family is $E^2$ which is also known as M\"obius energy due to the fact
that it is invariant under M\"obius transformations.  While for $\alpha \in (2,3)$ the Euler-Lagrange equation is a non-degenerate elliptic sub-critical operator it is a critical equation for the case of the M\"obius energy $E^2$.

In \cite{Freedman1994}, Freedman, He, and Wang showed that even $E^2$ can be minimized within every prime knot class. Whether or not the same is true for composite knot classes is an open problem, though there are some clues that this might not be the case in very such knot class.
Furthermore, they derived a formula  for the $L^2$-gradient of the  M\"obius
energy which was extended by Reiter in \cite{Reiter2012} to the 
energies $E^\alpha$ for $\alpha \in [2,3)$. They showed that the
first variation of these functionals can be given by
\begin{equation*}
  \nabla_h E^\alpha (\g) := \lim_{\varepsilon \rightarrow 0} \frac {E^{\alpha}(\g + \varepsilon h) - E^\alpha (\g)}{ h} = \int_{\mathbb R / l\mathbb Z}  \langle
 H^\alpha(\g) (x), h(x) \rangle  \cdot |\g'(x)|dx
\end{equation*}
where
\begin{multline} \label{eq:GradientOfEAlpha}
 H^\alpha(\g)(x) \\:= p.v. \int_{-\frac l2}^{\frac l2} P^\bot_{\g'(x)} \bigg\{2 \alpha \frac  {\g(x+w)
   -\g(x)}{|\g(x+w)-\g(x)|^{2+\alpha}}  
 -(\alpha-2) \frac
 {\kappa(x)}{d_\g(x+w,x)^\alpha}
 \\-2 \frac
 {\kappa(x)}{|\g(x+w)-\g(x)|^\alpha} \bigg\} |\g'(x+w)|
dw.
\end{multline}
Here, $P^\bot_{\g'}(u) = u - \langle u, \frac {\g'}{|\g'|}\rangle \frac {\g'}{|\g'|}$ denotes the orthogonal projection onto the normal part, and $ p.v \int_{-\frac l2}^{\frac l2}= \lim_{\varepsilon \downarrow 0} \int_{[-l/2,l/2] \setminus [-\varepsilon, \varepsilon]}$ denotes Cauchy's principal value.

In the case that $\g$ is parameterized by arc length this reduces to 
\begin{multline}
H^\alpha(\g)(x) \\:= p.v. \int_{-\frac l2}^{\frac l2} P_{\g'(x)}^\bot \bigg\{2 \alpha \frac  {\g(s+w)
   -\g(s)}{|\g(s+w)-\g(s)|^{2+\alpha}}  
 -(\alpha-2) \frac
 {\g''(s)}{|w|^\alpha}
 \\-2 \frac
 {\g''(s)}{|\g(s+w)-\g(s)|^\alpha} \bigg\} 
dw.
\end{multline}

Using the M\"obius invariance of $E^2$, Freedman, He, and Wang showed that local minimizers of the M\"obius energy are of class $C^{1,1}$ \cite{Freedman1994} -- and thus gave a first answer to the question about the niceness of the optimal shapes. Zheng-Xu He combined this with a sophisticated bootstrapping argument to find that minimizers of the M\"obius energy are of class $C^\infty$ \cite{He2000}.
Reiter could prove that critical points $\g$ of $E^\alpha$, $\alpha \in (2,3)$ with
$\kappa \in L^\alpha$ are smooth embedded curves \cite{Reiter2012}, a result we
extended to critical points of finite energy in \cite{Blatt2013} and to the M\"obius energy in \cite{Blatt2015a}.
 
Let us now turn to the last of the three questions we started this article with: Is there a natural way to transform a given knot into its optimal
shape? Since $E^\alpha$ is not scaling invariant for $\alpha \in (2,3)$,
the $L^2$-gradient flow of $E^\alpha$ alone cannot have a nice
asymptotic behavior. We want to avoid that the curve would get larger and larger in order to decrease the energy. Here, the length $L(c)$ of the curve $c$ will help us.

To transform a given knotted curve into
a nice representative, we will look at the $L^2$-gradient flow of $E=E^\alpha + \lambda L$ for $\alpha \in (2,3)$ and $\lambda >0$ instead.
This leads to the evolution equation
\begin{equation} \label{eq:EvolutionEquation}
 \partial_t \g = -H^\alpha (\g) + \lambda \kappa_\g.
\end{equation} 

%

We will see that the right hand side of this equation can be written as the normal part of a  quasilinear elliptic but non-local operator of order $\alpha+1 \in [3,4)$. 

The main result of this article is the following theorem. Roughly speaking,
it tells us that, given an initial regular embedded curve of class $C^\infty$, there exists a unique solution to the above evolution equations.
This solution is immortal and converges to a critical point. More precisely we have:

\begin{theorem} \label{thm:LongTimeExistenceResult}
Let $\alpha \in (2,3)$ and
$\g_0 \in C^{\infty}(\mathbb R / \mathbb Z, \mathbb R^n)$ be an injective regular curve. Then
there is a unique smooth solution 
$$\g \in C^\infty([0,\infty) \times \mathbb R / \mathbb Z)$$
to \eqref{eq:EvolutionEquation} 
with initial data $\g(0) = \g_0$ that after suitable re-parameterizations converges smoothly to a critical point of
$E^\alpha + \lambda L$.
\end{theorem}

Lin and Schwetlick showed similar results for the elastic energy plus some positive multiple of the M\"obius energy and the length \cite{Lin2010}. They succeeded in treating the term in the
$L^2$-gradient coming from the M\"obius energy as a lower order perturbation of the gradient of the elastic energy of curves. This allowed them to carry over the analysis due to Dziuk, Kuwert and Sch\"atzle of the latter flow \cite{Dziuk2002}. They proved long time existence for their flow and sub-convergence to a critical point up to re-parameterizations and translations.

The situation is quite different in the case we treat in this article. We have to understand
the gradient of O'Hara's energies in a much more detailed way and have to use sharper estimates than
in the work of Lin and Schwetlick. Furthermore, in contrast to Lin and Schwetlick we show that the complete flow, without going to a subsequence and applying suitable translations,  converges to a critical point of our energy. 
%

We want to conclude this introduction with an outline of the proof of Theorem~\ref{thm:LongTimeExistenceResult}. In Section~\ref{sec:ShortTimeExistence}, we prove short time existence results of these flow for initial data in little H\"older spaces and smooth dependence on the initial data. To do that, we show that the gradient of $E^\alpha$ is the normal part  of an abstract quasilinear differential operator of fractional order (cf. Theorem~\ref{thm:QuasilinearStructure}). Combining Banach's fixed-point theorem with a maximal regularity result for the linearized equation, we get existence for a short amount of time.
In order to keep this article as easily accessible as possible, we give a detailed prove of the necessary maximal regularity result.

The most important ingredient to the proof of long time existence in Section~\ref{sec:ProofOfLongTimeExistence} is a strengthening of the classification of curves of finite energy $E^\alpha$ in \cite{Blatt2012} using fractional Sobolev spaces. For $s \in (0,1)$, $p \in [1,\infty)$ and $k\in \mathbb N_0$ the space $W^{k+s,p}(\mathbb R / \mathbb Z, \mathbb R^n)$ consists of all functions $f\in W^{k,p}(\mathbb R / \mathbb R, \mathbb R^n )$ for which
$$
 |f^k|_{W^{s,p}}:= \left( \int_{\mathbb R / \mathbb Z} \int_{\mathbb R / \mathbb Z} \frac {|f(x) - f(y)|^p}{|x-y|^{1+sp}} dx dy \right)^{ 1/ p} 
$$
is finite. This space is equipped with the norm $\|f\|_{W^{k+s,p}} := \|f\|_{W^{k,p}} + |f^{(k)}|_{W^{s,p}}.$  For a thorough discussion of the subject of fractional Sobolev space we point the reader to the monograph of Triebel \cite{Triebel1983}. Chapter 7 of \cite{Adams2003} and the very nicely written and easy accessible introduction to the subject \cite{DiNezza2012}.

We know that a curve parameterized by arc length has finite energy $E^\alpha$ if and only if it is bi-Lipschitz and belongs to the space $W^{\frac {\alpha +1}2,2}$. In Theorem~\ref{thm:Coercivity} we show that even
\begin{equation*}
 |\gamma'|_{W^{\frac {\alpha -1} 2 ,2}} \leq C E^\alpha
\end{equation*}
and hence the $W^{\frac {\alpha +1} 2 ,2 }$-norm of the flow is uniformly bounded in time.

We then derive the evolution equation of
\begin{equation} \label{eq:Energies}
 \mathcal E^k = \int_{\mathbb R / \mathbb Z} |\partial_s^k\kappa|^2 |\g'(x)|dx.
\end{equation}
where $\kappa$ denotes the curvature of the curve $\g$ and 
$\partial_s := \frac {\partial_x}{|\g'(x)|}$ is the derivative with respect to 
the arc length parameter $s$. Note that in contrast to previous works like the work due
to Dziuk, Kuwert, and Sch\"atzle on elastic flow or the work due to Lin and
Schwetlick on the gradient flow of the elastic plus a positive multiple of the 
M\"obius energy, we do not consider the normal derivatives of the curvature but
the full derivatives with respect to arc length.

Using Gagliardo-Nirenberg-Sobolev inequalities for Besov spaces, which quite naturally appear 
during the calculation, together with commutator estimates we can then show that also the $\mathcal E^k$ are bounded uniformly in time. By standard arguments this will lead to
long time existence and smooth subconvergence to a critical point after suitable translations and re-parameterization of the curves.

To get the full statement, we study the behavior of solutions near such critical
points using a {\L}ojasievicz-Simon gradient estimate. This allows us to show that
flows starting close enough to a critical point and remaining above this critical point in the sense of the energy, exist for all time and converge to critical points. More precisely we have:

\begin{theorem}[Long time existence above critical points] \label{thm:LongTimeExistenceI}
Let $\g_{M}\in C^{\infty}(\mathbb{R}/\mathbb{Z}, \mathbb R^n)$ be a critical
point of the energy $E= E^\alpha + \lambda  L$, $\alpha \in [2,3)$, let $k\in \mathbb N$,
$\delta>0$, and $\beta >\alpha$. Then there is a constant $\varepsilon >0$ such that the
following is true:

Suppose that  $(\g_t)_{t\in [0,T)}$ is a maximal solution of the gradient flow of the energy
$E_\lambda$ for $\lambda >0$ with smooth initial data satisfying
\begin{equation*}
  \|\g_0 - \g_M\|_{C^{\beta}} \leq \varepsilon
\end{equation*}
and 
\begin{equation*}
 E(\g_t) \geq E(\g_M) 
\end{equation*}
whenever there is a diffeomorphism $\phi_t:\mathbb R / \mathbb Z
\rightarrow \mathbb R /\mathbb Z$ such that $\|\g_t \circ \phi_t -
\g_M\|_{C^\beta}\leq \delta$.
 Then the flow $(\g_t)_t$ exists for all times and converges,
after suitable re-parameterizations, smoothly to a critical point $\g_{\infty}$
of $E_\lambda$ satisfying \[
E(\g_{\infty})=E(\g_{M}).\]
\end{theorem}

This shows that in the situation of Theorem~\ref{thm:LongTimeExistence} the complete solution converges to a critical point of $E^\alpha + \lambda L$ -- even without applying any translations or re-parameterizations.

\section{Short Time Existence} \label{sec:ShortTimeExistence}

 This section is devoted to an almost self-contained proof of short time existence for equation \eqref{eq:EvolutionEquation}. We will show that for all $c \in C^\infty(\mathbb R/ \mathbb Z)$ a solution exists for some time.

 For any space $X \subset C^1(\mathbb R / \mathbb Z, \mathbb R^n)$ we will denote by $X_{ir}$ the (open) subspace consisting of all injective (embedded) and regular curves in $X$.

\begin{theorem} [Short time existence for smooth data] \label{thm:ShortTimeExistence}
 Let $c_0 \in C_{ir}^\infty(\mathbb R /  \mathbb Z)$. Then there exists some $T = T(c_0) >0$ and a unique solution
 $$
  c\in C^\infty([0,T) \times \mathbb R / \mathbb Z, \mathbb R^n )
 $$
 of \eqref{eq:EvolutionEquation}.
\end{theorem}

We can strengthen this result above.
In fact, we can reduce the regularity of the initial curve even below the level where our evolution equation makes sense. We still can prove existence and in a sense also uniqueness of a family of curves with normal velocity given by the gradient of $E^\alpha + \lambda L$. For this purpose we will work in little H\"older spaces $h^\beta$, $\beta \notin \mathbb N$,
which are the completion of $C^\infty$ with respect to the $C^\beta$-norm.

\begin{theorem} [Short time existence for non-smooth data] \label{thm:ShortTimeExistenceNonSmooth}
 For $\alpha \in (2,3)$
let $\g_0\in
h^{\beta}_{i,r}(\mathbb{R}/\mathbb{Z},\mathbb{R}^{n})$ for some $\beta >\alpha$,
$\beta \notin \mathbb N$.
Then there is a constant $T>0$ and a re-parameterization $\phi
\in C^{\beta} (\mathbb R / \mathbb Z, \mathbb R / \mathbb Z)$ such that there is a
solution $$\g\in C([0,T),h_{i,r}^{\beta}(\mathbb R / \mathbb Z,
\mathbb R ^n))\cap C^1((0,T),C^{\infty}(\mathbb{R}/\mathbb{Z},\mathbb{R}^{n}))$$
of the initial value problem
\begin{equation*}
\begin{cases}
\partial^{\bot}_{t}\g =-H(\g) + \lambda \kappa_\g& \forall t\in[0,T],\\
\g(0)=\g_{0} \circ \phi.\end{cases}
\end{equation*}
This solution is unique in the sense that for each other solution $$\tilde
\g \in C([0,\tilde T),h_{i,r}^{\beta}(\mathbb R / \mathbb Z,
\mathbb R ^n))\cap C^1((0,\tilde
T),C^{\infty}(\mathbb{R}/\mathbb{Z},\mathbb{R}^{n}))$$ and all
$t\in (0,\min(T,\tilde T)]$ there is a smooth diffeomorphism
$\phi_t \in C^\infty (\mathbb R / \mathbb Z, \mathbb R / \mathbb Z)$
such that
\begin{equation*}
  \g(t, \cdot ) = \tilde \g (t, \phi_t (\cdot)).
\end{equation*}
\end{theorem}

As in the special case of the M\"obius energy dealt with in \cite{Blatt2012b}, these results are based on
the fact that the functional $H^\alpha$ possesses a quasilinear
structure -- a statement that will be proven in the next subsection. After that
we build a short time existence theory for the linearization of these equations from scratch 
and prove a maximal regularity result for this equation.

To a get a solution of the evolution equation \eqref{eq:EvolutionEquation}, we first have to break the symmetry that comes from the invariance of the equation under re-parameterizations. We do this by writing the time dependent family of curves $c$ as a normal graph over some fixed smooth curve $c_0$. Applying  Banach's fixed-point theorem as done in \cite{Angenent1990}
for nonlinear and quasilinear semiflows,
we get short time existence for the evolution equation of the normal graphs and 
and continuous dependence of the solution on the initial data. A standard re-parameterization then gives Theorem~\ref{thm:ShortTimeExistence} while Theorem~\ref{thm:ShortTimeExistenceNonSmooth} is obtained via an approximation argument.

\subsection{Quasilinear Structure of the Gradient}

By exchanging every appearance of $|\g(u+w)-\g(u)|$ and $d_\g(u+w,u)$
by their first order Taylor expansion $|\g'(u)||w|$  and
$|\g'(u+w)|$ by $|\g'(u)|$ in the formula for 
$\tilde H^\alpha$, we are led to the conjecture that the leading order term of ${\mathcal H}^\alpha$ is the normal part of  $\frac {\alpha} {|\g'|^{\alpha +1 }} Q ^\alpha (\g)$ where
\begin{equation*}
  Q^\alpha (\g) := p.v.
  \int_{[-l,l]} \left(2 \frac {\g(u+w)- \g(u) - w \g'(u)  }{w^2} - \g''(u) \right)\frac{dw}{|w|^\alpha}.
\end{equation*}
 
This heuristic can be made rigorous using Taylor's expansions of the error terms and estimates for multilinear Hilbertransforms
(cf. Lemma~\ref{lem:HilberttransformGeneral}) leading to the next theorem. It will be essential later on that the remainder term is an analytic operator between certain function spaces -- which we denote by $C^\omega$.

\begin{theorem}[Quasilinear structure]\label{thm:QuasilinearStructure}

For $\alpha \in (2,3)$ there is a mapping 
\begin{equation*}
F^\alpha \in \bigcap_{\beta>0}
C^{\omega}(C_{i,r}^{\alpha+\beta}(\mathbb{R}/\mathbb{Z},\mathbb{R}^{n}),C^{\beta}(\mathbb{R}/\mathbb{Z},\mathbb{R}^{n}))
\end{equation*}
such that 
\begin{equation*}
H^\alpha\g=\frac{\alpha}{|\g'|^{\alpha+1}}P_{\g'}^{\bot}(Q^{\alpha}\g)+F^\alpha \g
\end{equation*} for all $\g\in
H_{i,r}^{\alpha+1}(\mathbb{R}/\mathbb{Z},\mathbb{R}^{n})
$.

\end{theorem}

\begin{proof}[Proof of \protect{Theorem~\ref{thm:QuasilinearStructure}}]

It is enough to show that there is a mapping 
\begin{equation*}
\tilde F^\alpha \in \bigcap_{\beta>0}
C^{\omega}(C_{i,r}^{\alpha+\beta}(\mathbb{R}/\mathbb{Z},\mathbb{R}^{n}),C^{\beta}(\mathbb{R}/\mathbb{Z},\mathbb{R}^{n}))
\end{equation*}
such that \[
\tilde H^\alpha\g=\frac{\alpha}{|\g'|^{\alpha+1}}Q^{\alpha}\g +
\tilde F^\alpha \g
\] for all $\g\in
H^{\alpha+1}(\mathbb{R}/\mathbb{Z},\mathbb{R}^{n})$, where 
\begin{multline} \label{eq:DefinitionOfTildeH}
  \tilde H^\alpha \g = \lim_{\varepsilon \searrow 0}  \int_{w \in
   I_{\epsilon,l}} \bigg\{2 \alpha \frac  {\g(x+w)
   -\g(x) - w \g'(s)}{|\g(x+w)-\g(x)|^{2+\alpha}}  
 -(\alpha-2) \frac
 {\g''(x)} {|\g'(x)|^2d_\g(x+w,x)^\alpha}
 \\
-2 \frac
 {\g''(x)}{|\g'|^2|\g(x+w)-\g(x)|^\alpha} \bigg\} |\g'(x+w)|
dw.
\end{multline}
The theorem then follows easily using $H^\alpha \g =
P_{\g'} ^\bot \tilde H^\alpha = \tilde H^\alpha - \langle H^\alpha, \frac {c'}{|c'|} \rangle \frac{c'} {|c'|}$.

We decompose
\begin{equation} \label{eq:DecompositionForQuasilinearStructure}
	\tilde H^{\alpha} \g 
= 
	\frac {\alpha} {|\g'|^{\alpha+1}}	Q^\alpha \g 
	+2 \alpha R^\alpha_1 \g  -  (\alpha - 2) R^\alpha_2 \g
	-2 R^\alpha_3 \g + \alpha R^\alpha_4 \g
\end{equation}
where
\begin{align*}
		(R^\alpha_1 \g) (x) 
	&:=  \int_{\mathbb R/ \mathbb Z} 
			\left( \g(x+w)-\g(x)-w \g'(x) \right) \\ & \quad \quad \quad \quad \quad
			 \bigg( 
				 \frac 1 {|\g(x+w) -\g(x)|^{\alpha+2}}  
				- \frac 1 {|\g'(x)|^{\alpha+2} |w|^{\alpha+2}
			}\bigg) 
		|\g'(x+w)|dw,
\\
		(R^\alpha_2 \g) (x) 
	&:=  \int_{\mathbb R/ \mathbb Z} 
			\frac{\g''(x)}{|\g'(x)|^2}
			\bigg(
				 \frac 1 {d_{\g}(x+w , x)^\alpha}  
				 - \frac 1 {|\g'(x)|^\alpha w^\alpha)}
			\bigg)
  		|\g'(x+w)|dw ,
\\
		(R^\alpha_3 \g) (x) 
	&:=  \int_{\mathbb R/ \mathbb Z} 
		 			\frac{\g''(x)}{|\g'(x)|^2}
			\bigg(
				 \frac 1 {|\g(x+w) -\g( x)|^\alpha}  
				 - \frac 1 {|\g'(x)|^\alpha w^\alpha)}
			\bigg)
  		|\g'(x+w)|dw ,
\\
		(R^\alpha_4 \g)(x)
	&:= \frac 1 {|c'|^{\alpha +2}} \int_{\mathbb R/ \mathbb Z}  			
			\bigg(
				 2 \frac {\g(x+w)-\g(x)-w\g'(x)}{w^2} 
				 -  \g''(x)
			\bigg)
  		\frac {|\g'(x+w)|-|\g'(x)|}{|w|^\alpha}dw.
\end{align*}

Using Taylor's expansion up to first order, we get
\begin{align*}
	d_\g(x+w, w) 
	&= w \g'(x) 
	+ w^2 \int_{0}^1 
		\left\langle 
			\frac {\g'(u+\tau w)}{|\g'(u+\tau w)|}, \g''(u+\tau w)
		\right\rangle d\tau 
	\\& = w \g'(x) (1+ w \tilde X_\g(x,w))
\end{align*}
where 
\begin{align*}
	\tilde X_\g(x,w) 
	:=
	\frac 1 {|\g'|} 
	\int_{0}^1 
		\left\langle 
			\frac {\g'(u+\tau w)}{|\g'(u+\tau w)|}, \g''(u+\tau w)
		\right\rangle
	d\tau,
\end{align*}
and
\begin{align*}
	|\g(x+w) - \g(x)|^2  
	= \left|w\g'(u)+w^2 \int_0^1 \g''(u+\tau w)d\tau \right |^2 
	= w^2 |\g'(x)| (1 +w X_\g(x,w) )
\end{align*}
where
\begin{align*}
	X_{\g}(x,w):= \frac 1 {|\g'|^2 } \left( \g'(u) \int_0^1 \g''(u+\tau w) d\tau +
	w \left(\int_0^1 \g''(u+\tau w) d\tau\right)^2\right).
\end{align*}
Together with the Taylor expansion
\begin{align*}
	|1+x|^{-\sigma} = 1 - \sigma x + \sigma (\sigma+1) x^2
	 \int_{0}^1 (1-\tau) |1+\tau x|^{-\sigma-2} d\tau
\end{align*}
for $\sigma >0$ and $x >-1$, this leads to 
\begin{align*}
		&\frac 1 {|\g(x+w)-\g(x)|^\sigma} 
		- \frac 1{|\g'(x)|^\sigma |w|^\sigma} 
	=	
		\frac 1 {|\g'(x)|^\sigma |w|^\sigma} 
		\left( (1+X_\g(x,w))^{-\frac \sigma 2}
			-1
		\right)
\\
	&=
		 \frac 1 {|\g'(x)|^\sigma |w|^\sigma}  
		\Bigg(
			-\frac \sigma 2 w X_\g(x,w) 
			\\  & \quad \quad \quad \quad \quad \quad \quad \quad + \frac \sigma 2 \left(\frac \sigma 2+1\right) w^2 X_\g (u,w)^2 
			\int_0^1 (1+ \tau w X_\g(x,w))^{-\frac \sigma 2-2}d\tau
		\Bigg)	
\end{align*}
and
\begin{multline*}
		\frac 1 {d(x+w,x)^\sigma} 
		- \frac 1{|\g'(x)|^\sigma |w|^\sigma}
\\
	=\frac 1 { |\g'(u)|^\sigma |w|^\sigma} 
		\left( - \sigma w\tilde X_\g(x,w) + \sigma (\sigma+1) w^2 \tilde X_{\g}(u,w)^2
		\int_0^1 (1+ \tau w \tilde X)^{-\sigma-2} d\tau\right).	
\end{multline*}

Furthermore, we will use the identities
\begin{align*}
 \g(x+w) - \g(x) - w \g'(x) = w^2 \int_{0}^1 (1-\tau) \g''(x+\tau w) d \tau 
\end{align*}
and
\begin{equation*}
 |\g'(x+w)| - |\g'(x)| = w\int_{0}^1  \left\langle \frac {\g'(x+\tau w)}{|\g'(x+\tau w|}, \g''(x+\tau w)\right\rangle d\tau.
\end{equation*}

Plugging these formulas into the expressions for the terms $R^\alpha_1, R^\alpha_2, R^\alpha_3, R^\alpha_4$
and factoring out, we see that they can be written as the sum of integrals as in the following Lemma~\ref{lem:estimateForQuasilinearStructure}. Hence, Lemma~\ref{lem:estimateForQuasilinearStructure} completes the proof.
\end{proof}

\begin{lemma} \label{lem:estimateForQuasilinearStructure}
For $\tilde l_1 \leq l_1$, $\tilde l_2 \leq l_2$, and $\tilde l_3 \leq l_3$ and a multilinear operator $M$ let
\begin{align*}
	I := \int_{[0,1]^{{\tilde l}_1+ {\tilde l}_2+{\tilde l}_3}}
	&M(\g''(x+\tau_1 w),\ldots,\g''(x+\tau_{l_1}w), 
		\g'(x+\tau_{l_1+1}), \ldots, \g'(x+\tau_{l_1 +l_2} w),
\\
		&\frac {\g'(x+\tau_{l_1+l_2+1}w)} {|\g'(x+\tau_{l_1+l_2+1}w)|}, 
		\ldots ,
		\frac {\g'(x+\tau_{l_1+l_2+l_3}w)} {|\g'(x+\tau_{l_1+l_2+l_3})|} \bigg)
\\ & \quad \quad \quad \quad \quad \quad
	d\tau_1 \cdots d\tau_{{\tilde l}_1} d\tau_{l_1+1} \cdots d\tau_{l_1+{\tilde l}_2} d\tau_{l_1+l_2 +1} \cdots d\tau_{l_1+l_2+{\tilde l}_3}, 
\end{align*}
i.e. we integrate over some of the $\tau_i$ but not over all.
Then for $\tilde \alpha \in (0,1)$ the functionals
\begin{align*}
	\tilde T_1(\g)(x) &:= 
	\int_{-1/2}^{1/2} 
		\frac {I}
		{w |w|^{\tilde \alpha}} dw, \\
	\tilde T_2(\g)(x) &:= 
	\int_{-1/2}^{1/2} 
		\frac {I (\int_0^1 (1+\tau w \tilde X(x,w))^{-\sigma})d\tau}
		{|w|^{\tilde \alpha}} 
	dw,
	\\
	T_1(\g)(x) &:= 
	\int_{-1/2}^{1/2} 
		\frac {I }
		{w |w|^{\tilde \alpha}} 
	dw, \\
	\intertext{and}
	T_2(\g)(x) &:= 
	\int_{-1/2}^{1/2} 
		\frac {I (\int_0^1 (1+\tau w X(x,w))^{-\sigma})d\tau}
		{|w|^{\tilde \alpha}} 
	dw
\end{align*}
are analytic from $C^{\beta + \tilde \alpha+2}$ to $C^{\beta}$ for all 
$\beta>0$.
\end{lemma}


\begin{proof}

The statement of the lemma for $\tilde T_1$ and $T_1$ follows immediately from the boundedness of the multilinear Hilberttransform in H\"olderspaces as stated in Remark \ref{rem:MultilineraHilbertTransform} combined with the Lemmata~\ref{lem:AnalyticityOfCompositions} and \ref{lem:AnalyticityOfParametricIntegrals}.

Using a similar argument, one deduces that $\g \rightarrow 1+ \tau X_\g$ is analytic, hence we get that for a given $\g_0$ there is a neighborhood
$U$ such that
\begin{equation*}
	\|D_\g^m (1+\tau X_\g(\cdot,w)\|_{L(C^{\beta+2},C^{\beta})} \leq C m!
\end{equation*}
for all $\g \in U$ where $C$ does not depend on $w$ and $\tau$. 
Using that $v \rightarrow |v|^{\sigma}$
is analytic away from $0$, we deduce that
\begin{equation*}
	\|D^m_\g (1+\tau X_\g(\cdot,w)^{\sigma}\|_{L(C^{\beta+2},C^{\beta})} \leq C m!
 \end{equation*} 
using Lemma~\ref{lem:AnalyticityOfCompositions}. Hence, the integrands in the definitions of $T_2$ and $
\tilde T_2$ satisfy the assumptions of Lemma~\ref{lem:AnalyticityOfParametricIntegrals}. Hence, $T_2$ and $\tilde T_2$ are even analytic operators from
$C^{\beta + 2}(\mathbb R / \mathbb Z, \mathbb R^n)$ to 
$C^{\beta} (\mathbb R / \mathbb Z, \mathbb R^n)$. 

\end{proof}

\subsection{Short Time Existence}

Using the quasilinear form of $\mathcal H^\alpha$, we derive short time existence
results for the gradient flow of  O'Hara's energies in this section.
For this task, we will work with families of curves that are normal graphs over a
fixed smooth curve $c_0$ and whose normal part belongs to a small neighborhood of $0$ in $h^\beta$, $\beta > \alpha$.


To describe these neighborhoods, note that there is a strictly positive,
lower semi-continuous
function $r:C_{i,r}^2(\mathbb R / \mathbb Z, \mathbb R^n) \rightarrow
(0,\infty)$ such that 
\begin{equation*} 
 \g + \{N\in C^1(\mathbb R / \mathbb Z, \mathbb
 R^n)^\bot_{\g_0} : \|N\|_{C^1} <r(\g)\}
\end{equation*}
only contains regular embedded curves for all $\g \in C_{i,r}^{1}(\mathbb R / \mathbb Z, \mathbb R^n)$ and
\begin{equation} \label{eq:RadiiComparabilityOfDerivatives}
 r(\g) \leq 1/2 \inf_{x\in \mathbb R / \mathbb Z} |\g'(x)|.
\end{equation}
Here, 
$C^\beta (\mathbb R / \mathbb Z, \mathbb R^n)_\g^\bot$ denotes the space
of all vector fields $N \in C^\beta (\mathbb R / \mathbb Z , \mathbb R^n)$ which are normal to $\g$, i.e. for which
$\langle\g'(u), N(u)\rangle=0$ for all $u \in \mathbb R / \mathbb
Z$. 
Letting 
$$
\mathcal V_{r,\beta} ( \g) := \{N\in h^{\beta}(\mathbb R / \mathbb Z, \mathbb
 R^n)^\bot_{\g} : \|N\|_{C^1} <r(\g)\}
$$ 
we have for all $\g \in h_{i,r}^{\beta}(\mathbb R/ \mathbb Z, \mathbb
R^n)$
\begin{equation} \label{eq:RadiiDontLeaveImbeddings}
 \g + \mathcal V_{r, \beta}(\g)\subset h^{\beta}_{i,r} (\mathbb R /\mathbb
 Z, \mathbb R^n).
\end{equation}

Let $N\in \mathcal V_{r,\beta}(\g)$. Equation~(\ref{eq:RadiiComparabilityOfDerivatives}) guarantees that
$P^{\bot}_{(\g + N)'(u)} $ is an isomorphism from the normal
space along $\g$  at $u$ to the normal space along $\g+N$. Otherwise there would be a $v\not=0$
in the normal space of $\g$ at $u$ such that 
\begin{equation*}
 0 =  P^{\bot}_{(\g + N)'(u)} (v)  = v - \left\langle v, \frac {(\g +
    N)'(u)}{|(\g + N)'(u)|} \right\rangle \frac {(\g +
    N)'(u)}{|(\g + N)'(u)|} 
\end{equation*}
which would contradict
\begin{align*}
\left| v - \left\langle v, \frac {(\g +
    N)'(u)}{|(\g + N)'(u)|} \right\rangle \frac {(\g +
    N)'(u)}{|(\g + N)'(u)|} \right|\geq |v| - \left|\left\langle v, \frac {
    N'(u)}{|(\g + N)'(u)|} \right\rangle| \right|\geq  |v|/2 >0.
\end{align*}
For $\g \in C^1 ((0,T),C_{i,r}^1(\mathbb R / \mathbb Z, \mathbb
R^n))$ we denote by
\begin{equation*}
  \partial_t^\bot \g = P_{\g'}^{\bot}( \partial_t \g)
\end{equation*}
the normal velocity of the family of curves. 

We prove the following strengthened version of the short time existence result mentioned at the beginning 
of Section~\ref{sec:ShortTimeExistence}.

\begin{theorem}[Short time existence for normal graphs] \label{thm:ShortTimeExistenceForGraphs}

Let $\g_0\in C^{\infty}(\mathbb{R}/\mathbb{Z}, \mathbb R^n)$ be an embedded regular
curve, $\alpha \in (2,3)$, and $\beta>\alpha$, $\beta \notin \mathbb N$.
Then for every $N_0 \in \mathcal V_{r,\beta} (\g_0)$ there is a constant $T=T(N_0)>0$ and a neighborhood
$U\subset \mathcal V_{r,\beta}$ of $N_0$
such that for every $\tilde N_0 \in U$ there is a unique solution
$N_{\tilde N_0}\in C([0,T),h^{\beta}(\mathbb{R}/\mathbb{Z})_{\g_0}^{\bot})\cap 
C^1((0,T),C^{\infty}(\mathbb{R}/\mathbb{Z})_{\g_0}^{\bot})$
of 
\begin{equation} \label{eq:EvolutionEquationNormalGraph}
\begin{cases}
  \left\{\partial^\bot_{t}(\g_0+N\right)=-H^\alpha(\g_0+N) + \lambda \kappa_{\g_0+N}& t\in[0,T],\\
  N(0)=\tilde N_{0}.
  \end{cases}
\end{equation}
Furthermore, the flow $(\tilde N_0,t) \mapsto N_{\tilde N_0}(t)$
is in $C^1((U \times (0,T)),C^{\infty}(\mathbb{R}/\mathbb{Z}))$.

\end{theorem}

The proof of Theorem~\ref{thm:ShortTimeExistenceForGraphs} consists of two steps. First we show that
(\ref{eq:EvolutionEquationNormalGraph}) can be transformed into an
abstract quasilinear system of parabolic type. The second step is to
establish short time existence results for the resulting equation. 

The second step can be done using general results about analytic
semigroups, regularity of pseudo-differential operators with rough
symbols \cite{Bourdaud1988}, and the short time existence results for quasilinear
equations in \cite{Angenent1990} or \cite{Amann1993}. Furthermore,
we need continuous dependence of the solution on the data and
smoothing effects in order to derive the long time existence results in Section~\ref{sec:ExponentialConvergence}.
  
For the convenience of the reader, we go a different way here and present a self-contained
proof of the short time existence that only relies on a
characterization of the little H\"older spaces as trace spaces. In
Subsection~\ref{subsec:APrioriEstimates}, we deduce a maximal
regularity result for solutions of linear equations of type $\partial_t u +a(t) Q u +b(t)u =f$ in little
H\"older spaces using heat kernel estimates. Following ideas from
\cite{Angenent1990}, we then prove short time existence and
differentiable dependence on the data for the quasilinear equation.

%

\subsubsection{The Linear Equation} \label{subsec:APrioriEstimates}

We will derive a priori estimates and existence results
for linear equations of the type
\begin{equation*}
 \begin{cases}  \partial_t u + a Q^s u + bu = f  \text { in } \mathbb R /
   \mathbb Z \times (0,T)\\
   u(0)=u_0
\end{cases}
\end{equation*} 
using little H\"older spaces, 
where $b(t) \in  L(
h^\beta(\mathbb R / \mathbb Z,\mathbb R^n), C^\beta(\mathbb R
/\mathbb Z , \mathbb R^n))$ and $a(t) \in h^\beta(\mathbb R / \mathbb Z,(0,\infty))$.
%

For $\theta \in (0,1)$, $\beta>0$,
and $T>0$ we will consider solutions that lie in the space
\begin{multline*}
X^{\theta,\beta}_T:=\bigg\{g \in C((0,T),h^{\beta+s}(\mathbb R / \mathbb
Z, \mathbb R^n)) \cap C^1((0,T), h^{\beta}(\mathbb R / \mathbb Z,
\mathbb R ^n): 
\\
 \sup_{t\in (0,T)} t^{1-\theta} \left( \|\partial_t
    g(t)\|_{C^{\beta}} + \|g(t)\|_{C^{s +\beta}}\right) < \infty \bigg\}
\end{multline*}
and equip this space with the norm
\begin{equation*}
\|g\|_{X^{\theta,\beta}_{T}} :=\sup_{t\in (0,T)} t^{1-\theta} \left(\|\partial_t
    g(t)\|_{C^{\beta}} + \|g(t)\|_{C^{s+\beta}}\right).
\end{equation*}
The right hand side $f$ of our equation should then belong to the space
\begin{equation*}
Y^{\theta,\beta}_T:=\left\{g \in C((0,T),h^{\beta}(\mathbb R / \mathbb
Z, \mathbb R^n)): \sup_{t\in (0,T)} t^{1-\theta }
\|g(t)\|_{C^{\beta}}<\infty \right\}
\end{equation*}
equipped with the norm
\begin{equation*}
\|g\|_{Y^{\theta,\beta}_{T}} :=\sup_{t\in (0,T)} t^{1-\theta} \|g(t)\|_{C^{\beta}}. 
\end{equation*}

From the trace method
in the theory of interpolation spaces (cf. \cite[Section 1.2.2]{Lunardi1995}), 
the following relation of the space $X_{\theta,\beta}$ to the little H\"older space $h^{\beta+s\theta}$ is well known
if $\beta + s \theta$ is not an integer:

If $u\in X^{\theta, \beta}_{T}$ then $u(t)$
converges in $h^{\beta + s \theta}$ to a function $u(0)$
as $t \searrow 0$ with
\begin{equation} \label{eq:TraceSpace}
  \|u(0)\|_{C^{\beta+s \theta}} \leq C \|u\|_{X^T_{\theta, \beta}}.
\end{equation}
On the other hand, for every $u_0 \in h^{\beta + s \theta}$
there is a $u \in X^{\theta, \beta}_{T}$ such that  $u(t)$
converges in $h^{\beta + s \theta}$ to  $u(0)$ for $t \rightarrow 0$  and
\begin{equation} \label{eq:TraceSpace2}
  \|u\|_{X^T_{\theta, \beta}} \leq \|u_0\|_{C^{\beta+s \theta}}.
\end{equation} 
Given $u_0$ we will see that the 
solution of the initial value problem
\begin{equation*}
 \begin{cases}
  \partial_t u + (-\Delta)^{s/2} u = 0 \quad &\text{ on } \mathbb R^n \\
  u = u_0 &\text{ at } t=0.
 \end{cases}
\end{equation*}
satisfies \eqref{eq:TraceSpace2}.
The well-known embedding $X_T^{\theta, \beta} \subset C^{\theta}((0,T),
C^\beta(\mathbb R / \mathbb Z, \mathbb R ^n))$ will also be essential in the proof.

The aim of this subsection is to prove the following theorem about the
solvability of our linear equation:

\begin{theorem} \label{lem:JIsAnIsomorphism}
Let $T>0$, $\beta>0$,  $\theta\in (0,1)$ with $\beta+s\theta
\notin \mathbb N$, and
$$a\in C^1([0,T],h^\beta (\mathbb R /\mathbb
  Z, [1/\Lambda,\infty))) , \quad  b\in C^0((0,T), L(h^\beta(\mathbb R /
  \mathbb Z, \mathbb R^n),h^\beta(\mathbb R / \mathbb Z, \mathbb R^n)))$$
  with
  \[ \|a\|_{C^1([0,T],C^\beta)}+\sup_{t \in (0,T)}t^{1-\theta} \|b(t)\|_{L({h^\beta, h^\beta})}< \infty \]
Then the mapping $J:u \mapsto (u(0), \partial_t u + a Q^{s-1}u + bu)$
defines an isomorphism between $X^{\theta, \beta}_T$ and
$h^{\beta+\theta s} (\mathbb R/ \mathbb Z, \mathbb R^n)\times Y^{\theta, \beta}_T$.
\end{theorem}

This will be enough to prove short time existence of a solution for some quasilinear equations later on
using Banach's fixed-point theorem.

Equation~\eqref{eq:TraceSpace} already guarantees that 
$J$ is a bounded linear operator. So we only have to prove that it is onto
for which we  will use some a priori results (also called
maximal regularity results in this context)
together with the method of continuity. 

To derive these estimates, we will freeze the coefficients and use
a priori estimates for $\partial_t u+ \lambda (-\Delta)^{s/2} u = f$ on
$\mathbb R$ where $\lambda >0$ is a constant. He observed 
in \cite{He1999}, that the fractional Laplacian can be expressed by
\begin{equation} \label{eq:DefinitionOfDelta32}
c_s  (-\Delta)^{s/2} u =  \; p.v.\! \int_{-\infty}^\infty
  \left(2\frac{u(x+w)-u(x)-wu'(x)}{|w|^2}-u''(x) \right) \frac{dw}{|w|^{s-1}}
\end{equation}
for a $c_{s}>0$ for all $u \in H^{s}(\mathbb R, \mathbb
 R^n)$. We will use this identity together with a localization argument
 to get from $(-\Delta)^{s/2}$ living on 
 $\mathbb R$ back to 
 our operator $Q^s$ which lives on the circle $\mathbb R / \mathbb Z$.

Note that the fractional Laplacian $(-\Delta)^{s/2}$ on $\mathbb R$ 
is bounded from $
C_0^{s + \beta}(\mathbb R, \mathbb R^n)$ to $C^{\beta}(\mathbb R , \mathbb R^n)$ for all 
$\beta > 0$, $\beta \notin \mathbb N$.

Let us consider the heat kernel of the equation $\partial_tu +
(-\Delta)^{s/2} u  = 0$ which is given by
\begin{equation} \label{eq:DefinitionOfHeatKernel}
  G_t (x):= \frac 1 {2\pi} \int_{\mathbb R} e^{2\pi ikx} e ^{-t|2 \pi k|^{s}} dk.
\end{equation}
for all $t>0$ and $x \in \mathbb R$. 

Since $k\mapsto e^{-t|2 \pi k|^s}$ is a Schwartz function, its
inverse Fourier transform $G_t$  is a Schwartz function as
well. Furthermore, one easily sees using the Fourier transformation that
\begin{equation} \label{eq:HeatKernelSolveEvolutionEquation}
  \partial_t G_t + (-\Delta)^{s/2} G_t =0 \quad \text{on }\mathbb R \quad \forall t>0.
\end{equation}
The most important property for us is  the scaling \begin{equation} \label{eq:ScalingOfHeatKernel}
G_t (x)=
t^{-1/s} G_1 (t^{-1/s}x),
\end{equation}
 from which we deduce 
\begin{equation} \label{eq:ScalingOfDerivativesOfTheHeatKernel}\partial^k_x G_t (x) = t^{-(1+k)/s} (\partial^k_x G_1)
(t^{-1/s}x)
\end{equation} and hence
\begin{equation} \label{eq:EstimateL1NormHeatKernel}
\| \partial_x^k G_t\|_{L^1 (\mathbb R)} \leq C_k t^{-k/s}\|\partial_x^kG_1\|_{L^1 (\mathbb
  R)} \leq C_k t^{-k/s}.
\end{equation}

Combining these relations with standard interpolation techniques, 
we get the following estimates for the heat kernel

\begin{lemma}[heat kernel estimates]  \label{lem:HeatKernelEstimates}
For all $0\leq \beta_1 \leq \beta_2$, and $T >0$ there is a constant $C=C(\beta_1,
  \beta_2, T)< \infty $ such that 
\begin{equation*}
 \|G_t \ast f\|_{C^{\beta_2}} \leq C t^{-(\beta_2-\beta_1)/s}
\|f\|_{C^{\beta_1}} \quad \forall f \in
 C^{\beta_1}(\mathbb R , \mathbb R^n), t\in (0,T].
\end{equation*}
\end{lemma}

\begin{proof}
 Let $k_i \in \mathbb N_0$ and $\tilde \beta_i \in [0,1)$ be
  such that $\beta_i = k_i +\tilde \beta_i$ for $i=1,2$.

For $l>m$ and $\beta\in (0,1)$, we deduce from $\int_{\mathbb R
} \partial_x^{l-m}G_t(y)dy=0$ and the fact that $G_1$ is a Schwartz function 
\begin{align*} 
 |\partial_x^{l} (G_t \ast f)(x)| &=  \left| \int_{\mathbb R} \partial^{l-m}_x G_t
    (y) (\partial_x^{m}f(x-y) - \partial_x^{m}f(x)) dy \right| \\
  &\lstackrel{(\ref{eq:ScalingOfDerivativesOfTheHeatKernel})}{\leq} \int_{\mathbb R} t^{-(1+l-m)/s}|(\partial^{l-m}_x G_1)(y/t^{1/3})
    | \cdot |(\partial_x^{m}f(x-y) - \partial_x^{m}f(x)) |dy \\
 & \lstackrel{z= y /t^{1/s}}{\leq}  \int_{\mathbb R} t^{-(l-m)/s}|(\partial^{l-m}_x G_1)(z)
  | \cdot|\partial_x^{m}f(x-t^{1/s}z) - \partial_x^{m}f(x)) |dz \\
    & \leq  t^{-(l-(m+\beta))/s} \hoel_{\beta} (\partial_x^m f)
    \int_{\mathbb R} |(\partial^{l-m}_x G_1)(z) ||z|^\beta dz \\
 &\leq
    C(l,m,\beta)  t^{-(l-(m+\beta))/s} \hoel_{\beta} (\partial_x^m f).
\end{align*}
For all $l\geq m$, $\beta \in (0,1)$ we have
\begin{equation*}
  \hoel_{\beta} (\partial_x^{l} (G_t \ast f)) \leq C(l-m)t^{l-m}
  \hoel_\beta (\partial_x^{m} f)
\end{equation*}
as for all $x_1,x_2 \in \mathbb R$
\begin{align*}
|\partial_x^{l} (G_t \ast f)(x_1) &- \partial_x^{l} (G_t \ast f)(x_2)|\\  &=  \left| \int_{\mathbb R} \partial^{l-m}_x G_t
    (y) (\partial_x^{m}f(x_1-y) - \partial_x^{m}f(x_2-y)) dy \right|
  \\
  &\leq \|\partial_x^{l-m}G_t\|_{L^1} \hoel_{\beta} (\partial_x^m f) |x_1 -x_2|^\beta \\
    & \lstackrel{~(\ref{eq:EstimateL1NormHeatKernel})}\leq C(l-m)
    t^{-((l-m))/s} \hoel_{\beta} (\partial_x^m f) |x_1 -x_2|^\beta .
\end{align*}
In a similar way  we obtain for all $l\geq m$
\begin{equation*}
  \|\partial_x^{l} (G_t \ast f)\|_{L^\infty} \leq C(l-m)t^{-(l-m)/s} \|\partial_x^{m}f\|_{L^\infty}. 
\end{equation*}
Combining these three estimates, we get
\begin{align} \label{eq:IntermediateEstimatesHeatKernel1}
 \|G_t \ast f\|_{C^{k_2+\tilde \beta_1}} &\leq C t^{-(k_2 -k_1)/s} \|
 f\|_{C^{k_1+\tilde \beta_1}}, \\
 \|G_t \ast f\|_{C^{k_2+1}} &\leq C t^{-((k_2+1) -(k_1+\tilde \beta_1)/s} \|
 f\|_{C^{k_1+\tilde \beta_1}}, \label{eq:IntermediateEstimatesHeatKernel2}
\end{align}
and if $k_2>k_1$
\begin{align} 
\|G_t \ast f \|_{C^{k_2}} \leq C  t^{-(k_2-(k_1+\tilde \beta_1))/s} \|
 f\|_{C^{k_1+\tilde \beta_1}} \label{eq:IntermediateEstimatesHeatKernel3}
\end{align}
Furthermore, we will use that for $0\leq \alpha \leq \beta \leq \gamma
\leq 1$, $\alpha \not= \gamma$, and
$f\in C^\gamma$ we have the interpolation inequality
\begin{equation} \label{eq:InterpolationHoelderSpaces}
 \| f \|_{C^\beta}\leq 2 \|f\|_{C^\gamma}^{\frac {\beta -
     \alpha}{\gamma -\alpha}}  \| f\|_{C^\alpha}^{\frac {\gamma-\beta}{\gamma -\alpha}}.
\end{equation}
For $\tilde \beta_2 \geq \tilde \beta_1$ we get
\begin{align*}
  \|G_t \ast f\|_{C^{k_2+\tilde \beta_2}} &\leq C \left( \|\partial_x ^{k_2}
    (G_t \ast f)\|_{C^{\tilde \beta_2}} + \|G_t \ast f\|_{L^\infty} \right) \\
&\lstackrel{(\ref{eq:InterpolationHoelderSpaces})}{\leq} C ( \|\partial_x ^{k_2}
 (G_t \ast f)\|_{C^{\tilde
     \beta_1}}^{\frac{1-\tilde\beta_2}{1-\tilde \beta_1}}\|\partial_x ^{k_2+1}
  (G_t \ast f)\|_{C^{0}}^{\frac{\tilde \beta_2 -\tilde
      \beta_1}{1-\tilde \beta_1}} +
  \|f\|_{L^\infty}) 
\\
&\lstackrel{(\ref{eq:IntermediateEstimatesHeatKernel1})\&
  (\ref{eq:IntermediateEstimatesHeatKernel2})}{\leq} C (t^{-(\beta_2 - \beta_1)/s} +1) \|f\|_{C^{\beta_1}}.
\end{align*}
For $\tilde \beta_1> \tilde \beta_2$ and hence $k_1<k_2$, we obtain
\begin{align*}
  \|G_t \ast f\|_{C^{k_2+\tilde \beta_2}} &\leq C \left( \|\partial_x ^{k_2}
    (G_t \ast f)\|_{C^{\tilde \beta_2}} + \|G_t \ast f)|_{L^\infty} \right) \\
&\lstackrel{(\ref{eq:InterpolationHoelderSpaces})} {\leq} C (
\|\partial_x ^{k_2} (G_t \ast
f)\|_{C^{\tilde \beta_1}}^{\frac{\tilde \beta_2}{\tilde \beta_1}}\|\partial_x ^{k_2}
  (G_t \ast f)\|_{C^0}^{\frac{\tilde \beta_1 -\tilde
      \beta_2}{\tilde \beta_1}} +
  \|G_t \ast f\|_{L^\infty}) \\
&\lstackrel{(\ref{eq:IntermediateEstimatesHeatKernel1})\&
  (\ref{eq:IntermediateEstimatesHeatKernel3})} {\leq} C (t^{-(\beta_2 -\beta_1)/s} +1) \|f\|_{C^{\beta_1}}.
\end{align*}
\end{proof}

To derive a representation formula for the solution of $\partial_t u +
(-\Delta)^{s/2}u =f$, we need the following simple fact
\begin{lemma} \label{lem:ConvergenceOfGt}For all $t>0$ we have
\begin{equation*}
 \int_{\mathbb R} G_t(x)dx =1.
\end{equation*}
Furthermore, for all $f\in h^\beta(\mathbb R,\mathbb R^n)$, $\beta\notin \mathbb
N$, there holds
\begin{equation*}
  G_t \ast f \xrightarrow{t\downarrow 0} f \quad 
\text{ in }h^\beta(\mathbb R, \mathbb R^n).
\end{equation*}
\end{lemma}

\begin{proof}
For $g \in L^2(\mathbb R)$ let $\hat g$ denote
the Fourier transform of $g$. 

For $t>0$ and $f\in L^2(\mathbb R)$ we obtain from  Lebesgue's
theorem of dominated convergence
\begin{equation*}
 (G_t \ast f) ^{\wedge}= e^{-t|2\pi \cdot |^s} \hat f \xrightarrow{t \downarrow
   0} \hat f
 \quad \text{ in } L^2.
\end{equation*}
Hence, Plancherel's formula shows
\begin{equation*}
  G_t \ast f \rightarrow f \quad \text{ in } L^2.
\end{equation*} 
Setting $f= \chi_{[-1,1]}$ and observing
\begin{equation*}
  \lim_{t\searrow 0}(G_t \ast f) (x) =\lim_{t\searrow 0}
  \int_{[t^{-1/s}(x-1),t^{1/s} (x+1)]} G_1 dy =\int_{\mathbb R} G_1 dy, \quad
  \forall x\in (-1,1),
\end{equation*}
we deduce that
\begin{equation*}
\int_{\mathbb R} G_1 dy =1.
\end{equation*}

To prove the second part, let $f\in h^\beta(\mathbb R, \mathbb R^n)$.
From convergence results for smoothing kernels we get for all $\tilde
f \in C^\infty(\mathbb R)$
\begin{align*}
 \limsup_{t\downarrow 0}\| f -G_t \ast f\|_{C^{\beta}} &\leq
 \limsup_{t\downarrow 0}\|(f-\tilde f) -G_t\ast (f-\tilde f)\|_{C^{\beta}} +
 \|\tilde f - G_t \ast \tilde f \|_{C^\beta} 
\\ &=  \limsup_{t\downarrow
   0}\|(f-\tilde f) -G_t\ast (f-\tilde f)\|_{C^\beta}
\\ &\lstackrel{\text{Lemma~\ref{lem:HeatKernelEstimates}}}{\leq} C  \|(f-\tilde f)\|_{C^\beta}.
\end{align*}
Since $h^\beta(\mathbb R,  \mathbb R^n)$ is the closure of
$C^\infty(\mathbb R , \mathbb R ^n)$ under $\|\cdot
\|_{C^\beta}$, this proves the statement.
\end{proof}

Linking the heat kernel $G_t$ to the evolution equation
$\partial_t + \lambda (-\Delta)^{s/2} =f $ for constant $\lambda>0$ we derive
the following a priori estimates

\begin{lemma}[Maximal regularity for constant coefficients] \label{lem:MaximalRegularityForConstantCoefficients}
For all $\beta>0 $, $\theta \in (0,1)$ with $ \beta + s \theta
\notin \mathbb N$, and $0<T<\infty $, $\lambda>0$ there is  a  constant
$C=C(\beta, \theta,T,\lambda)$ such that the following holds:

Let $u\in C^1((0,T),
h^\beta(\mathbb R , \mathbb R^n))\cap C^0((0,T),h^{s+\beta}(\mathbb R )) \cap C^0([0,T), h^{\beta + s\theta}(\mathbb R))$ such that $u(t)$ has compact support
for all $t\in (0,T)$. Then
\begin{multline} 
 \sup_{t\in (0,T]}t^{1-\theta}
 \left(\|\partial_tu\|_{C^{\beta}}
   +\|u\|_{C^{s+\beta}}\right)  \\ \leq 
C \left(\sup_{t\in(0,T]}t^{1-\theta}\|\partial_t u + \lambda
(-\Delta)^{s/2})u\|_{C^{\beta}} + \|u(0)\|_{h^{\beta+s\theta}} \right)
\end{multline} 
\end{lemma}

\begin{proof}

Setting $\tilde u(t,x):=u(t, \lambda^{1/s}x)$ and observing that 
$\partial_t \tilde u (x,t)+ (-\Delta)^{s/2} \tilde u(x,t) = \partial_t
u (t, \lambda^{1/s}x) + \lambda (-\Delta)^{s/2} u(t,\lambda^{1/s} x)$,
one sees that it is enough to prove the lemma for
$\lambda =1$ 

To this end, we first show that Duhamel's formula 
\begin{equation} \label{eq:ReprensentationFormula}
  u(t,\cdot)= \int_0^t G_{t-\tau} \ast f(\tau,\cdot) d\tau + G_t \ast u(0)
\end{equation}
holds, where $f=\partial_t u + (-\Delta)^{s/2}u$. For fixed $t>0$ we decompose the integral
in equation~(\ref{eq:ReprensentationFormula})
into 
\begin{equation*}
I_\varepsilon := \int_{t-\varepsilon}^{t} G_{t-\tau} \ast f(\tau,\cdot) dx d\tau 
\end{equation*} and
\begin{equation*}
J_\varepsilon := \int_{0}^{t-\varepsilon} G_{t-\tau} \ast f(\tau,\cdot) dx d\tau
\end{equation*}
and see that 
\begin{equation*}
\|I_\varepsilon\|_{L^\infty} \stackrel{Lemma~\ref{lem:HeatKernelEstimates}}{\leq }C\varepsilon
 \sup_{\tau\in (t-\varepsilon,t)}\|f(\tau,\cdot)\|_{L^\infty}\xrightarrow{\varepsilon \downarrow 0} 0.
\end{equation*}
As our assumptions imply that $u(t) \in H^s(\mathbb R, \mathbb R^n)$,
we get, comparing the Fourier transform of both sides ,
\begin{equation} \label{eq:Selfadjointness}
\left(G_{t-\tau}  \ast ((-\Delta)^{s/2} u(\tau,\cdot)) \right) (x)=
\left(((-\Delta)^{s/2}G_{t-\tau}) \ast u(\tau,\cdot) \right)(x).
\end{equation}
Partial integration in time and equation~(\ref{eq:Selfadjointness}) yields
\begin{align*}
J_\varepsilon & =  \int_{0}^{t-\varepsilon} G_{t-\tau} \ast \partial_t
u(\tau,\cdot) d\tau + \int _0^{t-\varepsilon} G_{t-\tau} \ast (-\Delta)^{s/2}u(\tau,\cdot) 
d\tau \\
& = G_{\varepsilon}\ast u (t-\varepsilon,\cdot )-G_{t} \ast u(0,\cdot ) + \int_0^{t-\varepsilon} (\partial_\tau ( G_{t-\tau}) + (-\Delta)^{s/2}
G_{t-\tau})  \ast u d\tau \\
& \lstackrel{~(\ref{eq:HeatKernelSolveEvolutionEquation})}{=}
G_{\varepsilon}\ast u (t-\varepsilon,\cdot )-G_{t} \ast u(0,\cdot ) \xrightarrow{Lemma~\ref{lem:ConvergenceOfGt}} u(t, \cdot ) - G_{t} \ast u(0, \cdot).
\end{align*}
in $C^\beta$ as $\varepsilon \searrow 0$. This proves Equation~(\ref{eq:ReprensentationFormula}).  

From Lemma~\ref{lem:HeatKernelEstimates} we get
\begin{equation} \label{eq:EstimateHomogeneousTerm}
 \|G_t \ast u_0\|_{C^{s+\beta}} \leq C t^{\theta-1} \|u_0\|_{C^{\beta
     + s \theta}} 
\end{equation}
We decompose $v(t):= \int_0^t G_{t-\tau} \ast f(\tau,\cdot)d\tau = v_1 (t) + v_2(t)$ where
\begin{equation*}
  v_1(t)= G_{t/2} \ast v(t/2), \quad \quad v_2(t)= \int_{\tau=t/2}^t
  G_{t-\tau}\ast f(\tau, \cdot)d\tau.
\end{equation*}
Then the definition of $\|\cdot\|_{Y^{\beta,\theta}_T}$ and the
estimates for the heat kernel in Lemma~\ref{lem:HeatKernelEstimates}
lead to 
\begin{equation} \label{eq:EstimateV1}
  \|v_1(t)\|_{C^{s+\beta}} \leq C (t/2)^{-1}
  \|f\|_{Y^{\beta,\theta}_T} \int_{0}^{t/2}
  \tau^{\theta-1} d\tau \leq C (t/2)^{\theta-1} \|f\|_{Y^{\beta,\theta}_T}.
\end{equation}
For $\xi >0$ and $\eta \in (0,1)$ we get 
\begin{align*}
  \|\xi^{1-\eta}(G_\xi \ast v_2(t))\|_{C^{2s+\beta-s\eta}}
    &=\Big\|\xi^{1-\eta}\int_{t/2}^t(G_{t-\tau+\xi} \ast f )d\tau \Big\|_{C^{2s+\beta-s\eta}}
    \\
    &\lstackrel{Lemma~\ref{lem:HeatKernelEstimates}}{\leq} C \xi^{1-\eta}\int_{t/2}^t (t-\tau+\xi)^{-2+\eta}\tau^{\theta-1} d\tau  \|f\|_{Y^{\beta,\theta}_T}
    \\
    & \leq C(t/2)^{\theta-1} \|f\|_{Y^{\beta,\theta}_T}
\end{align*}
and
\begin{align*}
  \|\xi^{1-\eta} \frac d {d\xi}(G_\xi \ast v_2)\|_{C^{s+\beta-s\eta}}
  &=\Big\|\xi^{1-\eta}\int_{t/2}^t(\partial_tG_{t-\tau+\xi} \ast f\Big\|_{C^{s+\beta-s\eta}}
    \\
    &\leq C \xi^{1-\eta}\int_{t/2}^t (t-\tau+\xi)^{-2+\eta}\tau^{\theta-1} d\tau  \|f\|_{Y^{\beta,\theta}_T}
    \\
    & \leq C(t/2)^{\theta-1} \|f\|_{Y^{\beta,\theta}_T}
\end{align*}
as
\begin{equation*}
  \xi^{1-\eta}\int_{\frac t2}^t (t-\tau+\xi)^{-2+\eta} d\tau = \frac
{\xi^{1-\eta} }{1-\eta} (\xi ^{\eta -1} - (\frac t 2 +\xi)^{\eta-1})
  \leq \frac 1 {1-\eta}.
\end{equation*}
Hence, by the estimate \eqref{eq:TraceSpace}
\begin{equation}\label{eq:EstimateV2}
\begin{aligned} 
  \|v_2 (t)&\|_{C^{s+\beta}}  \\&\leq C \sup_{\xi \in (0,T/2)}
  \left(\xi ^{1-\eta}\left( \| (G_{\xi}
   \ast v_2(t))\|_{C^{6+\beta-s\eta}}\right) + \|\partial_\xi (G_{\xi} \ast
   v_2(t))\|_{C^{s+\beta-s\eta}}\right) \\
  &\leq C t^{\theta-1}\|f\|_{Y^{\beta,\theta}_T}.
\end{aligned}
\end{equation}
From (\ref{eq:ReprensentationFormula}),\eqref{eq:EstimateHomogeneousTerm}, \eqref{eq:EstimateV1}, and
\eqref{eq:EstimateV2} we obtain the desired estimate for $\|u\|_{C^{s+\beta}}.$

The estimate for $\partial_t u$ then follows from $\partial_t u = f
- (-\Delta) ^{s/2} u$ and the triangle inequality.
\end{proof}

\begin{lemma}[Maximal regularity] 
\label{lem:MaximalRegularity} 
Let $\Lambda,T>0$, $n \in
  \mathbb N$, and $\beta>0$, $\theta\in (0,1)$ with $\beta+s\theta \notin
  \mathbb N$ be given. 
  Then there is a constant $C=C(\Lambda, \beta, \theta,n,T),
  <\infty$ such that the following holds:
For all $$a\in C^1([0,T],h^\beta (\mathbb R /\mathbb
  Z, [1/\Lambda,\infty))) , \quad  b\in C^0((0,T), L(h^\beta(\mathbb R /
  \mathbb Z, \mathbb R^n),h^\beta(\mathbb R / \mathbb Z, \mathbb R^n)))$$
  with
  \[\|a\|_{C^1([0,T],C^\beta)}+t^{1-\theta} \|b(t)\|_{L({h^\beta, h^\beta})}\leq \Lambda \] and all $u\in C^1((0,T),
h^\beta(\mathbb R / \mathbb Z, \mathbb R^n))\cap C^0((0,T),h^{s+\beta}) \cap C^0([0,T], h^{\beta + s\theta})$
we have
\begin{align*}
 &\sup_{t\in[0,T]}t^{1-\theta}\left\{\|\partial_t u(t)\|_{C^\beta} +
   \|u(t)\|_{C^{s+\beta}}\right\} \\
& \quad \leq C \left(
 \sup_{t\in[0,T]}t^{1-\theta}\|\partial_t u(t) + a(t) Q^{s-1}u(t) +b(t)u(t)\|_{C^{\beta}} +
 \|u(0)\|_{h^{\beta + s \theta}}\right). 
\end{align*}
\end{lemma}

\begin{proof}

Note that it is enough to prove the statement for small $T$. Let us
fix $T_0>0$ and assume that $T \leq T_0$. Furthermore, we use the
embedding $h^\beta (\mathbb R / \mathbb Z, \mathbb R^n) \rightarrow
h^\beta (\mathbb R , \mathbb R^n)$ and extend the definition of $Q^{s-1}$
to functions $f$ defined on $\mathbb R$ by setting
\begin{equation*}
  Q^{s-1}f(x) := \lim_{\varepsilon \searrow 0}
  \int_{[-1/2,1/2]-[\varepsilon, \varepsilon]} \left(2\frac
    {f(u+w)-f(u)-wf'(u)}{w^2} - f''(x) \right) \frac {dw}{w^2}.
\end{equation*}

\vspace{2em}

 \noindent {\bf Step 1:  $\beta \in (0,1)$ and $b=0$}

\vspace{2em}

\noindent Let $\phi, \psi \in C^\infty (\mathbb R)$ be two cutoff functions
satisfying 
\begin{gather*}
\chi_{B_{1/2}(0)} \leq \phi \leq \chi_{B_1 (0)} \\
\chi_{B_{2}(0)} \leq \psi \leq \chi_{B_4 (0)}. 
\end{gather*}
and $\phi_r(x) := \phi(x/r )$, $\psi_r(x) =\psi(x/r)$. We set 
\begin{equation*}
  f = \partial_t u + a Q^{s-1} u.
\end{equation*}
For $r <
1/8$ we set $a_0 = a(0,0)$ and calculate
\begin{align*}
  \partial_t (u\phi_r) +  a_0 c_s &(-\Delta_{\mathbb R})^{s/2} (u\phi_r) =
  (\partial_t u + a Q^{s-1}u)  \phi_r - a ( Q^{s-1}(u) 
  \phi_r  -Q^{s-1}(u \phi_r))  \\
& \quad - (a-a_0) Q^{s-1} (u \phi_r) - a_0
  (Q^{s-1} (u \phi_r) - c_s (-\Delta_{\mathbb R})^{s/2} (u\phi_r)  \\ &= f
 \phi_r - f_1 - f_2 - f_3, 
\end{align*}
where
\begin{align*}
f_1&:=a ( Q^{s-1}(u) 
  \phi_r  -Q^{s-1}(u \phi_r)) \\
f_2&:= (a-a_0) Q^{s-1} (u \phi_r) \\
f_3&:=  a_0
  (Q^{s-1} (u \phi_r) -c_s(-\Delta_{\mathbb R})^{s/2} (u\phi_r).\\
\end{align*}
From Lemma~\ref{lem:MaximalRegularityForConstantCoefficients} we get
\begin{align*}
  t^{1-\theta}&(\|\partial_t u(t) \phi_r\|_{C^\beta} +
    \|u(t) \phi_r\|_{C^{s+\beta}} ) \\ &\leq C
 \bigg( \sup_{s\in [0,T]}s^{1-\theta}\left(\|f(s)\phi_r\|_{C^\beta}
    +\|f_1(s)\|_{C^\beta} + \|f_2(s)\|_{C^\beta} +
    \|f_3(s)\|_{C^\beta } 
  \right) \\
&+ \|u(0)\|_{C^{\beta+s\theta}}\bigg).
\end{align*}
Using Lemma~\ref{lem:LeibnizRule}, we obtain
\begin{equation*}
  \|f_1(s)\|_{C^\beta}\leq C \Lambda \|u(s)\|_{C^{2+\beta}}
  \|\phi_r\|_{C^{s+\beta}}.
\end{equation*}
Using $|a(x,t)- a_0| \leq \Lambda (|x|^\beta +
T)$, we derive
\begin{align*}
  \|f_2\|_{C^\beta} &
\leq \| \psi_r (a-a_0) Q^{s-1}(u\phi_r)\|_{C^\beta}
  + \|(\psi_r -1) (a-a_0) Q^{s-1}(u\phi_r)\|_{C^\beta} \\
& \leq C_1 \Lambda ( (2r)^{\tilde\beta} + T) \|u\phi_r\|_{C^{s+\beta}}+  \|(\psi_r -1) (a-a(0)) Q^{s-1}(u\phi_r)\|_{C^\beta}
\end{align*}
where $C_1$ does not depend on $r$ or $T$.
Since $\spt{1-\psi_r} \subset \mathbb R -B_{4r}(0)$ and $\spt \phi_r
\subset B_r(0)$, we see that
\begin{multline*}
  (\psi_r -1) (a-a_0) Q^{s-1}(u\phi_r) (x)  \\= (\psi_r(x)-1) (a(x) -a(0)) \! \!
  \int_{[-1/2,1/2]-[-r,r]} \! \! \! \frac {u(x+w)\phi_r(x+r)}{w^2} dw
\end{multline*}
and hence
\begin{equation*}
  \|(\psi_r -1) (a-a_0) Q^{s-1}(u\phi_r)\|_{C^\beta}\leq C(\Lambda, \psi,
  \phi, r) \|u\|_{C^\beta}.
\end{equation*}
This leads to 
\begin{equation*}\label{eq:EstimateForf1}
\|f_2(s)\|_{C^\beta}  \leq C_1 \Lambda ( (2r)^{\tilde\beta} + T)
\|u(s)\phi_r\|_{C^{s+\beta}}+  C \|u(s)\|_{C^\beta}.
\end{equation*}

Furthermore,
\begin{align*}
  \|f_3\|_{C^\beta}\leq C(\Lambda, \phi,r) \|u\|_{C^{2+\beta}}
\end{align*}
since for $v\in C^{s+\beta}(\mathbb R)$ with compact support we have
\begin{equation*}
Q^{s-1} (v) -c_s (-\Delta_{\mathbb R})^{s/2} (v)  =
-\int_{\mathbb R -[-1/2,1/2] } \left(2 \frac {v(u+w)-v(u)}{w^2},
  v''(u)\right) \frac {dw}{w^2}. 
\end{equation*}
and hence
\begin{equation*}
  \|Q^{s-1} (v) - c_s (-\Delta_{\mathbb R})^{s/2} (v)\|_{C^\beta} \leq (4
  \|v\|_{C^\beta} + \|v\|_{C^{2+\beta}}) \cdot 2 \int_{\frac 1
    2}^\infty \frac 1 {w^2} dw \leq 24 \|v\|_{C^{2+\beta}}.
\end{equation*}

Summing up, we thus get 
\begin{multline*}
 \sup_{t \in (0,T]} t^{1-\theta}(\|\partial_t (u \phi_r)(t)\|_{C^\beta} +
    \|u(t)\phi_r\|_{C^{s+\beta}} ) \\ \leq C_1 \Lambda ((2r)^{\tilde\beta} + T)
  \sup_{s\in (0,T]}s^{1-\theta}
    \|u\phi_r\|_{C^{s+\beta}}  \\
  + C
    (\phi,\psi,r ,\Lambda))  \bigg(\sup_{s\in (0,T]}(s^{1-\theta}
    \|f(s)\|_{C^\beta} + s^{1-\theta}\|u(s)\|_{C^{s}}   )+\|u_0\|_{C^{\beta + s \theta}}\bigg),
\end{multline*}
where $C_1$ does not depend on $r$.
Choosing $r$ and $T$ small enough and absorbing the first term on the
right hand side, leads to
\begin{align*}
 \sup_{t\in (0,T]}& t^{1-\theta}\left(\|\partial_t u\|_{C^\beta(B_{r/2}(0))} +
    \|u\|_{C^{s+\beta}(B_{r/2}(0))} \right) \\ &\leq  C
    (\phi,\psi,r,\Lambda)) \bigg( \sup_{s\in (0,T]}(s^{1-\theta}
    \|f(s)\|_{C^\beta} +s^{1-\theta} \|u(s)\|_{C^{s}} )+\|u(0)\|_{C^{\beta + s \theta}}\bigg).
\end{align*}
Of course, the same inequality holds for all balls of radius
$r/4$. Thus, covering $[0,1]$  with balls of radius $r/4$ we obtain
\begin{align*}
 \sup_{t \in (0,T]}t^{1-\theta}& \big(\|\partial_t u(t)\|_{C^\beta} + \|u(t)\|_{C^{s+\beta}}
\big) \\
&\leq C\left(
 \sup_{s\in (0,T]} (s^{1-\theta}\|f(s)\|_{C^\beta} +s^{1-\theta}
  \|u(s)\|_{C^{s}}   )+ \|u(0)\|_{C^{\beta + s \theta}}\right).
\end{align*}
Using the interpolation inequality for H\"older spaces
\begin{equation*}
  \|u\|_{C^s } \leq \varepsilon \|u\|_{C^{s+\beta}} + C(\varepsilon) \|u\|_{C^\beta}
\end{equation*}
 and absorbing, this leads to 
\begin{multline*}
\sup_{t \in (0,T]}t^{1-\theta}\left(\|\partial_t u(t)\|_{C^\beta} + \|u(t)\|_{C^{s+\beta}}
\right) \\
\leq C \left(
  \sup_{s\in (0,T]} \left( s^{1-\theta}\|f(s)\|_{C^\beta} + \|u(s)\|_{C^\beta
    } \right) + \|u(0)\|_{C^{\beta + s \theta}}  \right).
\end{multline*}
Since 
\begin{align*}
\|u(s) \|_{C^\beta} &\leq  \int_0^s \|\partial_t u
  (\tau) \|_{C^\beta} d\tau + \|u(0)\|_{C^{\beta+ s \theta}} \\
&\leq \int_0^T\tau^{\theta -1 } d\tau \sup_{\tau \in [0,T]}
\tau^{1-\theta}\|\partial_t u(\tau)\|_{C^\beta}  + \|u(0)\|_{C^{\beta+ s \theta}}
\\
&\leq \frac 1 {\theta} T^\theta
\sup_{\tau \in (0,T]}\tau^{1-\theta}\|\partial_t u(\tau)\|_{C^\beta} +\|u(0)\|_{C^{\beta+ s \theta}},
\end{align*}
 we  can absorb the first term for $T>0$ small enough to obtain
\begin{multline*}
\sup_{s\in (0,T)}s^{1-\theta}\left(\|\partial_t u(s)\|_{C^\beta} + \|u(s)\|_{C^{s+\beta}}
\right) \\
\leq C(\Lambda))
  \left( \left(\sup_{s\in [0,T]}s^{1-\theta}\|f\|_{C^\beta}\right)   +
    \|u(0)\|_{C^{\beta+ s \theta}} \right).
\end{multline*}

\vspace{2em}

\noindent {\bf Step 2: General $\beta$ but $b=0$}

\vspace{2em}

\noindent
Let $k\in \mathbb N_0$, $\tilde \beta \in (0,1)$ and let the lemma be
true for $\beta = k+\tilde \beta$. We deduce the statement for $\beta =
k+1 + \tilde \beta$. 

From $\partial_t u + a Q^{s-1}u =f$ we deduce that 
\begin{equation*}
  \partial_t (\partial_x u ) + a Q^{s-1} \partial_x u =\partial_x f -
  (\partial_x a) Q^{s-1}u 
\end{equation*}
and we obtain by applying the induction hypothesis to get
\begin{align*}
 \|\partial_x u\|_{X^{k+\beta,\theta}_T} &\leq C  \left(\|\partial_x
 f\|_{Y^{k+\tilde \beta, \theta}_T} +  \|(\partial_x a)
 Q^{s-1}u\|_{Y^{k+\tilde \beta, \theta}_T} + \|\partial_x u(0)\|_{C^{\beta+s\theta}}\right)
\\
&\leq C \left( \|f\|_{Y^{k+1+\tilde \beta, \theta}_T} + \Lambda
\|u\|_{X^{k+\tilde \beta, \theta}_T} + \|\partial_x
u(0)\|_{C^{\beta+s\theta}}\right) 
\\
&\leq  C \left(
\|f\|_{Y^{k+1+\tilde \beta, \theta}_T}  + \|\partial_x
u(0)\|_{C^{\beta+s\theta}} \right).
\end{align*}

\vspace{2em}

\noindent  {\bf Step 3: General $\beta$ and $b$}

\vspace{2em}

\noindent From Step 2 we get
\begin{equation*}
  \|u\|_{X^{\beta, \theta}_T} \leq C\left(\|f\|_{Y^{\beta,\theta}_T}
    + \|((t,x) \mapsto b(t)(u(t))(x))\|_{Y^{\beta,\theta}_T} +\|u_0\|_{C^{\beta+s\theta}}\right).
\end{equation*}
As 
\begin{equation*}
  \|((t,x) \mapsto b(t)(u(t))(x))\|_{Y^{\beta,\theta}_T}= \sup_{s\in (0,T]} t^{1-\theta}
  \|b(t)(u(t))\|_{C^\beta} \leq \Lambda \sup_{s\in (0,T]}\|u(s)\|_{C^\beta}
\end{equation*}
and
\begin{align*}
\|u(s) \|_{C^\beta} &\leq  \int_0^s \|\partial_t u
  (\tau) \|_{C^\beta} d\tau + \|u_0\|_{C^{\beta+ s \theta}} \\
&\leq \int_0^T\tau^{\theta -1 } d\tau \sup_{\tau \in [0,T]}
\tau^{1-\theta}\|\partial_t u(\tau)\|_{C^\beta}  + \|u(0)\|_{C^{\beta+ s \theta}}
\\
&\leq \frac 1 {\theta} T^\theta
\sup_{\tau \in (0,T]}\tau^{1-\theta}\|\partial_t u(\tau)\|_{C^\beta} +\|u(0)\|_{C^{\beta+ s \theta}} .
\end{align*}
 we  get, absorbing the first term for $T>0$ small enough, 
\begin{equation*}
\|u\|_{X^{\beta, \theta}_T} \leq C\left(\|f\|_{Y^{\beta,\theta}_T}
   +\|u_0\|_{C^{\beta+s\theta}}\right).
\end{equation*}
\end{proof}

Now we can finally prove Theorem~\ref{lem:JIsAnIsomorphism}.

\begin{proof} [Proof of Theorem~\ref{lem:JIsAnIsomorphism}]
It only remains to show that this mappping $J$ is onto. To prove this, we use the method of
continuity for the family of operators $J_\tau: u \mapsto (u(0),\partial_t u +
((1-\tau)\lambda Q^{s-1} u + \tau (a Q^{s-1}u + bu)$. In view of Lemma 5.2 in
\cite{Gilbarg2001}, we have to show is that $J_0$ is onto. 

By \cite[Lemma 2.3]{He2000} and \cite[Proposition~1.4]{Reiter2012} we have for all  $f \in H^s (\mathbb R / \mathbb
Z, \mathbb R^n )$
\begin{equation*}
  Q^{s-1}(f)^\wedge(k)= \lambda_k|2\pi k|^s \hat f(k) 
\end{equation*}
where
\begin{equation*}
  \lambda_k =c_s + O(\frac 1 k).
\end{equation*}
for some positive constants $c_s$.
For $u_0,f$ in $C^\infty$
a smooth solution of the equation
\begin{equation*}
\begin{cases}
  \partial_t u + \lambda Q^{s-1}u = f \quad  &\forall t \in (0,T] \\
  u(0)=u_0
\end{cases}
\end{equation*}
can be given by Duhamel's formula
\begin{equation*}
  u(t,x) = \sum_{k\in \mathbb Z} \hat{u}_0(k)e^{-t\lambda \lambda_k|2\pi k|^s}
  e^{2\pi ikx} + \int_0^t \sum_{k\in \mathbb Z} (f(s))^\wedge(k)
  e^{-(t-s) \lambda \lambda_k|2\pi k|^s} ds.
\end{equation*}

Let now $u_0 \in h^{\beta +  s \theta }(\mathbb R / \mathbb Z,
\mathbb R ^n)$ and $f \in Y_T ^{\beta, \theta}$. We set $f_k(t) :=
f(t+1/k)$ and observe that 
\begin{equation*}
  f_k \rightarrow f  \quad \text{in }C^0((0,T],h^\beta(\mathbb R /
  \mathbb Z, \mathbb R ^n) ).
\end{equation*}
 Since $f_k \in C^0 ([0, T-1/k], h^\beta(\mathbb R / \mathbb Z,
 \mathbb R^n))$, we can
find functions $f_{n,k} \in C^{\infty } ([0, T-1/k] \times \mathbb R /
\mathbb Z, \mathbb R^n) $ such that $f_{n,k} \rightarrow f_n$ in $C^0
([0, T-1/k], C^\beta)$ for $n\rightarrow \infty$ and smooth $u_0^{(k)}$
converging to $u_0$ in $h^{\beta+s\theta}$. Let $u_{n,k} \in C^\infty$ be the solution of
\begin{equation*}
  \begin{cases}
    \partial_t u_{n,k} + Q^{s-1} u_{n,k} = f_{n,k} \\
    u_{n,k} (0) = u^{(k)}_0.
  \end{cases}
\end{equation*}
Using the a priori estimate of Lemma~\ref{lem:MaximalRegularity}, one
deduces that the sequence $\{u_{n,k}\}_{n \in \mathbb N}$ is a Cauchy
sequence in $X_{T-\varepsilon} ^{\beta, \theta}$ for every
$\varepsilon >0$. The limit $u_n$ solves the
equation
\begin{equation*}
  \begin{cases}
    \partial_t u_{n} + Q^{s-1} u_{n} = f_{n} \\
    u_{n} (0) = u_0.
  \end{cases}
\end{equation*}
 
Using the a priori estimates again, one sees that
$\{u_n\}_{n\in \mathbb N}$ is bounded in $X_{T-\varepsilon}^{\beta,\theta}$. Since
$X^{\beta,\theta}_{T-\varepsilon}$ is embedded continuously in $C^{\theta/2}([0,T-\varepsilon],
h^{\beta+s \frac \theta2 })$ and $C^{1-\eta}([\delta,T-\varepsilon],
h^{\beta+s\eta})$ for all $\eta \in (0,1), \varepsilon, \delta >0$, we can assume, after going to a subsequence,
that there is a $u_\infty \in X^{\beta, \theta}_T$ such that
\begin{align*}
  u_n & \rightarrow u_\infty &\text{in } C^{0}((0,T-\varepsilon), C^{s+\beta}) 
\end{align*}
for $0\leq \beta < \beta, \varepsilon >0$ and
\begin{equation*}
 u_\infty (0) = u_0.
\end{equation*}
Hence we get
\begin{equation*}
    \partial_t u_n = f_n -  \lambda Q^{s-1} u_n \rightarrow f + \lambda Q^{s-1}
    u_{\infty} \text{ in } C^{0}((0,T-\varepsilon), C^{s+\beta})
\end{equation*}
for all $\varepsilon >0$ which implies that $u_\infty$ solves
\begin{equation*}
 \begin{cases}
   \partial_t u_\infty +  \lambda Q^{s-1} u_\infty = f  \\
    u_{\infty} (0) = u_0.
 \end{cases}
\end{equation*}
\end{proof}

\subsubsection{The Quasilinear Equation}
Now we are in position to prove short time existence for quasilinear equations and $C^1$-dependence on the initial data.
\begin{proposition} [Short time existence] \label{prop:GeneralShortTimeExistence}
Let $0<\beta$, $0 < \theta < \sigma <1 $, $\beta,\beta+s\theta,
\beta +s \sigma \notin \mathbb N_0$, $U \subset
C^{\beta+s\theta}(\mathbb R / \mathbb Z, \mathbb R^n)$ be open and let
$a \in  C^1(U,
C^\beta(\mathbb R / \mathbb Z, (0,\infty)))$, $f \in  C^1(U,
C^\beta(\mathbb R / \mathbb Z, \mathbb R^n))$.

Then for every $u_0 \in h^{\beta+s\sigma}(\mathbb R
/ \mathbb Z, \mathbb R^n)\cap U$ there is a constant $T>0$ and 
a unique $u \in C^0([0,T),h^{\beta+ s \sigma}(\mathbb R / \mathbb Z,
\mathbb R ^n))
\cap C^1((0,T), h^{s+\beta}(\mathbb R / \mathbb Z, \mathbb R^n))$ such that 
\begin{equation*}
  \begin{cases} \partial_t u +a(u) Q^{s-1}(u) =f(u) \\
    u(0)=u_0.
  \end{cases}
\end{equation*}
\end{proposition}

\begin{proof}

Let us first prove the existence. We set $\tilde
X^{\beta,\sigma}_T:=\{w \in X^{\beta,\sigma}_T: w(0)=u_0\}$. For $w \in
\tilde X^{\beta,\sigma}_T(\mathbb R / \mathbb Z, \mathbb R^n)$ let $\Phi w$
denote the solution $u$ of the problem
\begin{equation*}
\begin{cases}
  \partial_t u + A_0 u = B(w)w + f(w), \\
  u(0)=u_0
\end{cases}
\end{equation*}
where $A_0 = a(u_0)Q^{s-1}$ and $B(w) = (a(u_0)-a(w)) Q^{s-1} $.

Let $v$ 
be the solution of 
\begin{equation*}
  \begin{cases} \partial_t  v +a(u_0) Q^{s-1}(v) =f(u_0) \\
   \tilde  v(0)=u_0.
  \end{cases}
\end{equation*}
and $\mathcal B_r (v):=\{w \in \tilde X_T^{\beta, \sigma}:
\|w-v\|_{X_T^{\beta,\sigma}}\leq r\}.$
We will show that $\Phi$ defines a contraction on $\mathcal B_r (v)$
if $r,T>0$ are small enough.

Since $a\in C^1(U, h^\beta(\mathbb R / \mathbb Z, \mathbb R^n ))$, we get
$\|B(z)\|_{L(C^{\beta+s}, C^\beta)} \leq
C\|z-u\|_{C^{\beta+s\theta}}$ for all $z,u \in C^{\beta+s\theta}$ close to $u_0$.  

Let $w_1,w_2 \in \mathcal B_r (v)$, $r \leq 1$.
Using that  the space $X^{\beta,\sigma}_T $ is embedded continuously in $C^{\sigma-\theta}([0,T],h^{\beta+s\theta}(\mathbb R / \mathbb Z,
\mathbb R^n))$ and $w_1(0)=w_2(0)=v(0)=u_0$ we get 
\begin{align}
  \|w_2(t)-u_0\|_{C^{\beta + s \theta}} &\leq C t^{\sigma-\theta}
  \|w_2 \|_{X^{\beta,\sigma}_T} \leq C t^{\sigma - \theta}
  (\|v\|_{X^{\beta,\sigma}_T} +r) , \label{eq:EstimateW1U0}\\
 \|w_1(t)-w_2(t)\|_{C^{\beta + s \theta}} &\leq C t^{\sigma-\theta}
  \|w_1 -w_2 \|_{X^{\beta,\sigma}_T}  \label{eq:EstimateW1W2}.
\end{align}

Using Lemma~\ref{lem:MaximalRegularity}, we estimate
\begin{align*}
 \|\Phi w_1-\Phi w_2\|_{X_T^{\beta,\sigma}} &
\leq C \|B(w_1)w_1-B(w_2)w_2\|_{Y_T^{\beta,\theta}}  + C
   \|fw_1-fw_2\|_{Y_T^{\beta,\sigma}}.
\end{align*}
and
\begin{align*}
  t^{1-\sigma}\|f(w_1(t))-f(w_2(t))\|_{C^\beta} &
  \leq C t^{1-\sigma}\|w_1(t) -w_2(t)\|_{C^{\beta+s\theta}}\\
  &\lstackrel{\eqref{eq:EstimateW1W2}}\leq C  T^{1-\theta} \|w_1 -w_2\|_{X^{\beta,\sigma}_T}.
\end{align*}
Furthermore,
\begin{align*}
  \|B(w_1)w_1 &-B(w_2)w_2\|_{Y^{\beta,\sigma}_T}  \\&\leq \|(B(w_1)-B(w_2))w_1\|_{Y^{\beta,\sigma}_T} + \|B(w_2)(w_1-w_2)\|_{Y^{\beta,\sigma}_T} \\
&\leq C \sup_{t \in (0,T]} t^{1-\sigma} \big( \|w_1(t) -w_2(t)\|_{C^{\beta+ s \theta}}
\|w_1(t)\|_{C^{s+\beta}}  
\\ & \quad \quad \quad \quad \quad \quad \quad \quad \quad+ \|w_2(t)-u_0\|_{C^{\beta +s\theta}}
\|w_1(t)-w_2(t)\|_{C^{s+\beta}} \big) \\
& \lstackrel{(\ref{eq:EstimateW1W2})\&(\ref{eq:EstimateW1U0})}{\leq}C \sup_{t \in (0,T]} \big( t^{\sigma-\theta}\|w_1
-w_2\|_{X_T^{\beta, \sigma}}
\|w_1\|_{X_T^{\beta,\sigma}}  \\ & \quad \quad \quad  \quad \quad \quad
\quad
\quad \quad + t^{\sigma-\theta} (\|v\|_{X^{\beta, \sigma}_T}+r)
t^{1-\sigma}\|w_1(t)-w_2(t)\|_{C^{s+\beta}} \big)  \\
& \leq C (  T ^{\sigma-\theta}
(\|v\|_{X^{\beta,\sigma}_T} + r) \|w_1 -w_2\|_{X^{\beta, \sigma}_T}.
\end{align*}
Thus,
\begin{align*}
 \|\Phi(w_1)-\Phi(w_2)\|_{Y^{\beta,\sigma}_T} \leq C
 ( T^{1-\sigma} + T^{\sigma-\theta} (\|v\|_{X^{\beta,\sigma}_T} + r)) \|w_1 -w_2 \|_{Y^{\beta,\sigma}_T}
\end{align*}
and hence $\Phi$ is a contraction on $\mathcal B_r(v)$ if $T$ and $r$ are
small enough.

Similarly, we deduce from the definition of $v$ that
\begin{align*}
  \|\Phi(w)-v&\|_{X^{\beta,\sigma}_T} \leq C \|B(w)w\|_{Y^{\beta,\sigma}_T} +
  \|f(w)-f(u_0)\|_{Y^{\beta,\sigma}_T} \\
&\leq C T^{\sigma-\theta} \|w\|_{X^{\beta,\sigma}_T} \|w-v\|_{X^{\beta,\theta}_T} +
T^{1-\theta} \|w-v\|_{X^{\beta,\sigma}_T} + \|v-u_0\|_{C^{\beta + s
    }} 
\\
&< \|w-v\|_{X^{\beta,\sigma}_T}
\end{align*}
if $T$ and $r$ are small enough. Then $\phi(\mathcal B_r(v)) \subset \phi(\mathcal B_r
(v))$. Hence, by Banach's fixed-point theorem there is a
unique $u \in B_r (v)$ with $\partial_t u + a(u) Q^{s-1}(u) u = f(u)$.

For the uniqueness statement, we only have to guarantee that every solution is in
$Y^{\beta, \sigma}_T$.  But this follows from Lemma~\ref{lem:MaximalRegularity}.
\end{proof}

\begin{proposition}[Dependence on the data] \label{prop:ContinuousDependence}
Let $a,b$ be as in Proposition~\ref{prop:GeneralShortTimeExistence}
and $u\in Y^{\theta, \beta}_T$ be a solution of  the quasilinear equation
\begin{equation*}
  \begin{cases} \partial_t u +a(u) Q^{s-1}(u) =0 \\
    u(0)=u_0.
  \end{cases}
\end{equation*}
Then there is a neighborhood $U$ of $u_0$ in $h^{\beta + s\theta}$
such that for all $x \in U$ there is a solution $u_x$ of 
\begin{equation*}
  \begin{cases}
    \partial_t u + a(u)Q^{s-1}u = 0 \\
    u(0)=x
  \end{cases}
\end{equation*}
Furthermore, the mapping
\begin{gather*}
 U  \rightarrow Y_T^{\theta, \beta} \\
 x \mapsto u_x
\end{gather*}
is $C^1$.
\end{proposition}

\begin{proof}
We define $\Phi:h^{\beta + s \theta}(\mathbb R/ \mathbb Z, \mathbb
R^n) \times X^{\beta, \theta}_T \rightarrow Y^{\beta, \theta}_T $ by 
\begin{equation*}
  \Phi(x,u):= (u(0)-x, \partial_t u + a(u) Q^{s-1}u)
\end{equation*}
Then the Fr{\'e}chet derivative of $\phi$ with respect to $u$ reads as 
\begin{equation*}
  \frac{\partial \phi(x,u)}{\partial u} (h) = (h, \partial_t h +
  a(u)Q^{s-1}h + a'(u)h Q^{s-1}u).
\end{equation*}
Setting $a(t) =a(u(t))$ and $b(t)(h) = a'(u) h Q^{s-1}u$,
Lemma~\ref{lem:MaximalRegularity} tells us that this is an isomorphism between $X^{\theta,
  \beta}_T$ and $h^\beta\times Y^{\theta, \beta}_T $. 
Hence, the statement of the lemma follows from the implicit function
theorem on Banach spaces.
\end{proof}

\subsubsection{Proof of
  Theorem~\protect{\ref{thm:ShortTimeExistenceForGraphs}} } \label{subsec:ProofOfShortTimeExistence}

%

Since the normal bundle of a curve
is trivial, we can find smooth normal vector fields $\nu_{1},\dots,\nu_{n-1}\in C^{\infty}(\mathbb{R}/\mathbb{Z},\mathbb{R}^{n})$
such that for each of $x\in\mathbb{R}/\mathbb{Z}$ the vectors $\nu_{1}(x),\ldots,\nu_{n-1}(x)$
form an orthonormal basis of the space of all normal vectors to $\g_{0}$
at $x$. Let $\tilde {\mathcal{V}}_r (\g) := \{(\phi_1, \ldots, \phi_{n-1}) \in h^{\beta}(\mathbb
R/ \mathbb Z, \mathbb R^{n-1}): \sum_{i=1}^{n-1} \phi_i \nu_i \in
\mathcal V_r(\g)\}$. 

 If we have $N_{t}=
\sum_{i=1}^{n-1} \phi_{i,t}\nu_{i}$, $(\phi_{1,t}, \dots \phi_{n-1,t}) \in \tilde{ \mathcal V}_r (\g)$, then \eqref{eq:EvolutionEquationNormalGraph},
using Theorem~\ref{thm:QuasilinearStructure}, can be written as \begin{align*}
 \sum_{i=1}^{n-1}\left(\partial_{t}\phi_{i,t}\right)\left(P_{\g'(u)}^{\bot}\nu_{i}\right)
& =- \frac 2 {|\g'|^{s}} P_{\g'}^{\bot}
\left( Q^{\alpha}\left(\g_{0}+ \sum_{i=1}^{n-1} \phi_{i,t}\nu_{i}\right)\right)+F(\g_{0}+  \sum_{i=1}^{n-1} \phi_{i,t}\nu_i) + \lambda
\kappa
\\
 &
 =-\frac{2}{|\g'|^{s}}   \sum_{i=1}^{n-1} \left( Q^{\alpha}\phi_{i,t}\right)P_{\g'}^{\bot}\nu_{i}-F(\g_{0}+  \sum_{i=1}^{n-1} \phi_{i,t}\nu_{i})
 \\
&\quad - \frac 2 {|\g'|^{s}}P_{\g'}^{\bot}\left(  \sum_{i=1}^{n-1} Q^{\alpha}\left(\phi_{i,t}\nu_{i}\right)-\sum_{i=1}^{n-1} \left(   Q^{\alpha}\phi_{i,t}\right)\nu_{i}+Q^{\alpha}\g_{0}\right)\\
 & \quad -F(\g_{0}+ \sum_{i=1}^{n-1}\phi_{i,t}\nu_{i}) + \lambda \kappa\\
 & = -\frac{2}{|\g'|^{s}}   \sum_{i=1}^{n-1} \left(Q^{\alpha}\phi_{i,t}\right) P_{\g'}^{\bot}\nu_{i}+\tilde{F}_{\g_{0}}(\phi_{t})
 + \lambda \kappa
\end{align*}
where
\begin{align*}
\tilde{F}_{\g_{0}}(\phi_{t})&=-F(\g_{0}+   \sum_{i=1}^{n-1}\phi_{i,t}\nu_{i})-\frac{2}{|\g'|^{s}}P_{\g'}^{\bot}\left(   \sum_{i=1}^{n-1} Q^{\alpha}\left(\phi_{i,t}\nu_{i}\right)-  \sum_{i=1}^{n-1} \left(Q^{\alpha}\phi_{i,t}\right)\nu_{i}+Q^{\alpha}\g_{0}\right)\\
& \quad+\frac{2}{|\g'|^{s}}\left(   \sum_{i=1}^{n-1} \left(Q^{\alpha}\phi_{i,t}\right)
P_{\g'}^{\bot}\nu_{i} - P_{\g'}^{\bot}\left(   \sum_{i=1}^{n-1} Q^{\alpha}\phi_{i,t}
\nu_{i}\right)\right).
\end{align*}
Using Lemma~\ref{lem:LeibnizRule}we see that $\tilde F\in
C^w(C^{\alpha + \beta}, C^{\beta})$ for all $\beta >0$. Furthermore, the term $\lambda \kappa$
belongs to $C^\omega(C^{\alpha + \beta}, C^{\beta})$.
Since (\ref{eq:RadiiComparabilityOfDerivatives}) implies that
$\{P_{\g'}^{\bot}\nu_{r}: r=1, \ldots, n-1\}$ is a basis of the normal space
of the curve $\g$ at the point $u$, the mapping $A:\mathbb{R}^{n-1}\rightarrow\left(\mathbb R \g'(u)\right)^{\bot},$
$(x_{1},\dots,x_{n-1})\mapsto   \sum_{i=1}^{n-1} x_{i}P_{\g'(u)}^{\bot}\nu_{i}$
is invertible as long as $\|\g'-\g_0'\|_{L^\infty } < 1$.

So we have brought the evolution equation into the form
\begin{equation} \label{eq:KeyToShortTimeExistence}
\partial_{t}\phi_{t}=\frac 2 {|\g'|^{s}}Q^{\alpha} \phi_{t}+A^{-1}\left(\tilde{F}(\phi_{t})\right)\end{equation}
where $A^{-1}\left(\tilde{F}(\phi_{t})\right)\in C^{\omega}(h^{\alpha+ \beta},h^{\beta})$ for all $\beta >0$. Now the statement follows from
Proposition~\ref{prop:GeneralShortTimeExistence},
Proposition~\ref{prop:ContinuousDependence}, and a standard
bootstrapping argument.
%


\subsubsection{ Proof of
  Theorem~\protect{\ref{thm:ShortTimeExistence} }}
   From Theorem~\ref{thm:ShortTimeExistenceForGraphs} we get a smooth solution of
   \begin{equation*}
    \left(\partial_t \gamma \right)^\bot = H^\alpha (\gamma_t) + \lambda \kappa_{\gamma_t}.
   \end{equation*}
  Let $\phi_t(x)$ for $(x,t) \in \mathbb R / \mathbb Z \times [0,T)$ denote the solution of
   \begin{equation*}
    \partial_t \phi_t (x) = \frac{- (\partial_t \gamma (\phi_t(x)))^T} {\gamma_t'(\phi_t(x))}.
   \end{equation*}
  We caculate for ${\tilde \gamma}_t = \gamma_t \circ \phi_t$
   \begin{equation*}
    \partial_t ({\tilde  \gamma}_t) = (\partial_t \gamma)^\bot \circ \phi_t = H^\alpha {\tilde \gamma}_t + \lambda \kappa_{{\tilde \gamma}_t}.
   \end{equation*}

\subsubsection{ Proof of
  Theorem~\protect{\ref{thm:ShortTimeExistenceNonSmooth} }}

The proof of Theorem~\ref{thm:ShortTimeExistenceNonSmooth} is an immediate consequence of Theorem~\ref{thm:ShortTimeExistenceForGraphs} and the following approximation argument
\begin{lemma} \label{lem:VrCoverAll}
  Let $r:h^{2,\beta}_{i,r} (\mathbb R / \mathbb Z , \mathbb R^n)
 \rightarrow (0,\infty)$ be a lower semi-continuous function. Then for every $\tilde \gamma \in h^{2+\beta}_{i,r}
(\mathbb R / \mathbb Z , \mathbb R^n)$ there is a $\gamma \in
C^\infty_{i,r}(\mathbb R/ \mathbb Z, \mathbb R^n)$, $N \in
\mathcal V_r (\gamma)$ and a diffeomorphism $\psi \in
C^{2+\beta}(\mathbb R / \mathbb Z , \mathbb R / \mathbb Z)$ such that
$\tilde \gamma \circ \psi = \gamma + N $.
\end{lemma}

\begin{proof}[Proof of
  Lemma~\protect{\ref{lem:VrCoverAll}}]

Let $\tilde \g \in h_{i,r}^{2+\beta}(\mathbb R / \mathbb Z, \mathbb R^n)$
and let us set $\g_\varepsilon := \phi_\varepsilon \ast \tilde \g$
where $\phi_\varepsilon (x)= \varepsilon^{-1} \phi(x/\varepsilon)$ is
a smooth smoothing kernel. Since $h_{i,r}^{2+\beta}(\mathbb R /
\mathbb Z, \mathbb R^n)$ is an open subset of $h^{2+\beta}(\mathbb R /
\mathbb Z, \mathbb R^n)$ and $(\varepsilon \rightarrow
\g_\varepsilon) \in C^0([0,\infty), h^{2+\beta}(\mathbb R /
\mathbb Z, \mathbb R^n)$, we get $\g_\varepsilon \in
h^{2+\beta}_{i,r}(\mathbb R/ \mathbb Z, \mathbb R^n)$ for
$\varepsilon$ small enough.

Since $\g_{\varepsilon} \in C^0 ([0,\infty),
h^{2+\beta}(\mathbb R / \mathbb Z, \mathbb R^n))$ there is an open neighborhood
$U$ of the set $\g(\mathbb R / \mathbb Z)$ and an 
$\varepsilon_0 >0$ such that the nearest neighborhood retract
$r_\varepsilon:U \rightarrow \mathbb Z / \mathbb R$ onto $\g_\varepsilon$
is defined on $U$ simultaneously for all $0 \leq \varepsilon < \varepsilon_0$. 
Note, that these retracts $r_\varepsilon$ are smooth as the curves $\g_\varepsilon$
are smooth and
\begin{gather*}
 [0,\varepsilon_0) \times U \rightarrow \mathbb R / \mathbb Z \\
 (\varepsilon,x) \mapsto r_{\varepsilon}(x)
\end{gather*}
belongs to $C^0 ([0,\varepsilon_0),C^{1+\beta})$.

We set $\psi_\varepsilon(x):= r_\varepsilon(\g(x))$. Now $\psi_{0} = id_{\mathbb R / \mathbb Z}$,
$\psi_{\varepsilon}$ is a $C^{1+\beta}$ diffeomorphism for
$\varepsilon>0$ small enough since the subset of diffeomorphism is open
in $C^{1+\beta}$.
Hence, we can set 
$N_\varepsilon(x) = \g_\varepsilon(
\psi_{\varepsilon}^{-1}(x))-\g_0 (x)$  for $\varepsilon_0$ small enough. We will show that
$\g:=\g_\varepsilon$, $N:=N_\varepsilon$, and
$\psi:=\psi_\varepsilon$  satisfy the statement of the lemma if
$\varepsilon$ is small enough.


 From $\psi_\varepsilon(x):=
r_\varepsilon(\g(x))$ we deduce that $\psi_{\varepsilon} $ is in
fact a $C^{2+\beta}$ diffeomorphism, as $r_\varepsilon$ is smooth.

Since
\begin{align*}
  N_{\varepsilon} = \g_\varepsilon\circ
\psi_{\varepsilon}^{-1} -\g_0  \in
C^0([0,\infty),C^{1+\beta}(\mathbb R / \mathbb Z, \mathbb R^n))
\end{align*}
and $N_0 =0$, we furthermore we get
\begin{equation*}
  \|N_\varepsilon\|_{C^1} \xrightarrow{\varepsilon \searrow 0}0.
\end{equation*}
Since $r$ is lower semi-continuous and
$r{(\g_{0})} >0$, we hence get
$\|N_{\varepsilon}\|_{C^1} < r({\g_{\varepsilon}})$ for
small $\varepsilon$. As $N_{\varepsilon} \in h^{2,\beta}$ we deduce
that $N_\varepsilon  \in \mathcal V_{r} (\g_\varepsilon)$ if $
\varepsilon$ is small enough.

 \end{proof}


\section{Long Time 
Existence} \label{sec:LongTimeExistence}

The aim of this section is to prove the following long time existence result. 

\begin{theorem}\label{thm:LongTimeExistence}
Let $\g_0 \in C^{\infty}(\mathbb R / \mathbb Z, \mathbb R^n)$. Then
there exists a unique solution 
$\g \in C^0(\mathbb [0,\infty), C^\infty (\mathbb R / \mathbb Z, \mathbb R^n))
\cap C^{\infty}((0,\infty),C^\infty(\mathbb R / \mathbb Z, \mathbb R^n))$
to \eqref{eq:EvolutionEquation} with initial data $\g(0) = \g_0$. This solution subconverges, after suitable re-parameterizations and translations, to a smooth critical point of 
$E^\alpha + \lambda_\infty L$ where $\lambda_\infty = \lim_{t \rightarrow \infty} 
\lambda(t)$.
\end{theorem}

Let me first sketch the strategy of the proof. Since we are looking at negative gradient flows
of $E^{\alpha}$ we have for a solution $\g(t)$ of equation \eqref{eq:EvolutionEquation}
\begin{equation*}
E^\alpha(\g(t)) + \lambda L(\g) \leq E^{\alpha}(\g(0)) + \lambda L(\g(0))
\end{equation*}
for all $t \in [0,T)$. So both, the energy and the length of the curve, is uniformly bounded in time. As Abrams et al \cite{Abrams2003} have shown that 
$$
 E^\alpha (\g) \geq E^\alpha(\mathbb S^1) = m_\alpha >0
$$
for all closed curves $\g$ of unit length, we get from the scaling of the energy
$$
 E^\alpha ( \g) = L^{2-\alpha}E^\alpha (\frac \g {L(\g)}) \geq L^{2-\alpha}m_\alpha
$$
and thus
\begin{equation}
 L(\g_t) \geq \left(\frac {m_\alpha}{E^{\alpha}(\g_t)}\right)^{\frac 1 {\alpha -2}} \geq \left(\frac {m_\alpha}{E^{\alpha}(\g_0)}\right)^{\frac 1 {\alpha -2}} >0
\end{equation}
uniformly in $t$.

We will show that the energies $E^{\alpha}$ are coercive in $W^{\frac {\alpha +1 }2, 2}$ (cf. Theorem~\ref{thm:Coercivity}) in Section~\ref{sec:Coercivity}.
Together with the above inequalities this implies that the 
 $W^{\frac {\alpha +1 }2 , 2}$ norm of the unit tangents of the curve is uniformly bounded.
 
To get higher order estimates, we calculate the evolution equations of the terms
 \begin{equation*}
 	\mathcal E^k = \int_{\mathbb R / \mathbb Z} |\partial_s^k \kappa|^2 ds
 \end{equation*}
(cf. \ref{lem:EvolutionOfDkappa}) in  Section \ref{sec:EvolutionEquations} and show that the resulting terms
can be estimated using Gagliardo-Nirenberg-Sobolev inequalities for fractional Sobolev spaces and Besov spaces. 
 
 In the Subsections \ref{sec:ProofOfLongTimeExistence} we put all these pieces together
 to show that the solution to the flow exists for all time and subconverges after suitable translations and re-parameterizations if necessary to
 a critical point.

\subsection{Coercivity of the Energy} \label{sec:Coercivity}

Theorem 1.1 in \cite{Blatt2012} states that for curves
parameterized by arc length, the energy $E^\alpha$ is finite if and only if the curve 
is injective and belongs to 
$W^{\frac {\alpha+1} 2,2}(\mathbb R / \mathbb Z, \mathbb R^n)$. 
One of the most important ingredients in the proof of the long time existence result 
is the following quantitative version
of the regularizing effects  of Theorem 1.1 in \cite{Blatt2012}:

\begin{theorem} [Coercivity of $E^\alpha$ ] \label{thm:Coercivity}
Let $\g \in C^{1}(\mathbb R / l \mathbb Z , \mathbb R^n )$, $l >0$, be a curve parametrized 
by arc length and $\alpha \in [2,3)$. Then there exists
a constant $C=C(\alpha)< \infty$ depending only on $\alpha$ such that
\begin{equation*}
	|\g'|_{W^{\frac {\alpha-1}{2},2}} 
	\leq C E^\alpha (\g).
\end{equation*}
\end{theorem}

\begin{proof}
We have
\begin{align*}
	E^{\alpha} (\g) &= \int_{\mathbb R / l\mathbb Z} \int_{-\frac l 2} ^{\frac l 2}
	\bigg(
		 \frac 1  {|\g(u+w)- \g(u)|^\alpha} -  \frac 1 {|w|^\alpha} 
	\bigg)dw du 
	\\
	&= 
		\int_{\mathbb R / l \mathbb Z} \int_{\mathbb R} 
			\frac {|w|^\alpha}{|\g(u+w) - \g(u)|^\alpha} 
			\left( \frac {1- \frac {|\g(u+w)-\g(u)|^\alpha}{|w|^\alpha}}
			{|w|^\alpha} \right) 
		dw du
	\\
	& \lstackrel{|\g(u+w)- \g(u)|\geq |w|}{\geq}
		\int_{\mathbb R / l \mathbb Z} \int_{-\frac l 2} ^{\frac l 2}
			\left( \frac {1- \frac {|\g(u+w)-\g(u)|^\alpha}{|w|^\alpha}}
			{|w|^\alpha} \right) 
		dw du
	\\
	& \lstackrel{1- a ^\alpha \geq 1-a^2 \text { for } a \in [0,1]}{\geq}
	\int_{\mathbb R / l \mathbb Z} \int_{-\frac l 2} ^{\frac l 2}
			\left( \frac {1- \frac {|\g(u+w)-\g(u)|^2}{|w|^2}}
			{|w|^\alpha} \right) 
		dw du
	\displaybreak[0] \\
	& = \int_{\mathbb R / l \mathbb Z} \int_{-\frac l 2} ^{\frac l 2}
			\left( \frac {1-
				\int_{0}^1 \int_{0}^1 
					\left\langle \g'(u+\tau_1 w), \g'(u + \tau_2 w) \right\rangle
				d\tau_1 d\tau_2
			} {|w|^\alpha} \right) 
		dw du
	\displaybreak[0] \\
	& \lstackrel {|\g'| = 1} {=}
		\int_{\mathbb R / l \mathbb Z} \int_{-\frac l 2} ^{\frac l 2}
			\left( \frac{
				\int_{0}^1 \int_{0}^1 
					| \g'(u+\tau_1 w) -  \g'(u + \tau_2 w)| ^2 
				d\tau_1 d\tau_2
			} {|w|^\alpha} \right) 
		dw du
\end{align*}
Using Fubini and successively substituting $u$ by $u+\tau_1 w$ and then $w$ by 
$(\tau_2 - \tau_1)w$,
we get
\begin{align*}
	E^{\alpha} (\g) 
	&\geq 
		\int_{\mathbb R / l \mathbb Z} \int_{-\frac l 2} ^{\frac l 2}
		 \int_{0}^1 \int_{0}^1 
			\left( \frac{
					| \g'(u+\tau_1 w) -  \g'(u + \tau_2 w)| ^2 
			} {|w|^\alpha} \right) 
		d\tau_1 d\tau_2 dw du
	\\
	&\geq \int_{0}^1 \int_{0}^1 \Bigg(|\tau_2 - \tau_1|^{\alpha-1} 
		\int_{\mathbb R / l \mathbb Z}
		 \int_{-\frac{ |\tau_2 - \tau _1| l} 2} ^{\frac{\|\tau_2 - \tau_1| l} 2}
			\left( \frac{
					| \g'(u) -  \g'(u + w)| ^2 
			} {|w|^\alpha} \right) 
		dw du \Bigg) d\tau_1 d\tau_2
	 \\
	&\geq 
		\left(\frac 12 \right) ^{\alpha-1}
		\int_{0}^{1/4} \int_{3/4}^1 
		\int_{\mathbb R / l \mathbb Z}
		 \int_{-\frac l 4} ^{\frac  l 4}
			\left( \frac{
					| \g'(u) -  \g'(u + w)| ^2 
			} {|w|^\alpha} \right) 
		dw du d\tau_1 d\tau_2
	\\
	&\geq 
		\left(\frac 12 \right) ^{\alpha+5}
		\int_{\mathbb R / l \mathbb Z}
		 \int_{-\frac l 4} ^{\frac  l 4}
			\left( \frac{
					| \g'(u) -  \g'(u + w)| ^2 
			} {|w|^\alpha} \right) 
		dw du 
\end{align*}
and finally
\begin{align*}
		|\g'|^2_{W^{\frac {\alpha -1} 2 }} &=
		\int_{\mathbb R / l \mathbb Z}
	 	\int_{-\frac l 2} ^{\frac  l 2}
			\left( \frac{
					| \g'(u) -  \g'(u + w)| ^2 
			} {|w|^\alpha} \right) 
		dw du 
	 \\
	& \leq C
		\int_{\mathbb R / l \mathbb Z}
	 	\int_{-\frac l 2} ^{\frac  l 2}
			\left( \frac{
					| \g'(u) -  \g'(u + w/2)| ^2 
			} {|w|^\alpha} \right) 
		dw du 
		\\
		&\quad+ C  \int_{\mathbb R / l \mathbb Z}
	 	\int_{-\frac l 4} ^{\frac  l 4}
			\left( \frac{
					| \g'(u+w/2) -  \g'(u + w)| ^2 
			} {|w|^\alpha} \right) 
		dw du 
	\\
	& \leq
		C \int_{\mathbb R / l \mathbb Z}
	 	\int_{-\frac l 4} ^{\frac  l 4}
			\left( \frac{
					| \g'(u) -  \g'(u + w/2)| ^2 
			} {|w|^\alpha} \right) 
		dw du 
	\\
	& \leq
		 C E^\alpha (\g).
\end{align*}
\end{proof}

\subsection{Evolution Equations of Higher Order Energies} \label{sec:EvolutionEquations}

As for most of our estimates the precise algebraic form of the terms
does not matter, we will use the following notation to describe the 
essential structure of the terms.

For two Euclidean
vectors $v,w$, $v \ast w$ stands for a bilinear operator in $v$ and $w$
into another Euclidean vector space. For a regular curve $\g$, let $\partial_s = \frac {\partial_x}{|\g'|}$ denote the derivative with respect to arc length. For $\mu, \nu \in \mathbb N$,  a
regular curve $\g \in C^\infty (\mathbb R / \mathbb Z, \mathbb R^n)$
and a function $f: \mathbb R / \mathbb Z \rightarrow \mathbb R^k$ we
let $P^{\mu}_\nu(f)$ be a linear combination of terms of the form
$\partial_s^{j_1} f \ast \cdots \ast \partial_s^{j,\nu} f$, $j_1 +
\cdots + j_\nu =\mu$. Furthermore, given a second function $g: \mathbb R /
\mathbb Z \rightarrow \mathbb R^k$ the expression
$P^{\mu}_{\nu}(g,f)$ denotes a linear combination of terms of the
form $\partial_s^{j_1} g \ast \partial_s^{j_2} f \ast  \partial_s^{j_3} f \ast  \cdots \ast \partial_s^{j,\nu} f$, $j_1 +
\cdots + j_\nu =\mu$.

Let $\g_t$ be a smooth family of smooth closed curves moving only in
normal direction, i.e., $V:=\partial_t \g_t$ is normal along
$\g_t$. Furthermore, let us denote by $s$ the arc length
parameter. It is well known that
\begin{equation} \label{eq:CommutationOfDerivatives}
  \partial_t \partial_s = \partial_s \partial_t + \langle \kappa, V
  \rangle \partial_s
\end{equation} 
and
\begin{equation} \label{eq:EvolutionEquationOfKappa}
 \partial_t \kappa = \partial_s^2 V + \partial_s (\langle \kappa, V
 \rangle \tau ) + \langle \kappa, V \rangle \kappa.
\end{equation}

Using these equations, we inductively deduce the following evolution equations of arbitrary derivatives of the curvature.

\begin{lemma} \label{lem:EvolutionOfDkappa}Let $I\subset \mathbb R$ be open and $\g: I \times
  \mathbb R / l\mathbb Z \rightarrow \mathbb R$ be a smooth family of
  curves such that $V:= \partial_t \g$ is normal along $\g$,
  i.e. $\langle V(x,t), \g'(x,t)\rangle = 0$. Then
\begin{equation*}
\partial_t (\partial_s ^k \kappa) = \partial^{k+2} V + \partial_s \left(
P^{k}_2(V,\kappa)  \tau \right) + P_3^{k}(V,\kappa)
\end{equation*}
for all $k \in \mathbb N_0.$
\end{lemma}

\begin{proof}
By (\ref{eq:EvolutionEquationOfKappa}) the statement is true for
$k=0$. If the statement was true for some $k$ then 
\begin{align*}
  \partial_t (\partial_s^{k+1} \kappa) &= \partial_s (\partial_t
  (\partial_s^k \kappa)) + \langle \kappa, V\rangle \partial_s^{k+1}
  \kappa \\
 & = \partial_s \left( \partial^{k+2} V + \partial_s \left(
P^{k}_2(V,\kappa) \tau\right) + P_3^{k}(V,\kappa) \right) + \langle \kappa,
V\rangle \partial_s^{k+1}
  \kappa \\
& = \partial_s \left( \partial^{k+2} V + 
P^{k+1}_2(V,\kappa) \tau + P^{k+1}_3 (V,\kappa) + P_3^{k}(V,\kappa) \right) +
\langle \kappa, V\rangle \partial_s^{k+1}
  \kappa \\
 & =  \partial^{(k+1)+2} V + \partial_s \left(
P^{k+1}_2(V,\kappa) \tau \right) + P_3^{k+1}(V,\kappa). 
\end{align*} 
Hence, induction gives the assertion.
\end{proof}

An immediate corollary of Lemma~\ref{lem:EvolutionOfDkappa} and
\[
 \partial_t \left(|\g'| \right) 
 = -\left\langle \kappa, \partial_s V\right\rangle |\g'|\]
is the following.

\begin{corollary} \label{cor:EvolutionOfEnergies}
 Let $\g$ be a family of curves moving with normal speed
 $V$. Then
\begin{align*}
  \partial_t \int_{\mathbb R / \mathbb Z} |\partial_s^{k} \kappa|^2 ds
  &= 2 \int_{\mathbb R / \mathbb Z}
  \langle \partial_s^{k+2}V,\partial_s^k \kappa\rangle ds 
 + 2\int
  \langle P^{k}_2(V,\kappa)  \tau, \partial^{k+1}_s
\kappa\rangle ds 
\\
& \quad + 2 \int \langle P_3^{k}(V,\kappa), \partial_s^{k} \kappa
\rangle ds
- \int |\partial_s^k \kappa|^2 \langle \kappa ,  V\rangle ds.
\end{align*}
\end{corollary}

\subsection{Interpolation Estimates} \label{sec:Estimates}

In this section we will prove several estimates that will be needed later in the proof of the
long time existence result. In a natural way the Besov spaces $B^{s,p}_q$ appear during our calculations.

\begin{lemma} [Gagliardo-Nirenberg-Sobolev type estimates for a typical term] \label{lem:GagliardoNirenbergTypeEstimates} 
Let $j_1,  \ldots, j_{\nu+2} \in \mathbb N$, $j_1, \ldots , j_{\nu}\geq 2$, 
and $s > \frac 3 2$ be such that there are 
 $p_1, \ldots, p_{\nu + 2} \in [1, \infty]$ with 
\begin{equation*}
	\sum_{i=1}^{\nu + 2} \frac 1 {p_i} =1,
\end{equation*} 
$
 	\frac \alpha 2 \leq j_i - \frac 1 {p_i} \leq s +  \frac \alpha 2
$
 for $i=1, \ldots, \nu$ and 
$
 	\frac 1 2 j_{i} - \frac 1 {p_i} \leq s+ \frac 1 2
$ for $i=\nu+1, \nu+2$.  Let $\theta:= (\sum_{i=1}^{\nu+2} j_i - \nu \frac \alpha 2 -2 )/s$.

 Then for all $\Lambda < \infty$ there is a $C=C(\Lambda)$ such that
 \begin{multline*}
 	\int_{\mathbb R / l \mathbb Z} \int_{- \frac l2} ^{\frac l 2}
	 		\left(\prod_{i=1}^{\nu} |\partial ^{j_i} f (u + \sigma_i w)| \right)
	 		\\ \times   
	 		\left( \prod_{i=\nu+1}^{\nu+2}
	 		\frac {
	 			\int_0^1 \int_0^1 
	 			|\partial^{j_{i}}f(u+ \tau_1 w) - \partial^{j_{i}} f(u+\tau_2 w)|
				d\tau_1 d\tau_2} {|w|^\alpha}
 		\right) 
	dw du
	\\
	\leq C \|f\|_{W^{\frac {\alpha+1} 2 +s,2}} ^\theta
	\|f\|_{W^{\frac {\alpha+1} 2 ,2}} ^{\nu +2 - \theta}
 \end{multline*}
 holds for all $f \in C^{\infty}( \mathbb R / l \mathbb Z, \mathbb R^n)$ if 
 $\Lambda^{-1} \leq l \leq \Lambda$, $\sigma_i \in \mathbb R$.
 
\end{lemma}

\begin{proof}
Using H\"older's inequality for the integration with respect to $u$, we get 
\begin{align*}
		&\int_{\mathbb R / l \mathbb Z} \int_{- \frac l2} ^{\frac l 2}
	 		\left(\prod_{i=1}^{\nu} |\partial_s ^{j_i} f (u+\sigma_i w)| \right)    \\ & \quad \quad \quad \quad \quad \quad \quad \quad \quad 
	 		\left( \prod_{i=\nu+1}^{\nu+2}
	 		\frac {
	 			\int_0^1 \int_0^1 
	 			|\partial^{j_{i}}f(u+ \tau_1 w) - \partial^{j_{i}} f(u+\tau_2 w)|
				d\tau_1 d\tau_2} {|w|^\alpha}
 			\right) 
		dw du
	\\
	& \leq \left(\prod_{i=1}^{\nu} \|\partial^{j_i}f\|_{L^{p_i}} \right) 
		\int_0^1 \int_0^1 \int_{-\frac l2 }^{\frac l 2}  \frac{
			\prod_{i=\nu+1}^{\nu+2} \|\partial^{j_{i} }
			f (\cdot + (\tau_1 - \tau_2) w )-\partial^{j_{i} }
			f \|_{L^{p_{i}}} }{|w|^\alpha} 
		dw d\tau_1 d\tau_2 
	\\
	&\leq \left(\prod_{i=1}^{\nu} \|\partial^{j_i}f\|_{L^{p_i}} \right)
	\int_{0}^1 \int_0^1
		\prod_{i=\nu+1}^{\nu+2}
		 \left( \int_{-\frac l2 }^{\frac l 2}  
		 	\frac{\|\partial^{j_{i} }
			f (\cdot + (\tau_1 - \tau_2) w )-\partial^{j_{i} }
			f \|^2_{L^{p_{i}}} }{|w|^\alpha} 
		dw \right)^{\frac 1 2} d\tau_1 d\tau_2.
\end{align*}
Substituting $w$ by $(\tau_1 - \tau_2) w$ we can estimate this further by
\begin{align*}
	&\leq C \left(\prod_{i=1}^{\nu} \|\partial^{j_i}f\|_{L^{p_i}} \right)
	\int_{0}^1 \int_0^1 |\tau_1 - \tau_2|^{\alpha-1}
		\prod_{i=\nu+1}^{\nu+2}
		 \left( \int_{-\frac l2 }^{\frac l 2}  
		 	\frac{\|\partial^{j_{i} }
			f (\cdot + w )-\partial^{j_{i} }
			f \|^2_{L^{p_{i}}} }{|w|^\alpha} 
		dw \right)^{\frac 1 2} d\tau_1 d\tau_2
	\\
	&\leq C
		\left(\prod_{i=1}^{\nu} \|\partial^{j_i}f\|_{L^{p_i}} \right) 
		\|\partial^{j_{\nu+1}} f\|_{B^{\frac {\alpha-1} 2, p_{\nu +1}}_{2}}
		\|\partial^{j_{\nu+2}} f\|_{B^{\frac {\alpha-1} 2, p_{\nu +2}}_{2}}.
\end{align*}
Scaling the Gagliardo-Nirenberg-Sobolev estimates (Theorem \ref{eq:GagliardoNirenbergEstimates}),
the last term can be estimated from above by
\begin{align*}
	C \prod_{i=1}^{\nu+2} \|f\|^{\theta_i}_{W^{s+\frac{\alpha +1 }2,2}} \|f\|^{1-\theta_i}_{W^{\frac {\alpha +1} 2,2}}
\end{align*}
where $\theta_i := \frac {\left(j_i - \frac 1 {p_i}\right) - \frac \alpha 2} s $ for $i =1, \ldots, \nu$
and  $\theta_i := \frac {(j_i - \frac 1 {p_i}) - \frac 1 2} s $ for $i=\nu+1, \nu+2$. Thus the
assertion of the theorem follows.
\end{proof}

\begin{lemma} [Estimates for terms containing the energy integrand] \label{lem:EstimateOfTypicalTerm}
For all $\Lambda <\infty$ there is a $C(\Lambda) < \infty$ such that the following holds:

Let $\Lambda^{-1} \leq l \leq \Lambda$ and $\g \in C^{\infty} (\mathbb R / l\mathbb Z, \mathbb R^n)$ 
be a curve parameterized by arc length satisfying the bi-Lipschitz estimate 
 $$|w|
\leq \Lambda |\g(u+w) - \g(u)| \quad \forall u\in \mathbb R / l\mathbb Z,
 w \in [- l/2,  l/2], $$ 
 Then the functions
 \begin{align*}
 	g_\beta:\mathbb R / \mathbb Z &\rightarrow \mathbb R,&
 	\quad g_\beta(s)&:=
 	\int_{-\frac l 2}^{\frac l 2} 
 		\left(
 		\frac {|w|^\beta} {|\g(s+w) - \g(s)|^{\alpha+\beta}} - \frac {|w|^\beta} {|w|^{\alpha+\beta}}
 		 \right)
	dw
 \end{align*}
 are in $C^{\infty}(\mathbb R / l\mathbb Z , \mathbb R)$ for $\beta >0$. If furthermore $\mu \leq s + \frac \alpha 2  $  and $k +1  \leq s$, we have
 \begin{multline} \label{eq:GN}
 \left| 
 \int_{\mathbb R /l \mathbb Z} \partial_s^k 
 	\left\{
     		\int_{-\frac l2}^{\frac l2} 
     		\left(
     			\frac {|w|^\beta} {|\g(s+w) -\g(s)|^{\alpha+\beta} }
     			 - \frac {|w|^\beta} {|w|^{\alpha+\beta}}
		\right) 
		dw 
	\right\}  P^\mu_{\nu}
   	(\g' ) (s) ds 
\right| \\ \leq C \sum_{l=1}^{k}\
			\|\g\|^{\theta_l}_{W^{\frac {\alpha +1}2 + s,2}(\mathbb R /\mathbb Z , \mathbb R^n)}
			\|\g\|^{l+\nu+2-\theta_l}_{W^{\frac {\alpha+1}2,2}(\mathbb R / \mathbb Z, \mathbb R^n)}
 \end{multline}
 where $\theta_l := (k +(l+2) + \mu +  \nu - (l+\nu) \frac \alpha 2 - 2 )/s < (k+\mu)/s.$ 
 If $(k+\mu /s \leq 2$, this implies that for every $\varepsilon >0$ there is a constant $C(\varepsilon,\|\g\|_{W^{\frac {\alpha+1}2,2}(\mathbb R / \mathbb Z, \mathbb R^n)}) < \infty$
 such that
 \begin{multline}  \label{eq:GNInterpolation}
 \left| 
 \int_{\mathbb R / l\mathbb Z} \partial_s^k 
 	\left\{
     		\int_{-\frac l2}^{\frac l2} 
     		\left(
     			\frac {|w|^\beta} {|\g(s+w) -\g(s)|^{\alpha+\beta} }
     			 - \frac {|w|^\beta} {|w|^{\alpha+\beta}}
		\right) 
		dw 
	\right\}  P^\mu_{\nu}
   	(\g' ) (s) ds 
\right| \\ \leq \varepsilon 
			\|\g\|^{2}_{W^{\frac {\alpha +1}2 + s,2}(\mathbb R /\mathbb Z , \mathbb R^n)}
			+C(\varepsilon,\|\g\|_{W^{\frac {\alpha+1}2,2}(\mathbb R / \mathbb Z, \mathbb R^n)})
 \end{multline}
\end{lemma}

\begin{proof}
For $\frac l2 > \varepsilon >0$ and $I_{l, \varepsilon}= [-\frac l2 , \frac l2] \setminus [-\varepsilon,\varepsilon]$ we set  
\begin{align*}
g_{\beta}^{(\varepsilon)}(s) &:= \int_{w \in I_{l,\varepsilon}} h_\beta(s,w) dw \\
\intertext{and}
h_{\beta}(s,w) &:= \frac {|w|^\beta} {|\g(s+w) - \g(s)|^{\alpha+\beta}}
		- \frac {|w|^\beta}{|w|^{\alpha+\beta}}
\end{align*}
for all $s \in \mathbb R / l \mathbb Z$ and $w \in [-l/2, l/2]$, $w \not=0$. Then due to the bi-Lipschitz estimate
for $\g$ we have $g_{\beta} ^{(\varepsilon)} \in C^\infty (\mathbb R / \mathbb Z, \mathbb R^n)$ and
\begin{equation*}
	\partial_s^k g_\beta^{(\varepsilon)} (s) = \int_{w \in I_{l,\varepsilon}} \partial_w^k h_{\beta} (s,w) dw.
\end{equation*}

With 
$G_\beta(v) := \frac 1 {|v|^{\alpha+\beta}} 
			\frac {1-|v|^{\alpha+\beta}} {1-|v|^2}$ 
we get  
\begin{multline*}
		\frac {|w|^\beta} {|\g(s+w) - \g(s)|^{\alpha+\beta}}
		- \frac {|w|^\beta}{|w|^{\alpha+\beta}}
	 = 
		 G_\beta \left(\frac {\g(s+w) - \g(s)}{w} \right) \cdot 
		\frac {1-\frac {|\g(s+w) - \g(s)|^2 }{w^2 }}{|w|^{\alpha}}
	\\
	=
		\frac 1 2 G_\beta \left(\frac {\g(s+w) - \g(s)}{w} \right) \cdot 
		\int_0^1 \int_0^1
			\frac{|\g(s+\tau_1 w) - \g(s + \tau_2 w)|^2 }
			{|w|^{\alpha}}
		d\tau_1 d\tau_2.
\end{multline*}
Note that $G_{\beta}$ is a smooth function on $\mathbb R^n / \{0\}$. Since
for $s \in \mathbb R / \mathbb Z$, $w \in [-l/2, l/2] - \{0\}$ we have 
$1 \geq |\frac  {\g(s+w) - \g(s)}{w}| \geq \Lambda ^{-1}$,
and there exist constants $C(k)$ such that
 \begin{equation} \label{eq:EstimatesForTheAuxiliaryFunction}
 	\bigg|(D^k G) \Big(\frac {\g(s+w)- \g(s)} {w}\Big) \bigg| \leq C(k) \quad 
 	\forall s \in \mathbb R / \mathbb Z, w \in [-l/2, l/2]- \{0\}. 
 \end{equation}
 Using the product rule together with F\'aa di Bruno's formula for higher derivatives
 of composite functions, we conclude that 
\begin{equation*}
	\partial_s^{ k} h_\beta (s,w) = \partial_s^{ k} 
	\left (
		 \frac {|w^\beta|} {|\g(s+w) - \g(s)|^{\alpha + \beta}}
		-  \frac {|w^\beta|} {|w|^{\alpha + \beta}}
	\right)
\end{equation*}
is a linear combination of terms $T_{k_1, \ldots, k_{l+2}}$ of the form
\begin{multline} \label{eq:ATypicalTerm}
	D^{l } G \left( \frac {\g(s+w) - \g(s)} {w} \right)
	\left( \frac{ \partial_s^{k_1} \g(s+w) - \partial_s^{k_1}\g (s)}{w}.
		\ldots,
		\frac{ \partial_s^{k_l} \g(s+w) - \partial_s^{k_l}\g (s)}{w}
	\right)
	\\
	\times	
	\int_{0}^1 \int_{0}^1 
		\frac{
			\left\langle \partial_s^{k_{l+1}+1}\g(s+\tau_1 w) - \partial_s^{k_{l+1}+1} \g(s + \tau_2 w),
				 \partial_s^{k_{l+2}+1}\g(s+\tau_1 w) - \partial_s^{k_{l+2}+1} \g(s + \tau_2 w)
		\right\rangle }
			{|w|^{\alpha}}
	d\tau_1 d\tau_2
\end{multline}
where $l,k_1, \ldots k_{l+2} \in \mathbb N_0$,
$ k_1, \ldots k_l \geq 1$,
and
\begin{equation*}
	\sum_{i=1} ^{l+2} k_i = k.
\end{equation*}
Using \eqref{eq:EstimatesForTheAuxiliaryFunction} and the fundamental theorem of calculus, such terms can be estimated by
\begin{multline*}
	C(k) \prod_{i=1}^l \left( \int_{0}^1 \partial_s ^{k_i +1}\g(s+ \sigma_i w) d \sigma_i \right) 
	\\
	\times
	\int_0^1 \int_0^1 \frac{
			\left\langle \partial_s^{k_{l+1}+1}\g(s+\tau_1 w) - \partial_s^{k_{l+2}+1} \g(s + \tau_2 w),
				 \partial_s^{k_{l+2}+1}\g(s+\tau_1 w) - \partial_s^{k_{l+2}+1} \g(s + \tau_2 w)
		\right\rangle }
			{|w|^{\alpha}}
	d\tau_1 d\tau_2 
	\\
	\stackrel{\alpha <3}\leq \frac{C \|\g\|_{C^{k+2}}}{|w|^{\alpha-2}}.
\end{multline*}
 Hence,
\begin{equation*}
	|\partial_s^k h_{\beta}(s,w)| \leq \frac{C \|\g\|_{C^{k+2}}}{|w|^{\alpha-2}}.
\end{equation*}
 From this we deduce that for $\varepsilon_1 > \varepsilon_2 >0$ we have  
\begin{align*}
		|\partial^k g_{\beta}^{\varepsilon_1} (s) - \partial^k g_\beta^{(\varepsilon_2)} (s)| 
	&\leq  \int_{w\in [-\varepsilon_1, \varepsilon_1] \setminus [-\varepsilon_2, \varepsilon_2]}| \partial_s^i h(s,w) |dw
	\\ &\leq  C \|\g\|_{C^{k+2}} \int_{\varepsilon_2}^{\varepsilon_2}\frac 1 {|w|^{\alpha -2}} dw \leq  C \|\g\|_{C^{k+2}} \varepsilon_1^{3- \alpha}.
\end{align*}
As $\alpha < 3$, $g_\beta^{\varepsilon}$ converges smoothly to a smooth representative of $g_\beta$ and
\begin{equation*}
		\partial_s^k g_{\beta} (s) 
	= 
		\int_{-\frac l2}^{\frac l2} \partial_s^k h(s,w) dw
	=
		\int_{-\frac l2}^{\frac l2} \partial_s^k  
			\left (
				 \frac {|w^\beta|} {|\g(s+w) - \g(s)|^{\alpha + \beta}}
				-  \frac {|w^\beta|} {|w|^{\alpha + \beta}}
			\right)
		dw.
\end{equation*}
Using that $\partial_s^k h(s,w)$ is a linear combination of terms like \eqref{eq:ATypicalTerm} 
together with Lemma \ref{lem:GagliardoNirenbergTypeEstimates}, we obtain
\begin{equation*}
	 	\left| 
 			\int_{\mathbb R / l \mathbb Z} \left( \left(\partial_s^k 
 				g_\beta (s)\right) P^\mu_{\nu}
   				(\kappa ) (s) \right)
			 ds 
		\right| 
	\leq 
		C \sum_{l=1}^{k}\
			\|\g\|^{\theta_l}_{W^{\frac {\alpha +1}2 + s,2}(\mathbb R /\mathbb Z , \mathbb R^n)} 
			\|\g\|^{l+\nu+2-\theta_l}_{W^{\frac {\alpha+1}2,2}(\mathbb R / \mathbb Z, \mathbb R^n)}
 \end{equation*}
 where $\theta_l := (k +(l+2) + \mu +  \nu - (l+\nu) \frac \alpha 2 - 2 )/s < (k +l+ \mu +  \nu - l-\nu  )/s\leq (k+\mu)/s.$ 
 This proves inequality \eqref{eq:GN} from which one obtains \eqref{eq:GNInterpolation} using Cauchy-Schwartz.
\end{proof}

\subsection{Proof of Long Time Existence} \label{sec:ProofOfLongTimeExistenceConstantLambda}
\label{sec:ProofOfLongTimeExistence}
First we will derive the following estimate from the evolution equation of the higher order energies Corollary~\ref{cor:EvolutionOfEnergies} and Lemma~\ref{lem:EstimateOfTypicalTerm}. For a periodic function $f \in C^\infty (\mathbb R / l\mathbb Z, \mathbb R^n)$ we use the shorthand
$$
  D^{s} f = (-\Delta)^{\frac s2} f 
$$
for the fractional Laplacian.

\begin{lemma} \label{lem:EstimateEvolutionEnergies}
For every $k\in \mathbb N$  and $\varepsilon>0$ there are constants $C_\varepsilon < \infty$ such that 
\begin{equation*}
	\partial_t \int_{\mathbb R/ \mathbb Z} | \partial_s^k \kappa_{\g_t}(s) |^2 ds 
	+ c_\alpha \int_{\mathbb R / \mathbb Z} |(\tilde D_s)^{k+2+\frac{\alpha+1}2} \kappa| ds 
	\leq \varepsilon \int_{\mathbb R / \mathbb Z} |(\tilde D_s)^{{\alpha+1}/2} \kappa|^2 ds
	+  C_\varepsilon.
\end{equation*}
\end{lemma}

\begin{proof}
Corollary~\ref{cor:EvolutionOfEnergies} tells us
that
\begin{multline} \label{eq:EvolutionOfEnergies}
  \partial_t \int_{\mathbb R / \mathbb Z} |\partial_s^{k} \kappa|^2 ds
  = 2 \int_{\mathbb R / \mathbb Z}
  \langle \partial_s^{k+2}V,\partial_s^k \kappa\rangle ds 
 + 2\int
  \langle P^{k}_2(V,\kappa)  \tau, \partial^{k+1}_s
\kappa\rangle ds 
\\
 \quad + 2 \int \langle P_3^{k}(V,\kappa), \partial_s^{k} \kappa
\rangle ds
- \int |\partial_s^k \kappa|^2 \langle \kappa ,  V\rangle ds
\end{multline}
where $V= -H^\alpha \g + \lambda \kappa $.

Let us now fix the time $t$ and let us re-parameterize $\g_t$ for this
fixed time  by arc length to estimate the right-hand side of this equation and let $l$ denote the length of the curve at time $t$.

We decompose
\begin{equation} \label{eq:DecompositionOfTildeH}
		\tilde H ^\alpha \g 
	= 
		\alpha Q^\alpha \g	 + 	(\alpha-2) R_1^\alpha\g 	+	 2\alpha R_2^\alpha \g
\end{equation} 
where
\begin{align*}
(Q^\alpha\g) (s) &:= p.v. \int_{-\frac l2}^{\frac l2}
\bigg\{2 \frac {\g(s+w) - \g(s) - w
  \g'(s)}{|w|^{2}} - \g''(s)) \bigg\} \frac{dw}{|w|^\alpha}, \\
(R^\alpha_1\g)(s) & := p.v. \int_{-\frac l2}^{\frac l2} \g''(s)
\bigg(\frac 1 {|\g(s+w) -\g(s)|^\alpha }- \frac 1
  {|w|^\alpha}\bigg) dw ,\\
(R^\alpha_2\g)(s) & := p.v. \int_{-\frac l2}^{\frac l2}
\left( \g(s+w)-\g(s) - w
  \g'(s) \right)
\bigg(\frac 1 {|\g(s+w) -\g(s)|^{\alpha + 2 }}- \frac 1
  {|w|^{\alpha+2}}\bigg) dw.
\end{align*}
Using the fundamental theorem of calculus, we rewrite $R^\alpha_2 \g (s) $ as
\begin{equation*}
	(R^\alpha_2\g)(s)  := \frac 1 2 \lim_{\varepsilon \downarrow 0} \int_{0}^1
\int_{I_{l,\varepsilon}}  \g''(s+ \tau_1 w) w
\bigg(\frac {w^2} {|\g(s+w) -\g(s)|^{\alpha + 2 }}- \frac {w^2}
  {|w|^{\alpha+2}}\bigg) dw.
\end{equation*}
and set
\begin{equation*}
	R^\alpha \g := (\alpha-2)R^{\alpha}_1 \g + 2 \alpha R^{\alpha}_2 \g.
\end{equation*}
Hence,
\begin{equation*}
	V = -P^\bot_{\g'} (\alpha Q^\alpha \g + R^\alpha \g) + \lambda \kappa.
\end{equation*}

Now we estimate all the terms appearing in \eqref{eq:EvolutionOfEnergies} except for the term 
$$ \int_{\mathbb R / \mathbb Z}
  \langle \partial_s^{k+2}Q^\alpha \g,\partial_s^k \kappa\rangle d,s
$$
using H\"older's inequality together with the standard Gagliardo-Nirenberg-Sobolev inequality or the version in Lemma~\ref{lem:GagliardoNirenbergTypeEstimates}. We get that for all $\varepsilon>0$ there is a constant $C_\varepsilon < \infty$ such that all these terms can be estimated from above by 
$$
 \varepsilon \|D^{k+ 2 + \frac {\alpha +1 } 2} \g \|^2_{L^2}
	+ C(\varepsilon).
$$
We will only give the details for some exemplary term as the estimates for all the other terms can be estimated following  exactly the same line of arguments.
We start by estimating the terms containing the remainder $R^{\alpha}$. For the first term
in Equation \eqref{eq:EvolutionOfEnergies} we get using Lemma~\ref{lem:GagliardoNirenbergTypeEstimates}
\begin{align*}
	\bigg|
	\int_{\mathbb R / l \mathbb Z}
		 \left\langle \partial_s^{k+2} (P^{\bot}_{\g'} R^{\alpha}(\g)),
		 	\partial_s^k \kappa \right\rangle 
	ds
	\bigg|
	&=
	\bigg|
	\int_{\mathbb R / l \mathbb Z}
		 \left\langle \partial_s^{k+1} (P^{\bot}_{\g'} R^{\alpha}(\g)),
		 	\partial_s^{k+1} \kappa \right\rangle 
	ds
	\bigg|
\\
	&\leq 
	\bigg|
	\int_{\mathbb R / l \mathbb Z}
		 \left\langle \partial_s^{k+1} ( R^{\alpha}(\g) ),
		 	\partial_s^{k+1} \kappa \right\rangle 
	ds 
	\bigg|
\\
	& \quad +
	\bigg|
	\int_{\mathbb R / l \mathbb Z}
		 \left\langle \partial_s^{k+1} \left( 
		 	\left \langle R^{\alpha}(\g) , \g' \right\rangle  \g'\right),
		 	\partial_s^{k+1} \kappa \right\rangle 
	ds  
	\bigg|
\\
	& \leq 
	\varepsilon \|D^{k+ 2 + \frac {\alpha +1 } 2} \g \|^2_{L^2}
	+ C(\varepsilon)
\end{align*} 
since $(2k+2+2)/(k+2) =2$.
Similarly we get
\begin{align*}
\bigg|
	\int_{\mathbb R / l \mathbb Z}
		\langle P^{k}_2(P^\bot_{\g'}R^\alpha \g,\kappa)  \tau, \partial^{k+1}_s
		\kappa\rangle
	ds
	\bigg|
	&\leq 
	\bigg|
	\int_{\mathbb R / l \mathbb Z}
		\langle P^{k}_2(R^\alpha \g,\kappa)  \tau, \partial^{k+1}_s
		\kappa\rangle
	ds 
	\bigg|
\\
	& \quad +
	\bigg|
	\int_{\mathbb R / l \mathbb Z}
		 \langle
		 	 P^{k}_2(\left\langle R^\alpha \g, \g'\right\rangle \g'
		 	 ,\kappa)  \g',
		 	  \partial^{k+1}_s\kappa
		\rangle
	ds  
	\bigg|
\\
	& \leq 
	\varepsilon \|D^{k+ 2 + \frac {\alpha +1 } 2 } \g \|^2_{L^2}
	+ C(\varepsilon)
\end{align*}
since $(k+1+1+k+2)/(k+2)=2$. Along the same lines we get
\begin{align*}
\bigg|
	\int_{\mathbb R / l \mathbb Z}
		\langle P^{k}_3(P^\bot_{\g'}R^\alpha \g,\partial^k_s\kappa) \rangle
	ds
	\bigg|+ \bigg| \int_{\mathbb R / l \mathbb Z}|\partial_s^k \kappa |^2  \langle \kappa, P^\bot R^\alpha \g \rangle \bigg| 
  \leq 
	\varepsilon \|D^{k+ 2 + \frac {\alpha +1 } 2 } \g \|^2_{L^2}
	+ C(\varepsilon)
\end{align*}

To estimate the terms containing $Q^\alpha$, we will use  the 
fact that $Q^\alpha = c_\alpha D^{\alpha+1} + \tilde R$ where $\tilde R$ is a bounded operator 
from $W^{s+2,p}$ to $W^{s,p}$. For $k_1,k_2 \in \mathbb N_0$ with $k_1 + k_2 =k$ we estimate
using H\"older's inequality and the Gagliardo-Nirenberg-Sobolev inequality
\begin{align*}
 \bigg|
	\int_{\mathbb R / l \mathbb Z} \langle 
		\partial_s^{k_1} Q^\alpha \g \ast \partial_s^{k_2}\kappa)  \tau, \partial^{k+1}_s
		\kappa\rangle
	ds
	\bigg| &\leq \|Q^\alpha \partial_s^{k_1} \g\|_{L^2} \|\partial_s^{k_2}\|_{L^4} \|\partial^{k+1}_s \kappa\|_{L^4} 
	\\ &
	\leq \| \partial_s^{k_1} \g\|_{W^{\frac {\alpha +1 }2,2}} \|\partial_s^{k_2}\|_{L^4} \|\partial^{k+1}_s \kappa\|_{L^4} \\
	&\leq \varepsilon \|D^{k+ 2 + \frac {\alpha +1 } 2 } \g \|^2_{L^2}
	+ C(\varepsilon).
\end{align*}
Similarly we obtain, using the Leibniz rule,
\begin{align*}
 \bigg|
	\int_{\mathbb R / l \mathbb Z} \langle 
		\partial_s^{k_1} P^T Q^\alpha \g \ast \partial_s^{k_2}\kappa)  \tau, \partial^{k+1}_s
		\kappa\rangle
	ds
	&\leq \varepsilon \|D^{k+ 2 + \frac {\alpha +1 } 2 } \g \|^2_{L^2}
	+ C(\varepsilon).
\end{align*}
Hence,
\begin{align*}
  \bigg|
	\int_{\mathbb R / l \mathbb Z} \langle P^k_2 ( 
		\partial_s^{k_1} P^T Q^\alpha \g , \partial_s^{k_2}\kappa)  \tau, \partial^{k+1}_s
		\kappa\rangle
	ds
	&\leq \varepsilon \|D^{k+ 2 + \frac {\alpha +1 } 2 } \g \|^2_{L^2}
	+ C(\varepsilon).
\end{align*}
Similarly,  one gets
\begin{align*}
 \bigg|
	\int_{\mathbb R / l \mathbb Z}
		\langle P^{k}_3(P^\bot_{\g'}Q^\alpha \g,\partial^k_s\kappa) \rangle
	ds
	\bigg|+ \bigg| \int_{\mathbb R / l \mathbb Z}|\partial_s^k \kappa |^2  \langle \kappa, P^\bot Q^\alpha \g \rangle \bigg| 
  \leq 
	\varepsilon \|D^{k+ 2 + \frac {\alpha +1 } 2 } \g \|^2_{L^2}
	+ C(\varepsilon).
\end{align*}
For the terms containing $\lambda \kappa$ we use H\"older's inequality and standard Gagliardo-Nirenberg-Sobolev estimates together with Cauchy's inequality as above, to estimate these terms by
\begin{align*}
 \varepsilon \|D^{k+ 2 + \frac {\alpha +1 } 2 } \g \|^2_{L^2}
	+ C(\varepsilon).
\end{align*}

Let us finally turn to the term
$$ \int_{\mathbb R / \mathbb Z}
  \langle \partial_s^{k+2}Q^\alpha \g,\partial_s^k \kappa\rangle ds = \int_{\mathbb R / \mathbb Z}
  \langle \partial_s^{k+2}D^{\alpha+1} \g,\partial_s^k \kappa\rangle ds + \int_{\mathbb R / \mathbb Z}
  \langle \partial_s^{k+2}\tilde R \g,\partial_s^k \kappa\rangle ds
$$
that contains the highest order part. Again H\"older's inequality together with interpolation estimates yields
$$
 \int_{\mathbb R / \mathbb Z}
  \langle \partial_s^{k+2}\tilde R \g,\partial_s^k \kappa\rangle ds \leq  \varepsilon \|D^{k+ 2 + \frac {\alpha +1 } 2 } \g \|^2_{L^2}
	+ C(\varepsilon).
$$
Furthermore,
\begin{align*}
	\int_{\mathbb R / l \mathbb Z} 
		\left\langle
			\partial_s^{k+2} (P_{\g'}^\bot D^{\alpha+1}\g), 
			\partial_s^k \kappa
		\right\rangle
	ds
&=
	\int_{\mathbb R / l \mathbb Z} 
		\left\langle
			\partial_s^{k+2} (D^{\alpha+1}\g),
			\partial_s^k \kappa
		\right\rangle
	ds 
\\
&\quad -
	\int_{\mathbb R / l \mathbb Z} 
		\left\langle
			\partial_s^{k+2} (\left\langle D^{\alpha+1}\g,  \g'\right\rangle \g')
			, \partial_s^k \kappa 
		\right \rangle
	ds 
\\
& = 
	\int_{\mathbb R / l \mathbb Z} 
		|D^{k+2 + \frac{\alpha+1 }2} \g|^2 
	ds
\\
&\quad -
	\int_{\mathbb R / \mathbb Z}	 
		\left\langle
			D^{k+\frac { 3- \alpha} 2 } (\left\langle D^{\alpha-1}\g'',  \g'\right\rangle \g')
			, D^{k +\frac {\alpha +1}2}\kappa 
		\right \rangle
	ds 
\end{align*}
Using the commutator estimate (Theorem~\ref{thm:CommutatorEstimate}), we get that
\begin{align*}
	\|\left\langle D^{\alpha-1} \g'', \g' \right\rangle\|_{W^{k+\frac{3-\alpha} 2,2}}
	&= \|
		D^{\alpha-1}\left\langle\g'', \g' \right\rangle
		-\left\langle D^{\alpha-1} \g'', \g' \right\rangle
		\|_{W^{k+\frac{3-\alpha} 2,2}} \\
	&\leq 
		C ( \|\g'\|_{W^{1,\infty}} \|\g''\|_{W^{k-1+{\alpha+1}/2}}
		+ \|\g'\|_{W^{k+\frac {\alpha+1}/2},2} \|\g''\|_{L^\infty})
	 \\ &\leq \varepsilon \|D^{k+2 \frac {\alpha+1} 2} \g\|_{L^2 } 
	 	+ C(\varepsilon)
\end{align*}
and hence
\begin{align*}
 \int_{\mathbb R / \mathbb Z}
  \langle \partial_s^{k+2}Q^\alpha \g,\partial_s^k \kappa\rangle ds \leq - c_\alpha \int_{\mathbb R / l \mathbb Z} 
		|D^{k+2 + \frac{\alpha+1 }2} \g|^2 
	ds -  \varepsilon \|D^{k+2\frac {\alpha+1} 2} \g\|^2_{L^2 } 
	 	+ C(\varepsilon) 
\end{align*}

Summing up these estimates proves Lemma~\ref{lem:EstimateEvolutionEnergies}.
\end{proof}

A standard argument now concludes the proof of Theorem~\ref{thm:LongTimeExistence}:

\begin{proof}[Proof Theorem~\ref{thm:LongTimeExistence}]
 Let us assume that $[0,T)$ is the maximal interval of existence and $T < \infty$. 
 If we apply Lemma~\ref{lem:EstimateEvolutionEnergies} with $\varepsilon = \frac {c_\alpha} 2$ we get
 $$
  \frac d {dt} \mathcal E^k + \frac{c_\alpha}2 \|D^{k+2+\frac {\alpha +1} 2} \g\|_{L^2}^2 \leq C_k.
 $$
 Together with the Poincare inequality
 \begin{equation*}
  \mathcal E^k \leq C (\|D^{k+2+\frac {\alpha +1} 2} \g\|_{L^2}^2  + \|D^{\frac {\alpha +1} 2} \g\|_{L^2}^2),
 \end{equation*}
 this yields
 $$
  \frac d {dt} \mathcal E^k + \mathcal E ^k \leq C_k.
 $$
 Hence,
 \begin{equation*}
 	\sup_{t\in [0,T)}\|\partial_s^k \kappa \|_{L^2} < \infty
\end{equation*} 
for all $k \in \mathbb N$. Furthermore,
\begin{equation*}
	\partial_t |\g'| = \left\langle V, \kappa \right\rangle |\g'|
\end{equation*}
and hence there is a constant $C>0$ such that
\begin{equation*}
	C^{-1 } \leq |\g'| \leq C. 
\end{equation*}

Thus, there is a subsequence $t_i\rightarrow T$ such that 
$\g_{t_i}$ converges smoothly to a $\tilde \g(T) $ and we can use the
short time existence result Theorem~\ref{thm:ShortTimeExistence} to extend the flow beyond $T$.

Let us finally prove the subconvergence to a critical point. 
From Lemma~\ref{lem:EstimateEvolutionEnergies} we get
\begin{equation*}
	\partial_ t \|\partial_s^k \kappa\|_{L^2} + c \|\partial_s^k \kappa\| \leq C (k).
\end{equation*}
But this implies 
\begin{equation*}
	\|\partial_s^k \kappa\|_{L^2 } \leq C(k) \quad \forall k \in \mathbb N.
\end{equation*}
Hence, there is a subsequence $t_i \rightarrow \infty$ that $\g_{t_i}$, after 
re-parameterization by arc length and suitable translations, converges to a smooth curve $\g_\infty$ 
parameterized by arc length.  Since
$$
 \int_{0}^\infty \Bigg( \int_{\mathbb R / \mathbb Z} |\partial_t c_t|^2 ds \Bigg) dt \leq E(\g_0) < \infty,
$$
we can furthermore choose this subsequence such that $c_\infty$ is a critical point of $E=E^\alpha + \lambda L.$
\end{proof}

\section{Asymptotics of the Flows} \label{sec:ExponentialConvergence}

\subsection{{\L}ojasiewicz-Simon Gradient Estimate}

In order to prove convergence to critical points of the complete flow without taking care of translations, we will prove a {\-L}ojawiewicz-Simon Gradient estimate.

\begin{theorem} [{\L}ojasiewicz-Simon Gradient Estimate] \label{thm:LojasiewichSimonGradientEstimate}
Let $\g_M$ be a smooth critical point of $E:=E^{\alpha} + \lambda L$ for some $\alpha \in (2,3)$. Then there
 are constants
$\theta\in[0,1/2]$, $\sigma,c>0$, such that every $\g\in 
H_{i,r}^{\alpha+1}(\mathbb{R}/\mathbb{Z},\mathbb{R}^{n})$
with $\|\g-\g_{M}\|_{H^{\alpha+1}}\leq\sigma$ satisfies \[
|E(\g)-E(\g_M)|^{1-\theta}\leq c\cdot \bigg(\int_{\mathbb R /
  \mathbb Z}|(V\g)(x)|^2 |\g'(x)|dx\bigg)^{1/2},\]
where $Vc=-H^\alpha \g + \lambda \kappa_c$.

\end{theorem}

\begin{proof}

After scaling the curve we can assume that $\g_M$ is parameterized by arc length and that the length of the curve $\g_M$ is $1$.

Let $H^{\alpha+1}(\mathbb R/\mathbb Z, \mathbb R^n)^\bot_{\g_M}$ denote the space of all
vector fields $N\in H^{\alpha+1}(\mathbb R / \mathbb Z , \mathbb R^n)$ which are
orthogonal to $\g'_M$.

We will first prove that there are constants $\theta\in[0,1/2]$, $\tilde \sigma,\tilde c>0$  such that 
\begin{equation} \label{eq:LojasiewiczNormalGraph}
|E(\g_{M}+N)-E(\g_{M})|^{1-\theta}\leq c\cdot\left(
  \int_{\mathbb R / \mathbb Z} |V(\g_M+N)|^2 dx
\right)^{1/2} 
\end{equation}
for all $\phi\in H^{\alpha+1}(\mathbb{R}/\mathbb{Z},\mathbb{R}^{n})^{\bot }$
with $\|\phi\|_{H^{\alpha+1}}\le\tilde{\sigma}$. That is, we show that the functional
\begin{gather*}
  \tilde E: H^{\alpha+1}(\mathbb R / \mathbb Z, \mathbb R^n)_{\g_M}^\bot \rightarrow \mathbb
  R \\
N \mapsto E^\alpha(\g_M+N) + \lambda L(\g_M + N)
\end{gather*}
satisfies a {\L}ojasiewicz-Simon gradient estimate. 
By \cite[Corollary~3.11]{Chill2003} is suffices to show that $E'$ is analytic with values in $L^2$ and that $E''$ is a Fredholm operator from $H^{\alpha + 1}(\mathbb{R}/\mathbb{Z},\mathbb{R}^{n})^{\bot}_{\g_M}$ to
$L^2(\mathbb{R}/\mathbb{Z},\mathbb{R}^{n})^{\bot}_{\g_M}$.

It is easy to see that $\kappa$ defines an analytic operator on a neighborhood of $0$ from $(H^{\alpha+1})^\bot_{\g_M}$ to
$(L^2)^\bot_{\g_M}$ using the fact that $H^{\alpha+1}$ is embedded in $C^2$. That the same is true for
$\mathcal H^{\alpha}$ can be seen from
Lemma~\ref{thm:QuasilinearStructure},  using the fact that $H^{\alpha+1}(\mathbb R
/\mathbb Z, \mathbb R^n)$ is embedded in $C^{\alpha + \beta}$ for every
$0 < \beta < 1/2$ and $C^{\beta}$ is embedded in $L^2$.

We calculate the second variation of $\tilde E$ at $0$ and try to write it as a compact perturbation of a Fredholm operator of index $0$.  Using $|\g'_M|=1$ and the fact that $V$ is the gradient of $E^\alpha + \lambda L$, we get
\begin{multline} \label{eq:SecondVariation}
  \tilde E''(0)(h_1,h_2) \\= \lim_{t\rightarrow 0} \frac {\int_{\mathbb R
      / \mathbb Z} \left<V(\g_M+ t h_1),h_2\right> |\g_M'+ t h_1'|dw - \int_{\mathbb R
      / \mathbb Z} \left<V(\g_M),h_2\right> |\g_M'| dw}{t} \\
 =\left\langle \nabla_{h_{1}}V\g_M,h_{2}\right\rangle _{L^{2}}+\left\langle L_{1}h_{1},h_{2}\right\rangle _{L^{2}}
\end{multline}
where $L_{1}h_1=H\g_M\cdot\left\langle \g_M',h_{1}'\right\rangle$
is a differential operator of order $1$ in $h_1$. 

We know from Theorem~\ref{thm:QuasilinearStructure}
that \[
V\g=H^\alpha \g  + \lambda \kappa_{\g} = \frac{2}{|\g'|^{3}}P_{\g_M'}^\bot(Q^{\alpha}\g_M)+F(\g_M)\]
where $F\in C^{\omega}(C^{\alpha+\beta},C^{\beta})$ for all $\beta>0.$
Thus\begin{equation}\label{eq:FormOfNablaH}
P_{\g'_M}^\bot  (\nabla_{h}H(\g'_M))=\frac{2}{|\g'|^{3}}(P_{\g_M'}^{\bot}(Q^{\alpha}h)+L_2(h))\end{equation}
where \begin{align*}
L_2(h) &=-\frac{6}{|\g'_M|^{5}}{\left\langle \g_M',h'\right\rangle
} P_{\g'_M}^\bot(Q^{\alpha} \g_M) + P_{\g'_M}^\bot \left( \nabla_{h}F(\g_M)
+ (\nabla_h P^\bot_{\g_M '})(Q^{\alpha}\g_M)\right) \\
&\in C^{\omega}(C^{\alpha+\beta},C^{\beta}) \quad \forall \beta>0.\end{align*}
Now let $\nu_{i}$, $i=1, 2, \ldots ,(n-1)$, be smooth functions such that $\nu_{1}(u),\ldots,\nu_{n-1}(u)$
is an orthonormal basis of the normal space on $\g$ at $u$.
Then each $\phi\in H^{\alpha + 1}(\mathbb{R}/\mathbb{Z}, \mathbb R^n)^{\bot}$ can be written
in the form \[
\phi= \sum_{i=1}^{n-1}\phi_{i}\nu_{i},\]
where $\phi_{i}:=\left\langle \phi,\nu_{i}\right\rangle \in H^{\alpha +1 }(\mathbb{R}/\mathbb{Z})$.
We calculate 
\begin{equation} \label{eq:FormOfPQ}
\begin{aligned}
P_{\g_M'}^{\bot}\left(Q^{\alpha}\phi\right) &
=  P_{\g'_M}^{\bot}\left(Q^{\alpha} \left( \sum_{i=1}^{n-1} \phi_{i}\nu_{i} \right) \right)= \sum_{i=1}^{n-1}\left( Q^{\alpha}\phi_{i} \right) \nu_{i} +P_{\g'_M}^{\bot}
(Q^{\alpha}(\phi_{i}\nu_{i})-(Q^{\alpha}\phi_{i})\nu_{i})\\
 & =\sum_{i=1}^{n-1} \left( Q^{\alpha}\phi_{i} \right) \nu_{i}+L_{3}(\phi)
\end{aligned}
\end{equation}
where $L_{3}\in C^{\omega}(C^{\alpha+\beta},C^{\beta})$ by the fractional Leibniz rule
Lemma~\ref{lem:LeibnizRule}. 

From \cite[Proposition 2.3]{Reiter2012} we know that there is a constant $a^{(\alpha)} >0$ such that
that $L_4:=Q^{\alpha}-a^{\alpha} (-\Delta)^{{(\alpha +1)}/2}$ is a bounded linear operator from
$H^2 $ to $L^2$.
Combining \eqref{eq:SecondVariation}, \eqref{eq:FormOfNablaH}, and \eqref{eq:FormOfPQ}, we get
\begin{equation} \label{eq:CompactPerturbation}
 E''(h_1, h_2) = \left\langle \sum_{i=1}^{n-1} a^{\alpha}  \left((-\Delta^{\frac {\alpha +1 }2 })  \left \langle h_1,\nu_i\right \rangle  \right) \nu_i+ L h_1, h_2\right\rangle
\end{equation}
where
\begin{equation*}
 L := L_1 + L_2 + L_3 + L_4 
\end{equation*}
is a bounded operator from $C^{\alpha + \varepsilon}$ to $L^2$ for all $\varepsilon >0$. 

Since the linear mapping
$$
 h_1 \rightarrow \begin{pmatrix}
                  \left\langle h_1, \nu_1\right\rangle  \\
                  \vdots \\
                  \left\langle h_1, \nu_{n-1}\right\rangle
                 \end{pmatrix}
$$
defines an homeomorphism between $H^{s} (\mathbb R / \mathbb Z, \mathbb R^n)^\bot$ and $H^{s} (\mathbb R / \mathbb Z, \mathbb R^{n-1})$ for all $s \geq 0$ 
and $(-\Delta)^{(\alpha +1)/2}$ is a Fredholm operator of
index zero from $H^{\alpha +1}(\mathbb R / \mathbb Z, \mathbb
R^{n-1})$ to $L^2(\mathbb R/ \mathbb Z, \mathbb
R^{n-1})$, the operator
\begin{align*}
 A:  H^{\alpha +1 }(\mathbb R / \mathbb Z, \mathbb R^n)^\bot_{\g_M}
 &\rightarrow L^2 (\mathbb R / \mathbb Z, \mathbb R^n)^\bot_{\g_M} \\
 \phi &\mapsto \sum_{i=1}^{n-1}\left((-\Delta)^{\frac{\alpha+1}2} \left \langle \phi, \nu_i\right \rangle \right) \nu_i
\end{align*}
is Fredholm of order $0$.
Hence, 
$\sum_{i=1}^{n-1} a^{\alpha} (-\Delta^{\frac {\alpha +1 }2 })h_1 + L h_1$ as a compact perturbation of $A$  is a Fredholm operator as well.
This implies that $E''$
is a Fredholm operator from $H^{\alpha + 1}(\mathbb R / \mathbb Z, \mathbb
R^n)^\bot _{\g_M}$ to $L^2(\mathbb R/ \mathbb Z, \mathbb
R^n)^\bot_{\g_M}$ of index $0$. 
The proof of \eqref{eq:LojasiewiczNormalGraph} is complete.

To prove the full estimate of Theorem~\ref{thm:LojasiewichSimonGradientEstimate}, we use Lemma~\ref{lem:TubularNeighborhoodAndNormalGraphs} to
 write curves close to $c_M$ as normal graphs over $c_M$.  More precisely, 
we can choose $0 < \sigma< \tilde \sigma$ such that for all $\g\in H^{\alpha +1}(\mathbb{R}/\mathbb{Z})$
with $\|\g-\g_{M}\|_{H^{\alpha+1}}\leq\sigma$ there is a re-parameterization
$\psi\in H^{\alpha+1}(\mathbb{R}/\mathbb{Z},\mathbb{R}/\mathbb{Z})$ and
a $\phi\in H^{\alpha+1}(\mathbb{R}/\mathbb{Z},\mathbb{R}^{n})^{\bot}$ such
that\[
\g\circ\psi=\g_{M}+N_\g\]
and \[
\|N_c\|_{H^{\alpha+1}}\leq C\cdot\|\g-\g_{M}\|_{H^{\alpha+1}}.\]
Making $\sigma >0$ smaller if necessary, using that $\g_M$ is parameterized by arc length and that $H^{\alpha +1 }$ is embedded continuously in $C^1$, we furthermore can achieve that
\begin{equation*}
 |\g'| > |\g'_M| - |\g'_M - \g| \geq \frac 1 2.
\end{equation*}
Thus,
\begin{align*}
\left(E(\g)-E(\g_{M})\right)^{1-\theta}&=\left(E(\g_{M}+N_{\g})-E(\g_{M})\right)^{1-\theta}
\\
 & \leq
c  \left(\int_{\mathbb R / \mathbb Z} |V(\g_M+N_{\g})(x)|^2 dx
\right)^{1/2}
\\
&\leq c \sqrt 2 \left( \int_{\mathbb R / \mathbb Z}
|V(\g_M+N_{\g})(x)|^2 |\g'(x)| dx \right)^{1/2}.
\end{align*}
%

\end{proof}

\subsection{The Flow Above Critical Points}

In this section we apply the {\L}ojasiewich-Simon gradient estimate to reprove long-time existence for solutions that approach a critical point from above as stated
in Theorem~\ref{thm:LongTimeExistenceI}. Using the techniques from this section, we will show that even the complete flow converges to a critical point without applying any translations.

%

Theorem~\ref{thm:LongTimeExistenceI} will follow  easily from the following long time
existence result for normal graphs over a critical point of the
energies $E_\lambda = H^\alpha + \lambda L$:

\begin{theorem}[Long time existence and asymptotics for normal
  graphs]\label{thm:LongTimeExistenceForGraphs}

Let $\g_{M}\in C^{\infty}(\mathbb{R}/\mathbb{Z}, \mathbb R^n)$ be a critical
point of $E_\lambda$ and let $k\in \mathbb N$,
$\delta>0$ and $\beta>\alpha$. Then there is an open neighborhood $U'$ of $0$ in $\left(C^{\beta}\right)^\bot$ such that the following holds:

Suppose that  $N\in C([0,T), h^{\beta}(\mathbb R / \mathbb Z,
\mathbb R^n)_{\g_M}^\bot ) \cap  C([0,T), C^\infty(\mathbb R / \mathbb Z,
\mathbb R^n)_{\g_M}^\bot )$ is a maximal solution of the equation
\begin{equation*}
  \partial_t^\bot (\g_M+N_t )= - H^\alpha(\g_M+ N_t) + \lambda \kappa
\end{equation*}
with
\begin{equation*}
  N_0 \in U'
\end{equation*}
and 
\begin{equation*}
 E(\g_t) \geq E(\g_M) 
\end{equation*}
whenever  $\|N(t)\|_{C^{k}}\leq \delta$.

Then $T=\infty$ and $N_t$ converges
smoothly to an $N_\infty\in C^\infty(\mathbb R / \mathbb Z, \mathbb R^n)^\bot_{\g_M}$ 
satisfying \[
E(\g_{M} + N_\infty)=E(\g_{M}).\] Furthermore $\g_M + N_\infty$
is a critical point of the energy.
\end{theorem}

\begin{proof}

Let $\tilde k = \max\{4,k\}$. Then of course we still have \begin{equation*}
 E(\g_t) \geq E(\g_M) 
\end{equation*}
under the stronger condition  $\|N(t)\|_{C^{\tilde k}}\leq \delta$.

It is crucial to use the smoothing properties of our short time existence result Theorem~\ref{thm:ShortTimeExistenceForGraphs}. 
This theorem tells us that there is an $\delta_1>0$ and a time $T>0$ such that the solution to 
\begin{equation*}
  \partial_t^\bot (\g_M+N_t )= H(\g_M+ N_t) + \lambda \kappa_{\g_M+ N_t}
\end{equation*}
with initial data $N_0 \in U=\{N \in C^{\beta}: \|N\|_{C^{\beta}}\leq \delta_1\}$ exists for $t \in [0,T)$ and satisfies
\begin{equation*}
 \|N_t\|_{C^{\tilde k}} \leq \delta
\end{equation*}
for all $t \in [\frac T2, T]$. Furthermore, we know that the flow can be continued at least up to the time $T$ as long as $\|N_t\|_{C^{\beta}} \leq \delta_1.$

Making $\delta_1$ smaller if necessary, we can achieve that the {\L}ojasiewich-Simon gradient estimate 
holds for all $N \in U$ and that both 
\begin{equation} \label{eq:OscProjection}
\|P_{\g_M'+N'} ^\bot - P_{\g_M'}^\bot\| \leq 1/2
\end{equation}
and
\begin{equation} \label{eq:Regular}
|\g_M'+N'| \geq  \frac 1 2 \inf|\g_M'| >0.
\end{equation}

The smoothing properties imply that for every $\varepsilon >0$ we can also get
\begin{equation*}
 \| N_{t} \|_{C^{\beta}} \leq \varepsilon \quad \forall t \in [\frac T2, T)
\end{equation*}
if the initial data belongs to a suitable smaller neighborhood $U' \subset U$ of $0$. Assuming that $\varepsilon< \delta$  we can achieve
$N_{t} \in U$ for all $N_0 \in  U'$ and $t \in [\frac T2,T] $.

%

Now let $N\in C([0,T_{max}), h^{\beta}(\mathbb R / \mathbb Z,
\mathbb R^n)_{\g_M}^\bot ) \cap  C((0,T_{max}), C^\infty(\mathbb R / \mathbb Z,
\mathbb R^n)_{\g_M}^\bot )$ be a maximal solution of the equation
\begin{equation*}
  \partial_t^\bot (\g_M+N_t )= H(\g_M+ N_t)
\end{equation*}
with $N_0 \in U'$ as in the theorem.
%

We know from the considerations above that $T_{max} \geq T$ and 
$$
 \|N_{\frac T2}\|_{C^{\beta}} < \delta_1.
$$
If the solution does not exist for all time, there hence is a $t_0> \frac T 2$ such that
$$
 \|N_t\|_{C^{\beta}} \leq \delta, \quad \forall t \in [\frac T 2, t_0).
$$
but
$$
 \|N_{t_0}\|_{C^{\beta}} = \delta. 
$$
Then we get
$$
 \|N_t\|_{C^5} \leq C, \quad \forall t \in [\frac T2, t_0],
$$
and $N_t$ satisfies a {\L}ojasiewich-Simon gradient estimate and \eqref{eq:OscProjection} and \eqref{eq:Regular} for all $t \in [\frac T2,t_0]$.

%

We can calculate
\begin{equation*} 
\begin{aligned}
\frac{d}{dt}E(\tilde \g_{t}) &
=-\int_{\mathbb R / \mathbb Z} \langle \partial_{t}^{\bot} \tilde \g_t ,
V(\tilde \g_t) \rangle
|\tilde \g_t'| = \\
& =-\int_{\mathbb R / \mathbb Z} |\partial_{t}^{\bot}\tilde
\g_t|^2 |\tilde \g_t'|
 \\
& =-\int_{\mathbb R / \mathbb Z} |V \tilde \g_t|^2 |\tilde \g_t'| 
\end{aligned}
\end{equation*}
and hence \begin{align*}
-\frac{d}{dt}\left(E(\tilde \g_{t})-E(\tilde \g_{M})\right)^{\theta} &
=-\theta\left(E(\tilde
  \g_{t})-E(\g_{M}\right)^{\theta-1}\frac{d}{dt}E(\tilde \g_{t})
\\ 
&\geq\frac{\theta}{\sigma}\left(\int_{\mathbb R / \mathbb
    Z}|\partial_t ^{\bot} \tilde \g_t |^2 |\tilde \g_t' | \right)^{1/2} \\
& \geq c \left(\int_{\mathbb R / \mathbb
    Z}|\partial_t \tilde \g_t|^2 \right)^{1/2}.
\end{align*}
Integrating the above inequality over $(\frac T2 ,t)$ yields
\begin{align*}
\|\tilde{\g}_{t}-\g_{M}\|_{L^{2}} & \leq \|\tilde \g_{\frac T2} - \tilde \g_M\|_{L^2} + \|\tilde
\g_{\frac T2 }-\tilde{\g}_{t}\|_{L^{2}}
 \\ & \leq \|\tilde \g_{\frac T2} - \tilde \g_M\|_{L^2} + \int_0^1 \left(\int_{\mathbb R / \mathbb Z }|\partial_\tau \g|^2 ds \right)^{\frac 1 2}d\tau
 \\ &\leq \|\tilde \g_{\frac T2} - \tilde \g_M\|_{L^2} +
C\left(E(\tilde \g_{\frac T2})-E(\g_{M})\right)^{\theta} \\ 
&\leq \|\tilde \g_{\frac T2} - \tilde \g_M\|_{L^2} +
C\left(E(\tilde \g_{\frac T2})-E(\g_{M})\right)^{\theta}\\
 & \leq C\|\tilde \g_{\frac T2}-\g_{M}\|_{C^{\beta}}^{\theta}.\end{align*}
Using the interpolation inequality
\begin{equation*}
\|f\|_{C^{\beta}} \leq \|f\|_{C^{\tilde k}}^{(1-\sigma)} \|f\|_{L^2}^{\sigma}
\end{equation*}
where $\sigma= \frac{\tilde k -2 - \beta}{\tilde k+\frac 12}$, we get for $t \in [\frac T 2,
t_0]$
\begin{equation} \label{eq:Convergence}
\begin{aligned}
\|\tilde{\g}_{t}-\g_{M}\|_{C^{2 + \beta}(\mathbb{R}/\mathbb{Z},\mathbb{R}^{n})}
& \leq
C\|\tilde{\g}_{t}-\g_{M}\|_{C^{\tilde k}(\mathbb{R}/\mathbb{Z},\mathbb{R}^{n})}^{1-\sigma}\|\tilde
\g_{t}-\g_{M}\|_{L^{2}}^{\sigma}\leq C\|\tilde \g_{\frac T2}-\g_{M}\|_{C^{\beta}}^{\theta\sigma}\\
 & \leq C \varepsilon^{\theta\sigma}
 \end{aligned}
 \end{equation}
 if $\varepsilon >0$ is small enough. 
 
So if $\varepsilon >0$ is small enough we have
$$
 \|N(t)\|_{C^{\beta}} < \frac \delta 4
$$
as long as the flow exists. As this contradicts our choice of $t_0$, we have shown that the flow exists for all time.

From Theorem~\ref{thm:ShortTimeExistenceForGraphs} we get  $\sup_{t \geq \frac T 2}\|\tilde \g_{t}\|_{C^{l}}<\infty$ for
all $l \in \mathbb N$ and hence there is a subsequence $t_i \rightarrow \infty$ such that
$$
 c_{t_i} \rightarrow c_\infty
$$
smoothly. 
 Since \[
 \partial_{t} \tilde \g\in L^{1}([0,\infty),L^{2})\]
we deduce that  $\delta E(\g_{\infty})=0$. Using the {\L}ojasievicz-Simon
gradient inequality again we get \[
(E(\g_{\infty})-E(\g_{M}))^{1-\theta}\leq c
\left(\int_{\mathbb R / \mathbb Z} |H\g_\infty|^2 |\g_\infty'|\right)^{1/2}=0\]
and hence $E(\g_{\infty})=E(\g_{M}).$

To get convergence of the complete flow, we repeat the estimates above with $c_\infty$ in place of $c_M$ to get in view of \eqref{eq:Convergence} that
\begin{equation*}
\begin{aligned}
\|\tilde{\g}_{t}-\g_{\infty}\|_{C^{\beta}(\mathbb{R}/\mathbb{Z},\mathbb{R}^{n})}
& \leq C\|\tilde \g_{t_i}-\g_{\infty}\|_{C^{\beta}}^{\theta\sigma}
 \end{aligned}
 \end{equation*}
for all $t \geq t_i$. So the complete flow converges in $C^{\beta}$ and hence by interpolation in $C^\infty$ to $c_\infty$
\end{proof}

\begin{proof}[Proof of Theorem~\ref{thm:LongTimeExistence}]
Due to Lemma~\ref{lem:TubularNeighborhoodAndNormalGraphs} for all $\g \in C^{\beta}(\mathbb R / \mathbb Z,
\mathbb R ^n ) $ with $\|\g-\g_M \|_{C^{\beta}} \leq \varepsilon$ there
is a diffeomorphism $\phi_\g$ and a vector field $N_{\g} \in C^{\beta}(\mathbb R /
\mathbb Z , \mathbb R^n)$ normal to $\g_M$ such that
\begin{equation}
  \g \circ \phi_\g = \g_M + N_{\g} 
\end{equation}
and
\begin{equation}
 \|N_\g \|_{C^{\beta}} \leq C \|\g - \g_M \|_{C^{\beta}}
\end{equation}
if $\varepsilon>0$ is small enough.

For $\g \in C^{2+\alpha}$ with
\begin{equation*}
  \|\g-\g_M\|_{C^{2+\alpha}} \leq \varepsilon
\end{equation*}
let $(N_t)_{t \in [0,\tilde T)}$ be the maximal solution of 
\begin{equation*}
\begin{cases}
  \partial_t^\bot (\g_M + N_t) = H^\alpha (\g_M + N_t)+ \lambda \kappa_{\g_M + N_t}\\
  N_0 = N_\g.
\end{cases}
\end{equation*}

Then $\tilde T \leq T$ and for all $t \in [0,\tilde T)$ there are diffeomorphisms $\phi_t$ such that $\g_t = (\g_M
+ N_t) (\phi_t)$. Hence $N_t$ satisfies all the assumptions of
Theorem~\ref{thm:LongTimeExistence} if $\varepsilon$ is small enough and thus $\infty= \tilde T$. Form
$\tilde T\leq T$ we deduce $T=\infty$. 
\end{proof}

\subsection{Completion of the Proof of Theorem \ref{thm:LongTimeExistenceResult}}

It is only left to show that we get converges of the flow without applying translations from the smooth
subconvergence of the re-parameterized and translated curves we get from 
Subsections~\ref{sec:ProofOfLongTimeExistence}.
 
 Let $\tilde \g_t$ be the re-parameterizations of $\g_t$ and let
 $t_i \rightarrow \infty$ and $p_i \in \mathbb R^n$ be such that the 
 curves $\tilde \g(t_i)-p_i$ converge smoothly to a curve $\tilde \g_\infty$
 parameterized by arc length. Due to the smooth convergence, the data 
 $\tilde \g_{t_i} - p_i$ satisfies all the assumptions of Theorem~\ref{thm:LongTimeExistenceI} for $i$ large 
 enough. Hence, the statement follows.

\appendix

\section{Analytic Functions on Banach Spaces}

We briefly prove some lemmata about analytic functions on Banach
spaces. A thorough discussion of this subject can be found in \cite[Chapter
2, Section 3]{Hille1957}.

\begin{definition}[Analytic operator]

Let $(X,\|\cdot\|_{X})$, $(Y,\|\cdot\|_{Y})$ be real Banach spaces.
A function $f\in C^{\infty}(A,Y)$, $A\subset X$ open, is called
real analytic if for every $a\in A$ there is a open neighborhood
$U$ of $a$ in $X$ and a constant $C<\infty$ such that\[
\left\Vert D^{m}f(x)\right\Vert \leq C^{m}m!\quad\forall m\in\mathbb{N},\, x\in U.\]

\end{definition}
In this context $\| \cdot \|$ denotes the operator norm.
The next lemmata show how to construct analytic functions:

\begin{lemma} \label{lem:AnalyticityOfCompositions}

Let $g: U\rightarrow\mathbb{R}^{k}$ be a real analytic
function, $U \subset \mathbb R ^n$ be an open subset, and let $V
\subset C^{k,\alpha}(\mathbb R /\mathbb Z, \mathbb R^n)$ be an open subset
such that $\im f \subset U$ for all $f \in V$ . Then
\begin{gather*}
T:V\rightarrow
C^{k,\alpha}(\mathbb{R}/\mathbb{Z},\mathbb{R}^{l})\\ x\rightarrow
g\circ x
\end{gather*} defines a real analytic function.
\end{lemma}

\begin{proof}

Let $f_0 \in V$. Since $\im f_0$ is a compact subset of the open set $U$
there is an $\varepsilon>0$ such that $K_\varepsilon := \bigcup_{y\in
  \im f_0}\overline{B_\varepsilon(y)} \subset
U$.

Since $g$ is real analytic and $K_\varepsilon \subset U$ is compact, there is a constant $C\leq\infty$ such
that\[
\|D^{m}g(y)\|\leq C^{m}m!,\quad\forall m\in\mathbb{N},\,y \in K_\varepsilon.\]
As \[
D^{m}T(y)(h_{1},\dots,h_{m})=D^{m}g(y)(h_{1},\dots,h_{m})\]
(can easily be deduced from the Taylor expansion of g) and since
$C^{\alpha}(\mathbb{R}/\mathbb{Z})$ is a Banach algebra, we get
\[
\|D^{m}T(f)\|\leq(k+1)C^{m+k+1}(m+k+1)!\leq\tilde{C}^{m}m!\]
for all $f \in B_\varepsilon (f_0) = \{y\in C^{\alpha} :
\|y\|_{C^\alpha} \leq \varepsilon
\}$ where
$$\tilde{C}:=\sup_{m\in\mathbb{N}_{0}}\left(\frac{(k+1)C^{m+k+1}(m+k+1!)}{m!}\right)^{1/m}$$.
\end{proof}

\begin{lemma}  \label{lem:AnalyticityOfParametricIntegrals}
Let $(X,\|\cdot\|_{X})$ and $(Y,\|\cdot\|_{Y})$ be Banach spaces and
assume that $T_{t}\in C^{\omega}(X,Y)$ for $t\in I$ is such that
the functions $t\rightarrow T_{t}$ are measurable and for all $a\in X$
there is a neighborhood $U$ of a in $X$ such that \begin{equation}
\int_{I}\left(\sup_{y\in U}\|D^{m}T_{t}(y)\|\right)dt\leq C^{m}m!.\label{eq:EstimateDerivativesOfTt}\end{equation}
Then the mapping $T:X\rightarrow Y$ defined by\[
Tx=\int_{I}T_{t}xdt\]
is real analytic.
\end{lemma}

\begin{proof}

We want to show that \[
D^{m}Tx(h_{1},\dots,h_{m})=\int_{I}D^{m}T_t x(h_{1},\dots,h_{m})dt.\]
from which we get that $T\in C^{\omega}(X,Y)$ using the estimates
\eqref{eq:EstimateDerivativesOfTt}. In fact this follows from well-known
facts about differentiation of parameter dependent integrals. 

\end{proof}

\begin{remark}

In the case that $Y=C^{k,\alpha}(\mathbb{R}/\mathbb{Z}, \mathbb R^n)$ it is well known,
that \[
(Tx)(u)=\int_{I}(T_{t})x(u)dt,\]
i.e., the value of function $Tx$ given by the  Bochner integral at the point $u$ is equal to the Lebesque integral of the functions
$T_t (u)$ evaluated at the point $u$.

\end{remark}

\section{Estimates for a generalization of the Multilinear Hilbert transform}

\begin{lemma} \label{lem:HilberttransformSpecial}

Let $\alpha \in (0,1)$. For  $1-\alpha \geq \beta>0$, $n,m\in\mathbb{N}$, and $t_{i}\in(0,1)$
the singular integral \[
T(\g_{1},\dots\g_{m})(u):= p.v. \int_{-\frac12}^{\frac 12}\frac{1}{w|w|^\alpha}\prod_{i\text{=1}}^{m}\g_{i}(u+t_{i}w)dw\]
defines a bounded multilinear operator from $C^{\alpha+\beta}(\mathbb R /\mathbb Z,\mathbb{R})$
to $C^{\beta}(\mathbb R / \mathbb Z,\mathbb{R})$. 

\end{lemma}

\begin{proof} 
For $u,v \in \mathbb R /\mathbb Z$ and $a \in (0,\frac 1
  2 )$ we get 
\begin{align*}
\big| T(\g_1,\dots,\g_m)(u) &-T(\g_1,\dots,\g_m)(v)\big| \\
& \leq\int_{1/2\geq|w|\geq a}\frac{1}{w|w|^\alpha}\left|\prod_{i\text{=1}}^{m}\g_{i}(u+t_{i}w)-\prod_{i\text{=1}}^{m}\g_{i}(v+t_{i}w)\right|dw\\
& \quad+\int_{a\geq |w|}\frac{1}{w|w|^\alpha}\left|\prod_{i\text{=1}}^{m}\g_{i}(u+t_{i}w)-\prod_{i\text{=1}}^{m}\g_{i}(u)\right|dw\\
& \quad +\int_{a\geq |w|}\frac{1}{w|w|^\alpha }\left|\prod_{i=1}^{m}\g_{i}(v+t_{i}w)-\prod_{i=1}^{m}\g_{i}(v) \right| dw
\end{align*}
 Since $|t_{i}|\leq1$, we obtain
 \[
\left|\prod_{i\text{=1}}^{m}\g_{i}(u+t_{i}w)-\prod_{i\text{=1}}^{m}\g_{i}(u)\right|\leq m\prod_{i\text{=0}}^{m}\|\g_{i}\|_{C^{\alpha+\beta}(\mathbb R /\mathbb Z,\mathbb{R}^{n})}|w|^{\alpha+\beta}\]
 and hence, using the H\"older-continuity of the $\g_{i}$,\begin{align*}
\big|T(&\g_{1},\dots,\g_{m})(u)-T(\g_{1},\dots,\g_{m})(v)\big| \\
 & \leq m \prod_{i=0}^{m}\|\g_{i}\|_{C^{\alpha+\beta}(\mathbb R / \mathbb Z,\mathbb{R}^{n})}\left(\int_{1\geq|w|\geq a}\frac{1}{|w|^{1+\alpha}}\left|u-v\right|^{\alpha +\beta}dw+2\int_{|w|\leq a}\frac{1}{|w|^{1-\beta}}dw\right)\\
 & \leq m \prod_{i=0}^{m}\|\g_{i}\|_{C^{\alpha+\beta}(\mathbb R /\mathbb Z,\mathbb{R}^{n})}\left(\frac 1 \alpha a^{-\alpha}\left|u-v\right|^{\alpha+\beta}dw+\frac 1 {\beta} a^{\beta}\right).\end{align*}
Choosing $a=|u-v|$ for $|u-v| \leq \frac 12$ we get \begin{align*}
\big|T(\g_{1},\dots,\g_{m})(u) &-T(\g_{1},\dots,\g_{m})(v)\big| \\
 & \leq C\prod_{i=0}^{m}\|\g_{i}\|_{C^{\alpha+\beta}(\mathbb R /\mathbb Z,\mathbb{R}^{n})}|u-v|^{\beta}.
 \end{align*}

\end{proof}

\begin{lemma} \label{lem:HilberttransformGeneral}
For arbitrary $\beta>0$, $\alpha \in (0,1)$, $n,m\in\mathbb{N}$, and $t_{i}\in(0,1)$
the singular integral \[
T(\g_{1},\dots\g_{m})(u):=p.v. \int_{-\frac 12}^{\frac 12} \frac{1}{w|w|^\alpha}\prod_{i\text{=1}}^{m}\g_{i}(u+t_{i}w)dw\]
defines a bounded multilinear operator from $h^{\alpha+\beta}(\mathbb R /\mathbb Z,\mathbb{R})$
to $h^{\beta}(\mathbb R / \mathbb Z,\mathbb{R})$. 
\end{lemma}

\begin{proof}

  First, let us note that it is enough to prove the statement for
  $\beta =\tilde \beta + n$, $n \in
  \mathbb N _0$, $1-\alpha \geq \tilde \beta >0$. One then uses real interpolation to prove the full statement. For
  $n=0$ the claim is the content of Lemma~\ref{lem:HilberttransformSpecial}.
  So let us assume that the statement holds for $\beta= \tilde \beta + n$ as above. 
  Let $(\tau_h f)(x):=f(x+h)$. Using the relation $\tau_h(T(\g_1,
  \dots, \g_m))= T(\tau_h(\g_1) \dots \tau_h(\g_m))$ and
  the multilinearity of $T$, the difference quotient can be written as
\begin{align*}
\frac{\tau_h (T(\g_1, \dots, \g_m)-T(\g_1,\dots,
  \g_m))}{h} &= \frac{T(\tau_n(\g_1), \dots, \tau_h(\g_m))-T(\g_1,\dots,
  \g_m))}{h} \\
& \sum_{i=1}^m T\left( \g_1, \dots, (\tau_h
  (\g_i)-\g_i)/h, \dots, \tau_h(\g_m)\right).
\end{align*}
 Since
\begin{gather*} (\tau_h (\g_i)-\g_i)/h
\xrightarrow{h\rightarrow 0} \g' \quad \text{ in } C^\beta \\
(\tau_h(\g_i)\xrightarrow{h \rightarrow 0} \g \quad \text{ in
} C^\beta
\end{gather*} and $T$ a bounded linear operator from $C^{\beta }$ to
$C^{\tilde \beta}$, we get
\begin{equation*} \frac{\tau_h (T(\g_1, \dots,
\g_m)-T(\g_1,\dots, \g_m))}{h} \xrightarrow {h\rightarrow
0} \sum_{i=1}^m T\left( \g_1, \dots, \g_i', \dots,
\g_m\right)
\end{equation*} in $C^{\beta}$. Hence, $T$ is a bounded
multilinear mapping from $C^{\alpha + \beta+1}$ to $C^{
\beta+1}$.
\end{proof}

\begin{remark} \label{rem:MultilineraHilbertTransform}
Let us state a simple extension of
Lemma~\ref{lem:HilberttransformGeneral}.
Given a multilinear form $M: \mathbb R^n \times \dots \times \mathbb
R^n \rightarrow \mathbb R^k$, $1>\alpha>0 $, $\beta>0$, $m\in\mathbb{N}$, and $t_{i}\in(0,1)$
the singular integral \[
T(\g_{1},\dots\g_{m})(u):=p.v. \int_{-\frac 12}^{\frac 12}\frac{1}{w|w|^\alpha}M(\g_{1}(u+t_{1}w),
\ldots, \g_m (u+t_m w)))dw\]
defines a bounded multilinear operator from $h^{\alpha+\beta}(\mathbb R /\mathbb Z,\mathbb{R})$
to $h^{\beta}(\mathbb R / \mathbb Z,\mathbb{R}^{n})$.

This can be
deduced by plugging 
\begin{equation*}
 M(\g_{1}(u+t_{1}w),
\ldots, \g_m (u+t_m w)))= \sum_{j_1, \ldots, j_m =1}^n M(e_{i_1},
\ldots , e_{j_m}) \Pi_{i=1}^m\left< \g_j(u+t_i w),e_{j_i}\right>
\end{equation*}
into the definition of $T$, where $e_1, \ldots e_n$ is the standard basis of $\mathbb R^n$, and
by applying Lemma~\ref{lem:HilberttransformGeneral} to all the
coordinates of the resulting  summands.
\end{remark}

\section{Facts about the Functional \protect{$Q^\alpha$}}

In this section we prove a commutator inequality that we use as a
substitute for the
Leibniz rule for $Q^{\alpha}$.

\begin{lemma}[Leibniz rule for \protect $Q^{\alpha}$] \label{lem:LeibnizRule}
For $f,g \in C^{\beta_1+\alpha}(\mathbb R / \mathbb Z, \mathbb R ^n)$ and
$\beta>0$ we have
\begin{equation*}
  \|Q^{\alpha}(fg)-Q^{\alpha}(f)g\|_{C^{\beta}} \leq C(\alpha, \beta) \|f\|_{C^{\alpha + \beta}}
  \|g\|_{C^{\alpha +1+\beta}}.
\end{equation*}
\end{lemma}

\begin{proof}
We have
\begin{equation} \label{eq:Leibniz}
\begin{aligned} 
  &(Q^{\alpha}(fg)-Q^{\alpha}(f)g)(u) \\ &= \int_{[-1/2,1/2]}\bigg\{ 2\frac
      {f(u+w)g(u+w)-f(u)g(u)-w(f'(u)g(u)-f(u)g'(u))}{w^2} 
      \\ & \quad -
    (f''(u)+2g'(u)f'(u)+f(u)g''(u))\bigg\}\frac{dw}{w^\alpha} \\
    &\quad - \int_{[-1/2,1/2]}\bigg\{ 2 \frac
    {f(u+w)g(u)-f(u)g(u)-wf'(u)g(u)}{w^2} - f''(u)g(u)\bigg\}\frac
    {dw}{w^\alpha} \\
 &= \int_{[-1/2,1/2]} \bigg\{2 \frac {(f(u+w)-f(u))(g(u+w)-g(u))}{w^2}
 - 2 f'(u)g'(u) \\ &  \quad + f(u) \big(2 \frac
 {g(u+w)-g(u)-wg'(u)}{w^2}-g''(u)\big) \bigg\}\frac{dw}{w^\alpha} \\
&= 2\int_{[-1/2,1/2]} \bigg\{ \frac {(f(u+w)-f(u))(g(u+w)-g(u))}{w^2}
 -  f'(u)g'(u) \bigg\} \frac{dw}{|w|^\alpha}\\ &\quad +(fQ^{\alpha}(g))(u).
\end{aligned}
\end{equation}
Taylor expansion yields
\begin{gather*}
  \frac {f(u+w)-f(u)}{w}= f'(u)+w\int_0^1 (1-t)f''(u+tw)dt \\
  \frac {g(u+w)-g(u)}{w}= g'(u)+w\int_0^1 (1-t)g''(u+tw)dt
\end{gather*}
and hence the first term in the last row of Equation (\ref{eq:Leibniz}) can be written
as
\begin{align*}
   2\int_{[-1/2,1/2]} &\bigg\{ \frac {(f(u+w)-f(u))(g(u+w)-g(u))}{w^2}
   -  f'(u)g'(u) \bigg\} \frac {dw}{|w|^\alpha}\\
  &= 2f'(u) \int_{[-1/2,1/2]} \frac {\int_0^1 (1-t)g''(u+tw)dt}{w} \frac {dw}{|w|^{\alpha-2}} \\
 &\quad    + 2g'(u) \int_{[-1/2,1/2]} \frac {\int_0^1
   (1-t)f''(u+tw)dt}{w} \frac {dw}{|w|^{\alpha-2}} \\
 &\quad  + 2 \int_{[-1/2,1/2]} \int_{[0,1]}\int_{[0,1]} (1-t)(1-s)g''(u+tw)f''(u+sw)dtds \frac {dw}{|w|^{\alpha-2}}.
\end{align*}
It is an easy exercise to prove that the last term defines a bounded operator from $C^{2+\beta}$ to $C^\beta$ for all
$\beta >0$. 
Using Lemma~\ref{lem:HilberttransformGeneral} to estimate the first
two terms, we get 
\begin{multline*}
 \|Q^{\alpha}(fg)-Q^{\alpha}(f)g-gQ^{\alpha}(f)\|_{C^\beta} \\ \leq C\big(\|f'\|_{C^\beta}
 \|g''\|_{C^{\alpha-2+\beta}} + \|g'\|_{C^\beta} \|f''\|_{C^{\alpha-2+\beta}} + \|f''\|_{C^\beta}\|g''\|_{C^\beta} \big)
\end{multline*}
and hence 
\begin{equation*}
  \|Q^{\alpha}(fg)-Q^{\alpha}(f)g-fQ^{\alpha}(g)\|_{C^\beta} \leq \|f\|_{C^{\alpha+\beta}} \|g\|_{C^{\alpha+\beta}}.
\end{equation*}
Together with the fact that $Q$ defines a bounded linear operator from $C^ {\alpha+1+\beta}$ to $C^{\beta}$ this proves the estimate.
\end{proof}

\section{Interpolation and Commutator Estimates}

For the convenience of the reader we present mutliplicative Gagliardo-Nirenberg-Sobolev inequalities
and commutator estimates in fractional Sobolev spaces and Besov space that we used in this text. We use the notation 
$$
 B^{s,p}_q 
$$
where $s$ denotes the order of differentiation and $p$ the integrability.
\begin{theorem}[Gagliardo-Nirenberg-Sobolev Inequality]
\label{thm:MultiplicativeSobolevInequalities} \label{eq:GagliardoNirenbergEstimates}
For $p \in [2, \infty)$, $q \in [1,\infty)$ and $s_1,s_2, s_3 \in [0,\infty)$
with $s_1 - 1/2 \leq s_2 - 1/p \leq s_3 - 1/2$  there is a
$C=C(s_1, s_2, s_3,p,q) < \infty$ such that
\begin{equation*}
  \|u\|_{B^{s_2, p}_q} \leq C \|u\|_{W^{s_3,2}} ^\theta
\|u\|_{W^{s_1,2}}^{1-\theta}
\end{equation*}
for all smooth functions $u \in C^\infty(\mathbb R / \mathbb Z)$ where $\theta:= (s_2-s_1 -1/p + 1/2)/(s_3 - s_1)$. Especially, 
\begin{equation*}
  \|u\|_{W^{s_2, p}} \leq C \|u\|_{W^{s_3,2}} ^\theta
\|u\|_{W^{s_1,2}}^{1-\theta}.
\end{equation*}

\end{theorem}

\begin{proof}
Using \cite[Proposition~1.20]{Lunardi2009}, it is enough to show that the real
interpolation space $(W^{s_1}_2, W^{s_2}_2)_{\theta,1}$ is continuously embedded 
in $B^{s_2,p}_q$. From \cite[Section~1.6.7]{Triebel1992}, we
get $(W^{s_1,2}, W^{s_2,2})_{\theta,1} = B^{\tilde s,2}_{1}$ where $\tilde s :=
 \theta s_3 + (1-\theta) s_1 \geq s_2$. Since $p\geq 2$, we get
 $\tilde s \geq s_2$. As furthermore $\tilde s -1/2 = s_2 -1/p$, the
Sobolev embedding tells us that
\begin{equation*}
  B^{2,\tilde s}_{1} \subset B^{s,p}_{1} \subset B^{s,p}_{q}
\end{equation*}
and especially
\begin{equation*}
  B^{2,1}_{\tilde s} \subset B^{p,1}_{s} \subset B^{s,p}_{p} = W^{s,p}
\end{equation*} 
which completes the proof.
\end{proof}

The following commutator estimate is a periodic version of known results on the Euclidean space. We leave the proof to the reader. 

\begin{theorem}[Commutator estimates alla Kato-Ponce] \label{thm:CommutatorEstimate}
Let $1< r < \infty$, $1<p_1, p_2,q_1,q_2 \leq \infty$ satisfy $\frac 1 r = \frac 1 {p_1} + \frac 1 {q_1} = \frac 1 {p_2} + \frac 1 {q_2}$ and $s >0$. Then there is a constant $C < \infty$ such that
\begin{align*}
 \|D^s [fg] - f D^s g\|_{L^r} \leq C \left( \|\nabla f\|_{L^{p_1}} \|D^{s-1} g\|_{L^{q_1}} + \|D^s f\|_{L^{p_2}} \|g\|_{L^{q_2}} \right)
\end{align*}
for all smooth functions $f,g \in C^\infty(\mathbb R / \mathbb Z).$
\end{theorem}

\section{Normal Graphs}
The following lemma is used in the proofs of Lemma~\ref{thm:LojasiewichSimonGradientEstimate} and 
Theorem~\ref{thm:LongTimeExistence}.
\begin{lemma} \label{lem:TubularNeighborhoodAndNormalGraphs}
  Let $\g_0 \in C_{i,r}^\infty (\mathbb R / \mathbb Z, \mathbb R ^n)$ and $W = C^s(\mathbb R / \mathbb Z)$ or $W= H^{s+\frac 1 2}(\mathbb R / \mathbb Z)$ for some $s>1$. Then there is an $\varepsilon>0$ such that
  for all $\g \in W$ with
\begin{equation*}
  \|\g - \g_0\|_{W} \leq \varepsilon,
\end{equation*}
there is a re-parameterization $\phi$ and a function $N \in C^\alpha(\mathbb R/ \mathbb Z, \mathbb R ^n) $ normal to $\g_0$ such that
\begin{equation*}
  \g \circ \phi = \g_0 + N
\end{equation*}
and
\begin{equation*}
  \|N\|_{W}\leq C  \|\g - \g_0\|_{W}.
\end{equation*}
\end{lemma}

\begin{proof}
Note that in the case $W=H^{s+\frac 1 2}$, still $W$ is embedded continuously in $C^s$
Let $U$ be an open neighborhood of $\g_0$ such that there is a nearest point retract
$r_U$, i.e. a $C^\infty$ function $r_U : U \rightarrow \mathbb R / \mathbb Z$ such that
\begin{equation*}
 \g \circ r_U \circ \g = \g
\end{equation*}
and
\begin{equation*}
  |\g(r_U (p))- p| = \inf_{x \in \mathbb R / \mathbb Z} |\g(x)-p|.
\end{equation*}
We  set
\begin{equation*}
  \psi_\g(x) = r_U (\g(x)).
\end{equation*}
Since $\psi_{\g_0} = id$ and the space of diffeomorphisms is open in $C^s$, the function $\psi_{\g}$ is a diffeomorphism for sufficiently small $\varepsilon>0$.
Furthermore, the mapping $\g \rightarrow \psi_\g$ is smooth from $C^s(\mathbb R / \mathbb Z, \mathbb R^n)$ to $C^s(\mathbb R / \mathbb Z, \mathbb R / \mathbb Z)$ in the case that $W=C^s$ and from 
$H^{s+\frac 1 2}(\mathbb R / \mathbb Z, \mathbb R^n)$ to $H^{s + \frac 1 2}(\mathbb R / \mathbb Z, \mathbb R / \mathbb Z)$ in the case that $W=H^{s+\frac 1 2}$.
We set
\begin{equation*}
  \phi_\g = \psi_\g ^{-1}
\end{equation*}
and
\begin{equation*}
  N_\g(x)= \g \circ \psi^{-1} - \g_0.
\end{equation*}
The estimate in the lemma now follows from the fact that $\g \rightarrow N_\g$ is a smooth function for $W$ to $W$ in neighborhood of $\g_0$ with $N_{\g_0}=0$.
\end{proof}

\newcommand{\etalchar}[1]{$^{#1}$}

\end{document}